\documentclass[a4paper,10.5pt,leqno]{amsart}
\usepackage{latexsym}
\usepackage[all]{xy}

\usepackage{amssymb} 
\usepackage{amsmath} 
\usepackage{color}
\usepackage{comment}
\usepackage{mathtools}

\usepackage{graphicx}
\usepackage{stmaryrd}
\usepackage{bigdelim, multirow}
\usepackage[all]{xy}

\usepackage{multicol}

\definecolor{gray}{gray}{0.7}
\definecolor{Gray}{gray}{0.3}

\usepackage{amscd}

\numberwithin{equation}{section}

\theoremstyle{break}
 \newtheorem{theorem}{Theorem}[section]
 \newtheorem{proposition}[theorem]{Proposition}
 \newtheorem{corollary}[theorem]{Corollary}
 \newtheorem{lemma}[theorem]{Lemma}

 \theoremstyle{definition}
 \newtheorem{definition}[theorem]{Definition}
 \newtheorem{remark}[theorem]{Remark}
 \newtheorem{example}[theorem]{Example}

\allowdisplaybreaks[3]

\def\C{\mathbb C}
\def\R{\mathbb R}
\def\Q{\mathbb Q}
\def\Z{\mathbb Z}
\def\A{\mathcal{A}}
\def\CR{\mathcal{R}}
\def\Quadraticmap{q}

\DeclareMathOperator{\diag}{diag}
\DeclareMathOperator{\M}{M}
\DeclareMathOperator{\GL}{GL}
\DeclareMathOperator{\Der}{Der}

\DeclareMathOperator{\Hess}{Hess}
\DeclareMathOperator{\Sym}{Sym}

\DeclareMathOperator{\height}{ht}

\def\e{e}

\begin{document}
  
\title[Uniform bases for ideal arrangements]{Uniform bases for ideal arrangements}
\author [M. Enokizono]{Makoto Enokizono}
\address{Graduate School of Mathematical Sciences, University of Tokyo, 3-8-1 Komaba, Meguro-ku, Tokyo 153-8914, Japan}
\email{enokizono@g.ecc.u-tokyo.ac.jp}

\author [T. Horiguchi]{Tatsuya Horiguchi}
\address{National Institute of Technology, Akashi College, 679-3, Nishioka, Uozumi-cho, Akashi, Hyogo 674-8501, Japan}
\email{tatsuya.horiguchi0103@gmail.com}

\author [T. Nagaoka]{Takahiro Nagaoka}
\address{Research Institute for Mathematical Sciences, Kyoto University, Kyoto 606-8502, Japan}
\email{takahiro.nagaoka3617@gmail.com}

\author [A. Tsuchiya]{Akiyoshi Tsuchiya}
\address{Department of Information Science, Faculty of Science, Toho University, Miyama, Funabashi-shi, Chiba, 274-8510, Japan}
\email{akiyoshi@is.sci.toho-u.ac.jp}



\begin{abstract}
In this paper we introduce and study uniform bases for the ideal arrangements in all Lie types.
Explicit uniform bases are given by Abe--Horiguchi--Masuda--Murai--Sato for types $A,B,C,G$ and we provide them for other types.
Combining the explicit uniform bases with the work of Abe--Horiguchi--Masuda--Murai--Sato, we also obtain explicit presentations of the cohomology rings of regular nilpotent Hessenberg varieties in all Lie types.
\end{abstract}

\maketitle

\setcounter{tocdepth}{1}

\tableofcontents

\section{Introduction}
\label{sec:introduction}

In this paper we study bases of the logarithmic derivation modules of the ideal arrangements. 
Ideal arrangements are subarrangements of the Weyl arrangement which are free arrangements from the work of Abe, Barakat, Cuntz, Hoge, and Terao (\cite{ABCHT}). 
Here, a (central) hyperplane arrangement $\A$ is free if its logarithmic derivation module $D(\A)$ (geometrically a polynomial vector fields tangent to $\A$) is a free module. 
To prove freeness for the ideal arrangements, they first provided the multiple addition theorem (MAT) and applied MAT to the ideal arrangements. 
In the proof of MAT they gave a method to construct a basis of the logarithmic derivation module $D(\A)$ from a basis of $D(\A')$ for suitable arrangements $\A \supset \A'$. 

For a construction of explicit bases of the logarithmic derivation modules of the ideal arrangements for each Lie type, Barakat, Cuntz, and Hoge provided ones for types $E$ and $F$ by computer when the work of \cite{ABCHT} was in progress. 
Also, Terao and Abe worked for types $A$ and $B$, respectively. 
In \cite{AHMMS} explicit and uniform bases were constructed for types $A$, $B$, $C$, $G$. 
Motivated by this, we introduce the notion of uniform bases for the ideal arrangements.
We then prove the existence of uniform bases by using the method to construct bases in MAT of \cite{ABCHT}. 

We now describe our uniform bases for the ideal arrangements.
Let $\mathfrak{t}$ be a real Euclidean space and $\Phi$ an irreducible root system of rank $n$ on $\mathfrak{t}^*$, the dual space of $\mathfrak{t}$. 
The set of positive roots is denoted by $\Phi^+$. 
To $\alpha \in \Phi^+$ we assign the hyperplane $H_\alpha:=\ker \alpha$ in $\mathfrak{t}$.
The set of hyperplanes $H_\alpha$ $(\alpha \in \Phi^+)$ is called the \textbf{Weyl arrangement}. 
Let $I \subset \Phi^+$ be a lower ideal and the set of hyperplanes $H_\alpha$ $(\alpha \in I)$ is called the \textbf{ideal arrangement}, denoted by $\A_I$. 
Let $\CR=\mbox{Sym}(\mathfrak{t}^*)$ be the symmetric algebra of the dual space $\mathfrak{t}^*$ and we consider the set of $\R$-derivations of $\CR$, denoted by $\Der \CR$. 
The \textbf{logarithmic derivation module} of an ideal arrangement $\A_I$ is defined as 
$$
D(\A_I)=\{ \theta\in \Der\CR \mid \theta(\alpha)\in \CR\alpha \ ({}^\forall \alpha\in I)\}.
$$
By the work of \cite{ABCHT} the logarithmic derivation module $D(\A_I)$ is a free $\CR$-module, namely $\A_I$ is a free arrangement.
To explain our uniform bases for the ideal arrangement, we need a ``good'' decomposition of positive roots $\Phi^+=\coprod_{i=1}^n \ \Phi^+_i$ (see Section~\ref{section:uniformbases} for the details).
We fix such a decomposition and define the \textbf{Hessenberg function $h_I: \{1,\ldots,n\} \to \Z_{\geq 0}$ associated with a lower ideal $I$}. 
Note that the Hessenberg function $h_I$ suitably converts a lower ideal $I$ into numerical values. 
A set of derivations $\{\psi_{i,j} \in \Der \CR \mid 1 \leq i \leq n, i \leq j \leq i+|\Phi_i^+| \}$ forms \textbf{uniform bases for the ideal arrangements} if derivations $\{\psi_{i,h_I(i)} \mid 1 \leq i \leq n \}$ form an $\CR$-basis of $D(\A_I)$ for any lower ideal $I$.  
Our main theorem states that there exist uniform bases for the ideal arrangements (Theorem~\ref{theorem:main1}). 
Furthermore, our uniform bases are inductively constructed by invertible matrices $P_m$ for $0 \leq m \leq \height(\Phi^+)$. 
More precisely, the initial data $\psi_{1,1},\ldots,\psi_{n,n}$ are the dual basis of the simple roots $\alpha_1, \ldots, \alpha_n$ up to a non-zero scalar multiplication. (Note that $P_0$ is the diagonal matrix whose diagonal entries are the non-zero scalars.)
For the next step $\psi_{1,2},\ldots,\psi_{n,n+1}$, each $\psi_{i,i+1}$ is defined by a linear combination of $\alpha_1\psi_{1,1},\ldots,\alpha_n\psi_{n,n}$ as follows: 
\begin{align*}
[\psi_{i,i+1}]_{1 \leq i \leq n}&=P_1[\alpha_{j}\psi_{j,j}]_{1 \leq j \leq n} 
\end{align*}
for some invertible matrix $P_1$. 
Proceeding inductively, $\psi_{i,i+m}$ is defined by a linear combination of the products $\alpha_{j,j+m}\psi_{j,j+m-1}$ where $P_m$ denotes the array of the coefficients in the linear combinations.
Here, the derivations $\psi_{j,j+m-1}$ are defined in the previous step, and the notation $\alpha_{j,j+m}$ means a positive root in $\Phi^+_j$ with height $m$. 
Hence, our uniform bases are determined by the invertible matrices $P_0,P_1,\ldots,P_{\height(\Phi^+)}$ in order.
We also prove that when we determined $P_0,P_1,\ldots,P_{m-1}$, the invertible matrix $P_m$ is uniquely determined up to an equivalence defined by the special elementary row operations (Theorem~\ref{theorem:main1uniqueness}).
By this construction of the uniform bases, it is natural to compute the invertible matrices $P_m \ (0 \leq m \leq \height(\Phi^+))$ for each Lie type.
In \cite{AHMMS}, bases of the logarithmic derivation modules $D(\A_I)$ are explicitly given for types $A,B,C,G$. 
Also, it is straightforward to see the associated invertible matrices for types $A,B,C,G$.
We construct explicit uniform bases and the invertible matrices for type $D$ in Section~\ref{section:typeD}.
In order to construct uniform bases for types $E$ and $F$, one computes invertible matrices for types $E$ and $F$ by using Maple in Appendix.
To summarize, we obtain uniform bases for all Lie types by our work together with \cite{AHMMS}.

It is geometrically important to describe uniform bases for the ideal arrangements.
In fact, the logarithmic derivation module of an ideal arrangement determines the cohomology ring of the regular nilpotent Hessenberg variety from the work of \cite{AHMMS}. 
Hessenberg varieties are subvarieties of the full flag variety which was introduced by
De Mari, Procesi, and Shayman around 1990 (\cite{dMPS, dMS}).
This subject is relatively new, and it has been found that geometry, combinatorics, and representation theory interact nicely on Hessenberg varieties (e.g. see the survey article \cite{AH}).
The family of regular nilpotent Hessenberg varieties can be regarded as a (discrete) family of subvarieties of the flag variety connecting the Peterson variety and the flag variety itself, where the Peterson variety is related with the quantum cohomology of the flag variety (\cite{Ko, R}).
An explicit presentation of the cohomology ring of the Peterson variety is given by \cite{FHM} in type $A$, and soon after is given in \cite{HHM} for all Lie types. 
Then, \cite{AHHM} gave explicit presentations of the cohomology rings of regular nilpotent Hessenberg varieties in type $A$. 
Note that \cite{AHHM, FHM, HHM} used localization techniques in equivariant cohomology. 
Soon later, \cite{AHMMS} established a connection between the cohomology rings of regular nilpotent Hessenberg varieties and the logarithmic derivation modules of the ideal arrangements. 
In particular, they gave explicit presentations of the cohomology rings of regular nilpotent Hessenberg varieties for types $A,B,C,G$ by using the connection.
One can obtain them for other types from uniform bases which we constructed.
To summarize, we obtain explicit presentations of the cohomology rings of regular nilpotent Hessenberg varieties for all Lie types (Corollary~\ref{corollary:cohomologyHess}).
Here, we note that their presentations generalize the result of \cite{HHM}.

The paper is organized as follows. 
After briefly reviewing some background and terminology on ideal arrangements in Section~\ref{section:ideal arrangements}, we introduce the notion of uniform bases and state a key proposition for uniform bases in Section~\ref{section:uniformbases}. 
The proofs of the main theorems (Theorems~\ref{theorem:main1} and \ref{theorem:main1uniqueness}) for the existence and ``uniqueness'' of uniform bases are achieved in Section~\ref{section:maintheorem}. 
We construct explicit uniform bases for type $D$ in Section~\ref{section:typeD} and discuss them for other types in Appendix.
In Section~\ref{section:Hessenberg varieties}, we explain the connection between ideal arrangements and Hessenberg varieties given in \cite{AHMMS} and give explicit presentations of the cohomology rings of regular nilpotent Hessenberg varieties in all Lie types (Corollary~\ref{corollary:cohomologyHess}).

\bigskip
\noindent \textbf{Acknowledgements.}  
We are grateful to Takuro Abe for his historical and valuable comments on bases of the logarithmic derivation modules of the ideal arrangements.
The first author was partially supported by JSPS Grant-in-Aid for Research Activity Start-up: 19K23407. 
The second author was partially supported by JSPS Grant-in-Aid for JSPS Research Fellow: 17J04330. 
He was also partially supported by JSPS Grant-in-Aid for Young Scientists: 19K14508. 
The third author was partially supported by JSPS Grant-in-Aid for JSPS Research Fellow: 19J11207.
The fourth author was partially supported by JSPS Grant-in-Aid for JSPS Research Fellow: 19J00312. 
He was also partially supported by JSPS Grant-in-Aid for Young Scientists: 19K14505.

\section{Ideal arrangements}
\label{section:ideal arrangements}

In this section we first refer some terminologies of hyperplane arrangements (for a general reference, see \cite{OT}). 
Then we explain ideal-free theorem proved by \cite{ABCHT}.
More specifically, we explain the method to construct bases of the logarithmic derivation modules of the ideal arrangements.

Let $V$ be a finite dimensional real vector space. 
A \textbf{hyperplane arrangement} $\A$ in $V$ is a finite set of linear hyperplanes in $V$.
Let $\CR=\mbox{Sym}(V^*)$ be the symmetric algebra of $V^*$, where $V^*$ is the dual space of $V$. 
A map $\theta: \CR \to \CR$ is an \textbf{$\R$-derivation} if it satisfies 
\begin{enumerate}  
\item $\theta$ is $\R$-linear, 
\item $\theta(f \cdot g)=\theta(f) \cdot g+f \cdot \theta(g)$ for all $f,g \in \CR$.
\end{enumerate}
We denote the set of $\R$-derivations by $\Der \CR$.
Note that $\Der \CR$ is an $\CR$-module.
We can naturally regard an element of $V$ as an $\R$-derivation, so we have the following  identification:
\[
\Der \CR=\CR \otimes V. 
\]
If we take a basis $x_1,\ldots,x_n$ of $V^*$, then the $\CR$-module $\Der \CR$ can be expressed as $\bigoplus_{i=1}^n \CR \, \frac{\partial}{\partial{x_i}}$ where $\frac{\partial}{\partial{x_i}}$ denotes the partial derivative with respect to $x_i$. 

A non-zero element $\theta \in \Der \CR$ is \textbf{homogeneous of (polynomial) degree} $d$ if $\theta=\sum_{k=1}^{\ell}f_k\otimes v_k$ $(f_k\in \CR,\ v_k\in V)$ and all non-zero 
$f_k$'s are homogeneous of degree $d$. In this case, we denote $\deg(\theta)=d$ as usual.

For each hyperplane $H \in \A$, we denote the defining linear form of $H$ by $\alpha_H \in V^*$. The \textbf{logarithmic derivation module} 
$D(\A)$ of a hyperplane arrangement $\A$ is defined as 
\[
D(\A):=\{\theta \in \Der \CR \mid 
\theta(\alpha_H) \in \CR \alpha_H \ (^\forall H \in \A)\}.
\]
Note that $D(\A)$ is a graded $\CR$-submodule, but not a free $\CR$-module in general.
We say that a hyperplane arrangement $\A$ is \textbf{free with the exponents} $\exp(\A)=
(d_1,\ldots,d_n)$ if $D(\A)$ is a free $\CR$-module with homogeneous 
basis $\theta_1,\ldots,\theta_n$ of degree $d_1,\ldots,d_n$.

Let $\alpha_1,\ldots,\alpha_n$ be a basis of $V^*$. 
Given derivations $\theta_1,\ldots,\theta_n$, we define a matrix $\textrm{M}(\theta_1,\ldots,\theta_n)$ by
$$
\textrm{M}(\theta_1,\ldots,\theta_n)=(\theta_i(\alpha_j))_{1 \leq i,j \leq n}.
$$
It is convenient to write $f \ \dot{=} \ g$ for $f,g \in \CR$ if $f=cg$ for some $c \in \R\setminus \{0\}$.
The following criterion for bases of the logarithmic derivation modules is known.

\begin{theorem}[Saito's criterion, \cite{S2}, see also \cite{OT}] \label{theorem:Saito's criterion}
Let $\A$ be a hyperplane arrangement in an $n$-dimensional real vector space $V$. 
Let $\theta_1,\ldots,\theta_n \in D(\A)$ be homogeneous derivations. 
Then the following arguments are equivalent$:$
\begin{enumerate}
\item $\theta_1,\ldots,\theta_n$ form an $\CR$-basis for $D(\A)$$;$ 
\item $\theta_1,\ldots,\theta_n$ are linearly independent over $\CR$ and $\sum_{i=1}^n \deg \theta_i=|\A|$$;$ 
\item $\det {\rm M}(\theta_1,\ldots,\theta_n) \ \dot{=} \ \big(\prod_{H \in \A} \alpha_H \big)$.
\end{enumerate}
\end{theorem}

We now explain ideal arrangements which are the main objects of the paper. 
Let $\mathfrak{t}$ be a (real) Euclidean space. 
Let $\Phi \subset \mathfrak{t}^*$ be an irreducible root system of rank $n$.
We denote the set of positive roots by $\Phi^+$.
We fix simple roots $\alpha_1,\dots,\alpha_n$ and define a partial order $\preceq$ on $\Phi^+$; $\alpha \preceq \beta$  if and only if $\beta-\alpha \in \sum_{i=1}^n \Z_{\ge 0} \alpha_i$.
A \textbf{lower ideal} $I \subset \Phi^+$ is a collection of positive roots such that if $\alpha \in \Phi^+$ and $\beta \in I$ with $\alpha \preceq \beta$, then $\alpha \in I$. 
The \textbf{ideal arrangement} $\A_I$ associated with a lower ideal $I$ is defined as  
\[\A_I:=\{\ker\alpha \mid \alpha \in I\}.\]
If we take $I=\Phi^+$, then $\A_{\Phi^+}$ is called the \textbf{Weyl arrangement}.
Recall that $\CR=\mbox{Sym}(\mathfrak{t}^*)$ and the logarithmic derivation module of the ideal arrangement $\A_I$ is 
\begin{equation*} 
D(\A_I)=\{ \theta\in \Der\CR=\CR\otimes \mathfrak{t} \mid \theta(\alpha)\in \CR\alpha \ ({}^\forall \alpha\in I)\}.
\end{equation*}

The height of a root $\alpha=\sum_{i=1}^n k_i \alpha_i$ is defined by $\height(\alpha)=\sum_{i=1}^n k_i$.
The \textbf{height distribution} in $I$ is a sequence $(i_1,i_2,\dots,i_m)$, where $i_j$ is the number of positive roots of height $j$ in $I$, and $m$ is the maximum of the height of positive roots in $I$. 
Also, the \textbf{dual partition} of the height distribution $(i_1, i_2,\dots,i_m)$ in $I$ is the sequence of $n$ elements given by $((0)^{i_0-i_1}, (1)^{i_1-i_2}, \dots, (m-1)^{i_{m-1}-i_m}, (m)^{i_m})$, where $i_0 = n$ and $(i)^j$ denotes the $j$-copies of $i$.
We denote the dual partition of the height distribution in $I$ by $\mathcal{DP}(I)$.

\begin{theorem}[Ideal-free theorem, {\cite[Theorem 1.1]{ABCHT}}]\label{theorem:ABCHT}
Any ideal arrangement $\A_I$ is free with the exponents $\mathcal{DP}(I)$.
\end{theorem}

To prove Theorem~\ref{theorem:ABCHT}, Abe, Barakat, Cuntz, Hoge, and Terao provided the multiple addition theorem (MAT) (\cite[Theorem~3.1]{ABCHT}).
In the proof of MAT they gave a method to construct an $\CR$-basis of $D(\A)$ from that of $D(\A')$ for suitable arrangements $\A \supset \A'$. 
For the rest of this section, we briefly explain the method to construct an $\CR$-basis under the circumstances of ideal arrangements. 

For a lower ideal $I$ we define the height of $I$ by $\height(I)=\mbox{max}\{\height(\alpha) \mid \alpha \in I \}$.
Suppose that $\height(I)=m+1$ with $m \geq 0$.
Let $I'$ be a lower ideal defined by 
$$
I':=\{\alpha \in I \mid \height(\alpha) \leq m \}
$$
and we set $I \setminus I' = \{\beta_1,\ldots,\beta_q \}$.
For each $j=1,\ldots,q$, we define a hyperplane arrangement $\A''_j$ in $H_j$ by $\A''_j=\{H \cap H_j \mid H \in \A_{I'} \}$ where $H_j$ is a hyperplane defined by the linear function $\beta_j$. 
For each $j=1,\ldots,q$, we fix a map 
$$
\nu_j: \A''_j \to \A_{I'}
$$
such that $\nu_j(X) \cap H_j =X$, and define a homogeneous polynomial 
$$
b_{\nu_j}:= \frac{\prod_{H \in \A_{I'}} \alpha_H}{\prod_{X \in \A''_j} \alpha_{\nu_j(X)}}= \frac{\prod_{\alpha \in I'} \alpha}{\prod_{X \in \A''_j} \alpha_{\nu_j(X)}}.
$$

\begin{proposition}[{\cite{T}, see also \cite[p.114, Proposition~4.41]{OT}}] \label{proposition:idealbnu}
For any $\theta \in D(\A_{I'})$, 
$$
\theta(\beta_j) \in \CR(\beta_j, b_{\nu_j}).
$$
\end{proposition}

\begin{proposition}[{\cite[Proposition~4.2]{ABCHT}}] \label{proposition:ABCHTdegree}
One has $\deg(b_{\nu_j})=m$.
\end{proposition}

Let $\theta_1, \ldots, \theta_n$ be an $\CR$-basis of $D(\A_{I'})$ with $d_i:=\deg(\theta_i)$ such that $d_1 \leq \ldots \leq d_{n-p} < d_{n-p+1}=\cdots=d_n=m$ for some $p$. 
Then, the degree of $\theta_1,\ldots,\theta_{n-p}$ is strictly less than $m$, so we have $\theta_1,\ldots,\theta_{n-p} \in D(\A_I)$ from Propositions~\ref{proposition:idealbnu} and \ref{proposition:ABCHTdegree}.
We put $\varphi_i:=\theta_{n-i+1}$ for $1 \leq i \leq p$.
By Proposition~\ref{proposition:idealbnu}, for $1 \leq i \leq p$ and $1 \leq j \leq q$ we can write 
$$
\varphi_i(\beta_j) \equiv c_{ij}^{(\nu_j)} b_{\nu_j} \ \ \ \ \ \mbox{mod} \ \beta_j
$$
for some rational number $c_{ij}^{(\nu_j)}$ ($1 \leq i \leq p$, $1 \leq j \leq q$).
The following is a key of the proof of MAT (\cite[Theorem~3.1]{ABCHT}). 

\begin{proposition}[\cite{ABCHT}] \label{proposition:ABCHTrank}
The $(p \times q)$-matrix $C:=(c_{ij}^{(\nu_j)})_{\substack{1 \leq i \leq p \\ 1 \leq j \leq q}}$ has  rank $q$. 
\end{proposition}

Hence by Proposition~\ref{proposition:ABCHTrank}, there exists $P=(p_{ik})_{1 \leq i,k \leq p} \in \GL(p,\Q)$ such that 
$$
PC=\begin{bmatrix}
E_q \\
O
\end{bmatrix}.
$$
We put 
$$
\psi_i:=\sum_{k=1}^p p_{ik} \varphi_k
$$
for $1 \leq i \leq p$.
One can see that $\beta_1\psi_1,\ldots,\beta_q\psi_q,\psi_{q+1},\ldots,\psi_{p} \in D(\A_I)$.
It is clear that 
\begin{align*}
\theta_{1},\ldots,\theta_{n-p},\beta_1\psi_1,\ldots,\beta_q\psi_q,\psi_{q+1},\ldots,\psi_{p}
\end{align*}
are linearly independent over $\CR$ and the sum of their degree is equal to $|\A_I|$, so they form an $\CR$-basis of $D(\A_I)$ from Theorem~\ref{theorem:Saito's criterion}.

\begin{theorem}[{\cite{ABCHT}}] \label{theorem:ABCHTbasis}
The derivations $\theta_{1},\ldots,\theta_{n-p}, \beta_1\psi_1,\ldots,\beta_q\psi_q, \psi_{q+1},\ldots,\psi_{p}$ form an $\CR$-basis of $D(\A_I)$.
\end{theorem}

Hence by Theorem~\ref{theorem:ABCHTbasis} we can construct an $\CR$-basis of $D(\A_I)$ for any lower ideal $I$ from that of $D(\A_{I'})$ for the smaller lower ideal $I' \subset I$.
Our motivation is to construct these bases uniformly.
In next section we introduce the notion of \textit{uniform bases}.
Then, we construct uniform bases inductively by using the method of Abe, Barakat, Cuntz, Hoge, and Terao explained above.

\section{Uniform bases}
\label{section:uniformbases}

In this section we first introduce the notion of Hessenberg functions $h_I$ associated with lower ideals $I$ for all Lie types. 
Hessenberg functions convert the lower ideals into numerical values. 
Then we define uniform bases and state a key proposition for uniform bases. 
In what follows, we frequently use the symbol
$$
[n]:=\{1,2,\ldots,n\}.
$$

Let $\e_1,\e_2,\ldots,\e_n$ be the exponents of the Weyl group $W$. (For the list of exponents $e_1,\ldots,e_n$, see for example \cite[p.59 Table~1 and p.81 Theorem~3.19]{Hum1990}.)
We define a decomposition of all positive roots $\Phi^+$ as follows. 
Let $\Phi^+_i$ be a set of positive roots $\alpha_{i,i+1}, \alpha_{i,i+2}, \ldots, \alpha_{i,i+\e_i}$ for $1 \leq i \leq n$ such that 
\begin{align}
&\alpha_{i,i+1}=\alpha_i \ \ \ {\rm the \ simple \ root} \label{eq:decomposition1} \\
&\alpha_{i,j} \lessdot \alpha_{i,j+1} \ \ {\rm for \ any} \ j \ {\rm with} \ i<j<i+\e_i. \label{eq:decomposition2}
\end{align}
We will see this decomposition for each Lie type in the later sections.
Here, we denote the covering relation by the symbol $\lessdot$, namely, $\alpha_{i,j} \prec \alpha_{i,j+1}$ and there is no element $\beta \in \Phi^+$ such that $\alpha_{i,j} \prec \beta \prec \alpha_{i,j+1}$.
The sets $\Phi^+_i \ (1 \leq i \leq n)$ give disjoint $n$ maximal chains.
In particular, we have a decomposition of the positive roots $\Phi^+=\coprod_{i=1}^n \ \Phi^+_i$.
Note that such a decomposition is not unique. 
We fix such a decomposition $\Phi^+=\coprod_{i=1}^n \ \Phi^+_i$ and define the \textbf{Hessenberg function $h_I: [n] \to \Z_{\geq 0}$ associated with a lower ideal $I$} by the following formula:
\begin{equation} \label{eq:Hessft}
h_I(i):=\begin{cases}
\mbox{max}\{j \mid \alpha_{i,j} \in I \cap \Phi^+_i \} \ \ \ &{\rm if} \ I \cap \Phi^+_i \neq \emptyset, \\
i \ \ \ &{\rm if} \ I \cap \Phi^+_i = \emptyset
\end{cases}
\end{equation}
for $1 \leq i \leq n$.

\begin{example} \label{example:An-1Hessenbergfunction}
In type $A_{n-1}$ we set the exponents $e_i=n-i$ for $i=1,\ldots,n-1$ and positive roots $\Phi_{A_{n-1}}^+=\{x_i-x_j \mid 1 \leq i < j \leq n \}$. 
Let $\alpha_{i,j}=x_i-x_j$ for $1 \leq i < j \leq n$. 
Then, $\Phi_i^+=\{\alpha_{i,i+1},\ldots,\alpha_{i,n} \}$ satisfies the conditions \eqref{eq:decomposition1} and \eqref{eq:decomposition2}.
In this setting, for example, the Hessenberg function\footnote{A Hessenberg function $h$ for type $A_{n-1}$ is usually defined on the set $[n]$. 
More specifically, $h:[n] \to [n]$ is a Hessenberg function for type $A_{n-1}$ if $h$ is weakly increasing and $h(i) \geq i$ for $i \in [n]$. We have $h(n)=n$ by the definition, so we may omit the $n$-th value of $h$.} $h_I:[n-1] \to \Z_{\geq 0}$ associated with the lower ideal $I=\{\alpha_1,\ldots,\alpha_{n-1}\}$ is given by $h_I(i)=i+1$ for $i \in [n-1]$.
\end{example}

\begin{definition}
We fix a decomposition $\Phi^+=\coprod_{i=1}^n \ \Phi^+_i$ satisfying \eqref{eq:decomposition1} and \eqref{eq:decomposition2}.
A set of derivations $\{\psi_{i,j} \in \Der \CR \mid 1 \leq i \leq n, i \leq j \leq i+e_i \}$ forms \textbf{uniform bases for the ideal arrangements} (or simply \textbf{uniform bases}) if derivations $\{\psi_{i,h_I(i)} \mid 1 \leq i \leq n \}$ form an $\CR$-basis of $D(\A_I)$ for any lower ideal $I$.  
\end{definition}

Noting that $\height(\alpha_{i,j})=j-i$, a set $\{h_I(i)-i \mid 1 \leq i \leq n \}$ is the dual partition of the height distribution in $I$.
From this together with Theorem~\ref{theorem:ABCHT} we have
\begin{equation} \label{eq:exponentsHessenbergfunction}
\exp(\A_I)=\{h_I(i)-i \mid 1 \leq i \leq n \}.
\end{equation}
In particular, we have
$$
\height(I)=\mbox{max} \{h_I(i)-i \mid 1 \leq i \leq n \}.
$$
For each lower ideal $I$ we define a subset $\Lambda_I$ of $[n]$ as follows:
$$
\Lambda_I:=\{i \in [n] \mid h_I(i)-i = \height(I) \}.
$$
Let us denote the cardinality of $\Lambda_I$ by $\lambda_I=|\Lambda_I|$.

\begin{example}
We consider the setting of Example~\ref{example:An-1Hessenbergfunction}. 
Let $n=5$ and we take 
$$
I=\{x_1-x_2, x_1-x_3, x_2-x_3, x_3-x_4, x_3-x_5, x_4-x_5 \}.
$$
Then, the associated Hessenberg function $h_I$ is given by $h_I(1)=3, h_I(2)=3, h_I(3)=5, h_I(4)=5$ and we have $\Lambda_I=\{1,3 \}$.
\end{example}

For each integer $m$ with $0 \leq m \leq \height(\Phi^+)$, we define the lower ideal $I_m$ by
\begin{equation} \label{eq:Im}
I_m:=\{\alpha \in \Phi^+ \mid \height(\alpha) \leq m \}.
\end{equation}
We write the Hessenberg function associated with the lower ideal $I_m$ by $h_m$.
For simplicity, we denote $\Lambda_{I_m}$ and $\lambda_{I_m}$ by $\Lambda_m$ and $\lambda_m$, respectively. Namely,  
\begin{equation} \label{eq:Lambdam}
\Lambda_m=\{i \in [n] \mid h_m(i)-i = m \}.
\end{equation}
We note that 
$$
\{(i,j) \mid 1 \leq i \leq n \ {\rm and} \ i \leq j \leq i+\e_i \}=\{(i,i+m) \mid 0 \leq m \leq \height(\Phi^+) \ {\rm and} \ i \in \Lambda_m \}.
$$
Let $R$ be a commutative ring.
For two subsets $S, T \subset [n]$ we denote by $\M(S,T;R)$ the set of matrices $[a_{s,t}]_{s \in S \atop t \in T}$ with entries $a_{s,t} \in R$.
We also denote by $\GL(S;R)$ the set of invertible matrices $[a_{s,t}]_{s,t \in S}$ with entries $a_{s,t} \in R$. That is, 
\begin{align*}
\M(S,T;R)=&\{(a_{s,t})_{s \in S \atop t \in T} \ \mbox{matrices} \mid a_{s,t} \in R \}, \\
\GL(S;R)=&\{(a_{s,t})_{s,t \in S} \ \mbox{invertible matrices} \mid a_{s,t} \in R \}. 
\end{align*}
The following is a key proposition.

\begin{proposition} \label{proposition:key}
Assume that derivations $\{\psi_{i,i+m} \in \Der(\CR) \mid 0 \leq m \leq \height(\Phi^+) \ {\rm and} \  i \in \Lambda_m \}$ satisfy the following four conditions$:$
\begin{enumerate}
\item For any $i\in[n]$, $\psi_{i,i}=p_i \, \alpha_i^*$ for some non-zero rational number $p_i$ where $\alpha_1^*,\ldots,\alpha_n^*$ is the dual basis of the simple roots $\alpha_1,\ldots,\alpha_n$$;$
\item For any $m \geq 1$ there exists $P_m \in \GL(\Lambda_m;\Q)$ such that 
$$
[\psi_{i,i+m}]_{i \in \Lambda_m}=P_m[\alpha_{i,i+m}\psi_{i,i+m-1}]_{i \in \Lambda_m};
$$
\item For any $m \geq 0$, $\{\psi_{i,h_m(i)} \mid i \in [n] \} \subset D(\A_{I_m})$$;$
\item For any $m \geq 0$ and any $(i,j) \in \Lambda_{m} \times \Lambda_{m+1}$ with $i \neq j$, $$
\psi_{i,i+m}(\alpha_{j,j+m+1}) \in \CR\alpha_{j,j+m+1}.
$$ 
\end{enumerate}
Then, derivations $\{\psi_{i,i+m} \in \Der(\CR) \mid 0 \leq m \leq \height(\Phi^+) \ {\rm and} \  i \in \Lambda_m \}$ form uniform bases.
\end{proposition}

\begin{proof}
One can see that $\deg(\psi_{i,j})=j-i$ from the conditions $(1)$ and $(2)$.
By \eqref{eq:exponentsHessenbergfunction} we have that $\sum_{i=1}^{n} \deg(\psi_{i,h_I(i)})=|\A_I|$.
From this together with Theorem~\ref{theorem:Saito's criterion} it is enough to prove the following two claims:\\
\textit{Claim~1} For any lower ideal $I$, $\psi_{i,h_I(i)} \in D(\A_I)$ for each $i \in [n]$.\\
\textit{Claim~2} For any lower ideal $I$, the derivations $\{\psi_{i,h_I(i)} \mid i \in [n] \}$ are linearly independent over $\CR$.

\smallskip

\noindent
\textit{Proof of Claim~1.}
We prove this by induction on $\height(I)$.
The base case $\height(I)=0$ is clear since $I=\emptyset$ and $D(\A_I)=\Der(\CR)$. 
Now we assume that $m>0$ and Claim~1 holds for any lower ideal $I'$ with $\height(I')=m-1$.
For a lower ideal $I$ with $\height(I)=m$, we define a lower ideal $I'$ by $I':=I \cap I_{m-1}$.
Namely, $I'=\{\alpha \in I \mid \height(\alpha) \leq m-1 \}$ by the definition \eqref{eq:Im}.
One can see that 
\begin{equation*}
h_{I'}(i)=\begin{cases}
h_I(i)-1 \ \ \ &{\rm if} \ i \in \Lambda_I, \\
h_I(i) \ \ \ &{\rm if} \ i \notin \Lambda_I.
\end{cases}
\end{equation*}
\noindent
\textit{Case~1} Suppose that $i \notin \Lambda_I$. 
In this case we have $h_I(i)=h_{I'}(i)$, so we prove that $\psi_{i,h_{I'}(i)} \in D(\A_{I})$.
Since $\height(I')=m-1$, we have $\psi_{i,h_{I'}(i)} \in D(\A_{I'})$ by inductive assumption. 
Noting that $I=I' \cup \{\alpha_{j,j+m} \mid j \in \Lambda_I \}$, it is enough to prove that $\psi_{i,h_{I'}(i)}(\alpha_{j,j+m}) \in \CR \alpha_{j,j+m}$ for any $j \in \Lambda_I$.
Let $H_{i,j}$ be the hyperplane defined by a positive root $\alpha_{i,j}$.
For each $j \in \Lambda_{I}$ we define a hyperplane arrangement $\A''_j$ in $H_{j,j+m}$ by $\A''_j=\{H \cap H_{j,j+m} \mid H \in \A_{I'} \}$ and we take a map 
$$
\nu_j: \A''_j \to \A_{I'}
$$
such that $\nu_j(X) \cap H_{j,j+m} =X$. 
The homogeneous polynomial 
$$
b_{\nu_j}:= \frac{\prod_{H \in \A_{I'}} \alpha_H}{\prod_{X \in \A''_j} \alpha_{\nu_j(X)}}= \frac{\prod_{\alpha \in I'} \alpha}{\prod_{X \in \A''_j} \alpha_{\nu_j(X)}}
$$
has degree $m-1$ by Proposition~\ref{proposition:ABCHTdegree}, and we obtain
\begin{equation} \label{eq:ProofProp1}
\psi_{i,h_{I'}(i)}(\alpha_{j,j+m}) \in \CR(\alpha_{j,j+m}, b_{\nu_j})
\end{equation}
from Proposition~\ref{proposition:idealbnu}.
If $i \notin \Lambda_{I'}$, then we have $\deg(\psi_{i,h_{I'}(i)})=h_{I'}(i)-i<\height(I')=m-1=\deg(b_{\nu_j})$.
Hence by \eqref{eq:ProofProp1} we obtain $\psi_{i,h_{I'}(i)}(\alpha_{j,j+m}) \in \CR(\alpha_{j,j+m})$ for $i \notin \Lambda_{I'}$.
If $i \in \Lambda_{I'}$, then we have $h_{I'}(i)=i+m-1$. 
Note that $i \in \Lambda_{I'} \subset \Lambda_{m-1}$ and $j \in \Lambda_I \subset \Lambda_m$.
We also have $i \neq j$ because $i \notin \Lambda_I$ and $j \in \Lambda_I$.
Thus, it follows from the condition~$(4)$ that $\psi_{i,h_{I'}(i)}(\alpha_{j,j+m})=\psi_{i,i+m-1}(\alpha_{j,j+m}) \in \CR(\alpha_{j,j+m})$.

\smallskip

\noindent
\textit{Case~2} Suppose that $i \in \Lambda_I$. Then, we have $h_I(i)=i+m=h_m(i)$.
From the condition~$(3)$ and $\A_I \subset \A_{I_m}$, we obtain $\psi_{i,h_I(i)}=\psi_{i,h_m(i)} \in D(\A_{I_m}) \subset D(\A_I)$.

Case~1 and Case~2 show Claim~1.

\smallskip

\noindent
\textit{Proof of Claim~2.}
In order to prove Claim~2, we prove Claim~2\!\'{} as follows:

\smallskip

\noindent
\textit{Claim~2\!\'{}} For any lower ideal $I$ with $\height(I)=m$ and any matrix $F=[f_{ij}]_{i \in \Lambda_I \atop j \in \Lambda_m} \in \M(\Lambda_I, \Lambda_m; \CR)$ such that row vectors $\textbf{f}_i=[f_{ij}]_{j \in \Lambda_m} \ (i \in \Lambda_I)$ are linearly independent over $\CR$, we put $[\xi_i^F]_{i \in \Lambda_I}:=F[\psi_{i,i+m}]_{i \in \Lambda_m}$.
Then, the derivations $\{\xi_i^F \mid i \in \Lambda_I\} \cup \{\psi_{i,h_I(i)} \mid i \notin \Lambda_I \}$ are linearly independent over $\CR$.\\

\smallskip

We prove Claim~2\!\'{} by induction on $\height(I)$.
The base case $\height(I)=0$ is clear since $I=\emptyset$ and the derivations $\{\psi_{i,i}=p_i \, \alpha_i^* \mid i \in [n] \}$ are $\CR$-basis of $D(\A_I)=\Der(\CR)$.

Now we assume that $m>0$ and Claim~2\!\'{} holds for any lower ideal $I'$ with $\height(I')=m-1$ and any matrix $F'\in \M(\Lambda_{I'}, \Lambda_{m-1}; \CR)$ such that row vectors in $F'$ are linearly independent over $\CR$.
For a lower ideal $I$ with $\height(I)=m$ and any matrix $F\in \M(\Lambda_{I}, \Lambda_{m}; \CR)$ such that row vectors in $F$ are linearly independent over $\CR$, we define the lower ideal $I':=I \cap I_{m-1}$.
From the condition~$(2)$ we have
\begin{align*}
\begin{bmatrix}
[\xi_i^F]_{i \in \Lambda_I} \\
[\psi_{i,h_I(i)}]_{i \notin \Lambda_I}
\end{bmatrix}
=&
\begin{bmatrix}
F[\psi_{i,i+m}]_{i \in \Lambda_m} \\
[\psi_{i,h_I(i)}]_{i \notin \Lambda_I}
\end{bmatrix}
=
\begin{bmatrix}
FP_m[\alpha_{i,i+m}\psi_{i,i+m-1}]_{i \in \Lambda_m} \\
[\psi_{i,h_I(i)}]_{i \notin \Lambda_I}
\end{bmatrix}\\
=&
\begin{bmatrix}
FP_m\diag(\alpha_{i,i+m})[\psi_{i,i+m-1}]_{i \in \Lambda_m} \\
[\psi_{i,h_I(i)}]_{i \notin \Lambda_I}
\end{bmatrix}\\
=&
\begin{bmatrix}
FP_m\diag(\alpha_{i,i+m})[\psi_{i,i+m-1}]_{i \in \Lambda_m} \\
[\psi_{i,i+m-1}]_{i \in \Lambda_{I'}\setminus\Lambda_I} \\
[\psi_{i,h_I(i)}]_{i \notin \Lambda_{I'}} \\
\end{bmatrix}\\
=&
\begin{bmatrix}
F'[\psi_{i,i+m-1}]_{i \in \Lambda_{m-1}} \\
[\psi_{i,h_I(i)}]_{i \notin \Lambda_{I'}}
\end{bmatrix}
=
\begin{bmatrix}
[\xi_i^{F'}]_{i \in \Lambda_{I'}} \\
[\psi_{i,h_I(i)}]_{i \notin \Lambda_{I'}}
\end{bmatrix}
\end{align*}
where $[\xi_i^{F'}]_{i \in \Lambda_{I'}}:=F'[\psi_{i,i+m-1}]_{i \in \Lambda_{m-1}}$ and $F'=[f'_{ij}]_{i \in \Lambda_{I'} \atop j \in \Lambda_{m-1}}$ is the matrix given by

\begin{equation*}
F'=\begin{array}{rcclll}
\ldelim[{2}{0.5ex}[] &\multicolumn{1}{c|}{FP_m\diag(\alpha_{i,i+m})}&O& \rdelim]{2}{0.5ex}[]&\hspace{-8pt} \rdelim\}{1}{0.2ex}[{\scriptsize $\Lambda_I$}]&\hspace{-5pt} \rdelim\}{2}{0.5ex}[{\scriptsize $\Lambda_{I'}$}] \\
\cline{2-3}
&\multicolumn{1}{c|}{O}& \delta_{ij}&&&\\
&\raisebox{2ex}[1ex][0ex]{$\underbrace{\hspace{20ex}}_{\Lambda_m}$}&&&&\\
&\multicolumn{2}{c}{\raisebox{2ex}[1ex][0ex]{$\underbrace{\hspace{25ex}}_{\Lambda_{m-1}}$}}&&&
\end{array}
\end{equation*}

More precisely, 
\begin{equation*}
f'_{ij}=\begin{cases}
(i,j)\mbox{-entry of} \ [FP_m\diag(\alpha_{i,i+m})] & {\rm if} \ (i,j)\in\Lambda_I \times \Lambda_m, \\
\delta_{ij} & {\rm if} \ (i,j)\in(\Lambda_{I'}\setminus\Lambda_I) \times (\Lambda_{m-1}\setminus\Lambda_m), \\
0 & {\rm otherwise}.  
\end{cases}
\end{equation*}
It is clear that row vectors in $F'$ are linearly independent over $\CR$.
Hence by the inductive assumption, $\{\xi_i^{F'} \mid i \in \Lambda_{I'}\} \cup \{\psi_{i,h_I(i)} \mid i \notin \Lambda_{I'} \}$ are linearly independent over $\CR$, so is $\{\xi_i^F \mid i \in \Lambda_I\} \cup \{\psi_{i,h_I(i)} \mid i \notin \Lambda_I \}$. 
Applying Claim~2\!\'{} to $F=[\delta_{ij}]_{i \in \Lambda_I \atop j \in \Lambda_m}$, we obtain Claim~2.

Therefore, we proved Claim~1 and Claim~2. This completes the proof.
\end{proof}

\begin{remark} \label{remark:key}
Assume that the derivations $\{\psi_{i,j} \in \Der(\CR) \mid 1 \leq i \leq n \ {\rm and} \ i \leq j \leq i+\e_i \}$ satisfy the following conditions: 
\begin{enumerate}
\item For any $i\in[n]$ we can write $\psi_{i,i}=p_i \, \alpha_i^*$ for some non-zero rational number $p_i$; 
\item For any $m \geq 1$ we can write 
$$
[\psi_{i,i+m}]_{i \in \Lambda_m}=P_m[\alpha_{i,i+m}\psi_{i,i+m-1}]_{i \in \Lambda_m}
$$
for some invertible matrix $P_m \in \GL(\Lambda_m;\Q)$;
\item For any lower ideal $I \subset \Phi^+$, $\{\psi_{i,h_I(i)} \mid i \in [n] \} \subset D(\A_I)$.
\end{enumerate}
Then the derivations $\{\psi_{i,j} \in \Der(\CR) \mid 1 \leq i \leq n \ {\rm and} \ i \leq j \leq i+\e_i \}$ form uniform bases by Proposition~\ref{proposition:key}.
In fact, if we consider the case $I=I_m$ in (3) above, then the derivations $\{\psi_{i,j} \}_{i,j}$ satisfy the condition~$(3)$ in Proposition~\ref{proposition:key}.
If we fix $i$ and take $I=I_m\cup\{\alpha_{j,j+m+1}\}_{j \in \Lambda_{m+1} \atop j \neq i}$ in (3) above, then the condition~$(4)$ in Proposition~\ref{proposition:key} holds.
\end{remark}

\section{Main theorem}
\label{section:maintheorem}

In this section we prove the existence of uniform bases which are inductively constructed by using invertible matrices (Theorem~\ref{theorem:main1}).
Moreover, we prove that the invertible matrices associated with our uniform bases are unique in some sense (Theorem~\ref{theorem:main1uniqueness}).

\begin{theorem}\label{theorem:main1}
For arbitrary decomposition $\Phi^+=\coprod_{i=1}^n \ \Phi^+_i$ satisfying \eqref{eq:decomposition1} and \eqref{eq:decomposition2}, there exist uniform bases $\{\psi_{i,i+m} \in \Der(\CR) \mid 0 \leq m \leq \height(\Phi^+) \ {\rm and} \  i \in \Lambda_m \}$ of the following form.
The initial data are of the form
\begin{equation*} 
\psi_{i,i}=p_i \, \alpha_i^* \ \ \ \ \ (i \in [n])
\end{equation*}
where $p_i$ is arbitrary non-zero rational number and $\alpha_1^*, \ldots, \alpha_n^*$ is the dual basis of the simple roots $\alpha_1, \ldots, \alpha_n$.
For any $m$ with $1 \leq m \leq \height(\Phi^+)$, 
\begin{equation*} 
\psi_{i,i+m}=\sum_{j \in \Lambda_m}p_{ij}^{(m)} \, \alpha_{j,j+m}\psi_{j,j+m-1} \ \ \ \ \ (i \in \Lambda_m)
\end{equation*}
for some rational numbers $p_{ij}^{(m)}$. 
\end{theorem}

\begin{proof}
It suffices to construct derivations $\{\psi_{i,i+m} \in \Der(\CR) \mid 0 \leq m \leq \height(\Phi^+) \ {\rm and} \  i \in \Lambda_m \}$ satisfying the four conditions in Proposition~\ref{proposition:key}.
For the construction, we apply the method of Abe, Barakat, Cuntz, Hoge, and Terao explained in Section~\ref{section:ideal arrangements} to the sequence of the lower ideals in \eqref{eq:Im}:  
$$
I_1 \subset I_2 \subset \cdots \subset I_{\height(\Phi^+)}.
$$

For $m$ with $0 \leq m < \height(\Phi^+)$, we construct inductively derivations $\{\psi_{i,i+m}\}_{i \in \Lambda_{m}}$ and $\{\theta_{i,i+m+1}\}_{i \in \Lambda_{m+1}}$ with $\deg(\psi_{i,j})=\deg(\theta_{i,j})=j-i$ so that the derivations $\{\theta_{i,i+m+1} \mid i \in \Lambda_{m+1}\} \cup \{\psi_{i,i+m} \mid i \in \Lambda_{m} \setminus \Lambda_{m+1}\} \cup \cdots \cup \{\psi_{i,i+1} \mid i \in \Lambda_{1} \setminus \Lambda_{2}\}$ form an $\CR$-basis for $D(\A_{I_{m+1}})$ as follows.
As the base case, when $m=0$, for any non-zero rational number $p_i$ we define 
\begin{equation} \label{eq:ProofConstructionPsiBaseCase}
\psi_{i,i}=p_i \, \alpha_i^* \ \ \ \ \ {\rm for} \ i \in \Lambda_0=[n],
\end{equation}
and 
$$
\theta_{i,i+1}=\alpha_{i,i+1} \psi_{i,i} \ \ \ \ \ {\rm for} \ i \in \Lambda_1=[n].
$$
For general $m$ with $0 < m < \height(\Phi^+)$ we proceed inductively as follows. 
Let $H_{i,j}$ be the hyperplane defined by a positive root $\alpha_{i,j}$. Then we have
$$
\A_{I_{m+1}}=\A_{I_{m}} \cup \{H_{j,j+m+1} \mid j \in \Lambda_{m+1} \}.
$$
For each $j \in \Lambda_{m+1}$ we define a hyperplane arrangement $\A''_{j}$ in $H_{j,j+m+1}$ by $\A''_{j}=\{H \cap H_{j,j+m+1} \mid H \in \A_{I_{m}} \}$ and we take a map 
$$
\nu_j: \A''_{j} \to \A_{I_{m}}
$$
such that $\nu_j(X) \cap H_{j,j+m+1} =X$. 
The homogeneous polynomial 
$$
b_{\nu_j}= \frac{\prod_{H \in \A_{I_{m}}} \alpha_H}{\prod_{X \in \A''_j} \alpha_{\nu_j(X)}} = \frac{\prod_{\alpha \in I_{m}} \alpha}{\prod_{X \in \A''_j} \alpha_{\nu_j(X)}}
$$
has degree $m$ by Proposition~\ref{proposition:ABCHTdegree}, and we obtain 
\begin{equation*} 
\theta_{i,i+m}(\alpha_{j,j+m+1}) \in \CR(\alpha_{j,j+m+1}, b_{\nu_j})
\end{equation*}
for $i \in \Lambda_{m}$ and $j \in \Lambda_{m+1}$ from Proposition~\ref{proposition:idealbnu}.
Hence, we can write
\begin{equation} \label{eq:ProofMod}
\theta_{i,i+m}(\alpha_{j,j+m+1}) \equiv c_{ij}^{(\nu_j)} b_{\nu_j} \ \ \ \ \ \mbox{mod} \ \alpha_{j,j+m+1}
\end{equation}
for some rational numbers $c_{ij}^{(\nu_j)}$ ($i \in \Lambda_{m}$, $j \in \Lambda_{m+1}$).
By Proposition~\ref{proposition:ABCHTrank}, the matrix $C_{m}:=(c_{ij}^{(\nu_j)})_{i \in \Lambda_{m} \atop j \in \Lambda_{m+1}}$ has full rank $\lambda_{m+1}$. 
Thus, there exists $P_{m}=(p_{ij}^{(m)})_{i,j \in \Lambda_{m}} \in \GL(\Lambda_{m},\Q)$ such that 
\begin{equation} \label{eq:ProofPmQm}
P_{m}C_{m}=(\delta_{ij})_{i \in \Lambda_{m} \atop j \in \Lambda_{m+1}}.
\end{equation}
We define
\begin{equation} \label{eq:ProofConstructionPsi}
\psi_{i,i+m}:=\sum_{j \in \Lambda_{m}} p_{ij}^{(m)} \theta_{j,j+m}
\end{equation}
for $i \in \Lambda_{m}$ and 
\begin{equation} \label{eq:ProofConstructionTheta}
\theta_{i,i+m+1}:= \alpha_{i,i+m+1} \psi_{i,i+m}
\end{equation}
for $i \in \Lambda_{m+1}$.
From the inductive assumption together with \eqref{eq:ProofConstructionPsi} we see that
\begin{equation} \label{eq:ProofCondition(3)}
\psi_{i,h_{m}(i)} \in D(\A_{I_{m}}) \ \ \ \ \ {\rm for} \ i \in [n].
\end{equation}
By the inductive argument, it follows from Theorem~\ref{theorem:ABCHTbasis} that the derivations $\{\theta_{i,i+m+1} \mid i \in \Lambda_{m+1} \} \cup \{\psi_{i,i+m} \mid i \in \Lambda_{m}\setminus\Lambda_{m+1} \} \cup\cdots \cup \{\psi_{i,i+1} \mid i \in \Lambda_1 \setminus \Lambda_2 \}$ form an $\CR$-basis for $D(\A_{I_{m+1}})$.

Finally, when $m=\height(\Phi^+)$, we define 
$$
\psi_{i,i+m}=\theta_{i,i+m}
$$ 
for $i \in \Lambda_{m}$.
Note that $|\Lambda_m|=1$ whenever $m=\height(\Phi^+)$ because the root with the highest height is uniquely determined. 
Hence, if $m=\height(\Phi^+)$, then we have $[\psi_{i,i+m}]_{i \in \Lambda_m}=P_m [\theta_{i,i+m}]_{i \in \Lambda_m}$ where $P_{m}=(1) \in \GL(\Lambda_{m},\Q)$.

Now we check that the derivations $\{\psi_{i,i+m} \in \Der(\CR) \mid 0 \leq m \leq \height(\Phi^+) \ {\rm and} \  i \in \Lambda_m \}$ satisfy the conditions~$(1), (2), (3), (4)$ in Proposition~\ref{proposition:key}.
The condition~$(1)$ is exactly the definition in \eqref{eq:ProofConstructionPsiBaseCase}.
The condition~$(2)$ follows from \eqref{eq:ProofConstructionPsi} and \eqref{eq:ProofConstructionTheta}. 
The condition~$(3)$ is nothing but \eqref{eq:ProofCondition(3)}.
We check the condition~$(4)$.
If $m=0$, then $\psi_{i,i}(\alpha_{j,j+1})=p_i\alpha_i^*(\alpha_j)=0$ for $i \in \Lambda_0=[n]$ and $j \in \Lambda_{1}=[n]$ with $i \neq j$. 
If $m \geq 1$, then we have for $i \in \Lambda_m$ and $j \in \Lambda_{m+1}$ with $i \neq j$
\begin{align*}
\psi_{i,i+m}(\alpha_{j,j+m+1})&=\sum_{k \in \Lambda_{m}} p_{ik}^{(m)} \theta_{k,k+m}(\alpha_{j,j+m+1}) \ \ \ \ \ \ ({\rm from} \ \eqref{eq:ProofConstructionPsi})\\
&\equiv \sum_{k \in \Lambda_{m}} p_{ik}^{(m)} c_{kj}^{(\nu_j)} b_{\nu_j} \ \ \ \ \ \mbox{mod} \ \alpha_{j,j+m+1} \ \ \ \ \ \ ({\rm by} \ \eqref{eq:ProofMod})\\
&= \delta_{ij} b_{\nu_j} \ \ \ \ \ \ ({\rm from} \ \eqref{eq:ProofPmQm})\\
&=0 \ \ \ \ \ \ ({\rm because} \ i \neq j).
\end{align*}
Therefore, we conclude that the derivations $\{\psi_{i,i+m} \in \Der(\CR) \mid 0 \leq m \leq \height(\Phi^+) \ {\rm and} \  i \in \Lambda_m \}$ form uniform bases from Proposition~\ref{proposition:key}. 
\end{proof}

We obtain from Theorem~\ref{theorem:main1} uniform bases $\{\psi_{i,i+m} \in \Der(\CR) \mid 0 \leq m \leq \height(\Phi^+) \ {\rm and} \  i \in \Lambda_m \}$ by the recursive description.
For any non-zero rational numbers $p_1, \ldots, p_n$, we fix the initial data 
\begin{align*}  
\notag \psi_{i,i}=p_i \, \alpha_i^* \ \ \ \ \ \ \ \ \ \ \ \ \ \ \ \ \ \ \ \ \ \ \ \ \ \ \ \ \ \ &(i \in [n]). 
\end{align*}
Proceeding inductively,  
\begin{align} \label{eq:main1-2} 
\psi_{i,i+m}=\sum_{j \in \Lambda_m}p_{ij}^{(m)} \, \alpha_{j,j+m}\psi_{j,j+m-1} \ \ \ \ \ &(1 \leq m \leq \height(\Phi^+), \ i \in \Lambda_m) 
\end{align}
for some rational numbers $p_{ij}^{(m)}$.
Then we define matrices $P_m$ ($0 \leq m \leq \height(\Phi^+)$) by
\begin{align} \label{eq:matrixP}
P_0&:=\diag(p_1,\ldots,p_n), \notag \\
P_m&:=(p_{ij}^{(m)})_{i,j \in \Lambda_m} \ \ \ {\rm for} \ m>0.
\end{align}
As seen in the proof of Theorem~\ref{theorem:main1}, $P_m$ is invertible for all $m$.
We call the matrices $\{P_m \mid 0 \leq m \leq \height(\Phi^+) \}$ \textbf{the invertible matrices associated with uniform bases} $\{\psi_{i,i+m} \in \Der(\CR) \mid 0 \leq m \leq \height(\Phi^+) \ {\rm and} \  i \in \Lambda_m \}$. 
These invertible matrices are not unique.
In fact, we can multiply $\psi_{i,i+m}$ by a non-zero scalar. 
Also, $\psi_{j,j+m}$ is an element of $D(\A_{\Phi^+})$ for any $j \in \Lambda_m \setminus \Lambda_{m+1}$, so we can replace $\psi_{i,i+m}$ with the $\psi_{i,i+m}$ plus a constant times $\psi_{j,j+m}$ for arbitrary $j \in \Lambda_m \setminus \Lambda_{m+1}$ with $j \neq i$.
These correspond to the following two types of elementary row operations for matrices in $\GL(\Lambda_m;\Q)$:
\begin{enumerate}
\item multiply one row of the matrix by a non-zero scalar constant$;$
\item replace an $i$-th row with the $i$-th row plus a constant times $j$-th row for $j \in \Lambda_m \setminus \Lambda_{m+1}$ with $j \neq i$.
\end{enumerate}
For $P_m, P'_m \in \GL(\Lambda_m;\Q)$, $P_m$ is \textbf{equivalent} to $P'_m$ if $P'_m$ is obtained from $P_m$ by finitely many elementary row operations $(1)$ and $(2)$ above.
The following theorem states that $P_m$ is uniquely determined up to the equivalence when we determined the invertible matrices $P_0,\ldots,P_{m-1}$.

\begin{theorem}\label{theorem:main1uniqueness}
Let $\Phi^+=\coprod_{i=1}^n \ \Phi^+_i$ be a decomposition satisfying \eqref{eq:decomposition1} and \eqref{eq:decomposition2}.
Let $\{P_m \mid 0 \leq m \leq \height(\Phi^+) \}$ and $\{P'_m \mid 0 \leq m \leq \height(\Phi^+) \}$ be two sets of the invertible matrices associated with uniform bases $\{\psi_{i,i+m} \in \Der(\CR) \mid 0 \leq m \leq \height(\Phi^+), i \in \Lambda_m \}$ and $\{\psi'_{i,i+m} \in \Der(\CR) \mid 0 \leq m \leq \height(\Phi^+), i \in \Lambda_m \}$ respectively.
Assume that $P_r=P'_r$ for any $r=0,1.\ldots,m-1$. Then $P_m$ is equivalent to $P'_m$.
\end{theorem}

\begin{proof}
Let $\{P_m \mid 0 \leq m \leq \height(\Phi^+) \}$ be the invertible matrices associated with uniform bases $\Psi:=\{\psi_{i,i+m} \in \Der(\CR) \mid 0 \leq m \leq \height(\Phi^+), i \in \Lambda_m \}$.
Then, we can write from \eqref{eq:main1-2}
\begin{equation} \label{eq:proof2psi}
\psi_{i,i+m}=\sum_{k \in \Lambda_m} p_{ik}^{(m)} \alpha_{k,k+m}\psi_{k,k+m-1}=\sum_{k \in \Lambda_m} p_{ik}^{(m)} \theta_{k,k+m}
\end{equation}
where $P_m=(p_{ij}^{(m)})_{i,j \in \Lambda_m}$ and $\theta_{k,k+m}:=\alpha_{k,k+m}\psi_{k,k+m-1}$ for $k \in \Lambda_m$.
For $k \in \Lambda_m$, we define a lower ideal $I_m^{(k)}$ by $I_m^{(k)}=I_m \setminus \{\alpha_{k,k+m} \}$.
Since the set of derivations $\Psi$ forms uniform bases, we have $\psi_{k,k+m-1} \in D(\A_{I_m^{(k)}})$ which implies that $\theta_{k,k+m}=\alpha_{k,k+m}\psi_{k,k+m-1} \in D(\A_{I_m})$.

We write $\A''_{j}=\{H \cap H_{j,j+m+1} \mid H \in \A_{I_{m}} \}$ where $H_{j,j+m+1}$ is the hyperplane defined by the positive root $\alpha_{j,j+m+1}$.
We fix a map $\nu_j: \A''_{j} \to \A_{I_{m}}$ such that $\nu_j(X) \cap H_{j,j+m+1} =X$, and define the homogeneous polynomial $b_{\nu_j}= (\prod_{H \in \A_{I_{m}}} \alpha_H) / (\prod_{X \in \A''_j} \alpha_{\nu_j(X)})$ which has degree $m$ from Proposition~\ref{proposition:ABCHTdegree}.
It follows from Proposition~\ref{proposition:idealbnu} that for $k \in \Lambda_{m}$ and $j \in \Lambda_{m+1}$ we can write
\begin{equation*} 
\theta_{k,k+m}(\alpha_{j,j+m+1}) \equiv c_{kj}^{(\nu_j)} b_{\nu_j} \ \ \ \ \ \mbox{mod} \ \alpha_{j,j+m+1}
\end{equation*}
for some rational numbers $c_{kj}^{(\nu_j)}$.
From this together with \eqref{eq:proof2psi} we have
\begin{equation} \label{eq:proof2psi2}
\psi_{i,i+m}(\alpha_{j,j+m+1}) \equiv b_{\nu_j} \big(\sum_{k \in \Lambda_m} p_{ik}^{(m)} c_{kj}^{(\nu_j)} \big) \ \ \ \ \ \mbox{mod} \ \alpha_{j,j+m+1}.
\end{equation}
We consider a lower ideal $I:=I_m\cup\{\alpha_{j,j+m+1}\}$, and one see that $\psi_{i,i+m} \in D(\A_I)$ for $i \in \Lambda_m$ and $j \in \Lambda_{m+1}$ with $i \neq j$ since $\Psi$ form uniform bases. 
In particular, $\psi_{i,i+m}(\alpha_{j,j+m+1}) \in \CR \alpha_{j,j+m+1}$.
From this together with \eqref{eq:proof2psi2}, we obtain
$$
\sum_{k \in \Lambda_m} p_{ik}^{(m)} c_{kj}^{(\nu_j)}=0 \ \ \ \ \ \ {\rm for} \ i \in \Lambda_m, \ j \in \Lambda_{m+1} \ {\rm with} \ i \neq j.
$$
Let $C_{m}:=(c_{ij}^{(\nu_j)})_{i \in \Lambda_{m} \atop j \in \Lambda_{m+1}}$. 
The matrix $C_m$ has full rank $\lambda_{m+1}$ by Proposition~\ref{proposition:ABCHTrank}.
Therefore, we obtain 
\begin{equation} \label{eq:proof2PmCm}
P_mC_m=(q_j\delta_{ij})_{i \in \Lambda_{m} \atop j \in \Lambda_{m+1}}
\end{equation}
for some non-zero rational numbers $q_j$ ($j \in \Lambda_{m+1}$).

We take another invertible matrices $\{P'_m \mid 0 \leq m \leq \height(\Phi^+) \}$ associated with uniform bases $\{\psi'_{i,i+m} \in \Der(\CR) \mid 0 \leq m \leq \height(\Phi^+), i \in \Lambda_m \}$.
From the assumption $P_r=P'_r$ for $r=0,\ldots,m-1$, we have $\psi_{s,s+r}=\psi'_{s,s+r}$ for $0 \leq r \leq m-1$ and $s \in \Lambda_r$.
Noting that $\theta'_{k,k+m}:=\alpha_{k,k+m}\psi'_{k,k+m-1}=\alpha_{k,k+m}\psi_{k,k+m-1}=\theta_{k,k+m}$ for $k \in \Lambda_m$, by similar argument we have
\begin{equation} \label{eq:proof2P'mCm}
P'_mC_m=(q'_j\delta_{ij})_{i \in \Lambda_{m} \atop j \in \Lambda_{m+1}}
\end{equation}
for some non-zero rational numbers $q'_j$ ($j \in \Lambda_{m+1}$).
From \eqref{eq:proof2PmCm} and \eqref{eq:proof2P'mCm}, we have 
\begin{equation} \label{eq:AmA'm}
A_m C_m= (\delta_{ij})_{i \in \Lambda_{m} \atop j \in \Lambda_{m+1}} =A'_m C_m 
\end{equation}
where $A_m=(q_j^{-1}\delta_{ij})_{i \in \Lambda_{m} \atop j \in \Lambda_{m+1}} \cdot P_m$ and $A'_m=({q'}_j^{-1}\delta_{ij})_{i \in \Lambda_{m} \atop j \in \Lambda_{m+1}} \cdot P'_m$.
Let $\mathbf{a}_i$ and $\mathbf{a}'_i$ be $i$-th row vectors of $A_m$ and $A'_m$ respectively for $i\in \Lambda_m$.
Then, the difference $\mathbf{a}'_i-\mathbf{a}_i$ belongs to the kernel of $C_m$ by \eqref{eq:AmA'm}.
For any $j \in \Lambda_m \setminus \Lambda_{m+1}$ the $j$-th row vector $\mathbf{p}_j$ of $P_m$ is an element of $\ker C_m$ by \eqref{eq:proof2PmCm}, and these row vectors $\mathbf{p}_j$ are linearly independent because $P_m$ is invertible.
Since the matrix $C_m$ has rank $\lambda_{m+1}$, the row vectors $\mathbf{p}_j$ $(j \in \Lambda_m \setminus \Lambda_{m+1})$ form a basis of $\ker C_m$. 
Hence, the difference $\mathbf{a}'_i-\mathbf{a}_i$ can be written as a linear combination of $\mathbf{p}_j$ $(j \in \Lambda_m \setminus \Lambda_{m+1})$, which means that $P_m$ is equivalent to $P'_m$, as desired.
\end{proof}

Uniform bases for the ideal arrangements are explicitly given in \cite{AHMMS} for types $A, B, C, G$. 
Note that one can find the associated invertible matrices $P_m \ (0 \leq m \leq \height(\Phi^+))$ from their explicit presentations for uniform bases in types $A,B,C,G$. 
We will explicitly describe uniform bases and invertible matrices for type $D$ in next section.
In order to find uniform bases for types $E$ and $F$, we will compute invertible matrices $P_{0}, P_{1}, \ldots, P_{\height(\Phi^+)}$ in order by using Maple. 
We will explain the details in Appendix.

\section{Uniform bases and invertible matrices for type $D$}
\label{section:typeD}

Let $\mathfrak{t}$ be an $n$-dimensional real Euclidean space and $\mathfrak{t}^*$ denotes the dual space of $\mathfrak{t}$.
Let $x_1,\ldots,x_n$ be an orthonormal basis of $\mathfrak{t}^*$.
Then we have 
\[
\CR = \Sym(\mathfrak{t}^*)=\R[x_1,\ldots,x_n].
\]
We take the set of positive roots of type $D_n$ as
\[
\Phi^{+}_{D_n} = \{x_i \pm x_j \in \mathfrak{t}^* \mid 1 \leq i < j \leq n\}
\]
and set the exponents $\e_1,\ldots,\e_n$ as 
\begin{align*}
e_i &= 2(n-i)-1 \ \ \ \ \ \ \mbox{for} \ 1 \leq i \leq n-1, \\
e_n &= n-1.  
\end{align*}
We arrange all positive roots of $D_n$ as shown in Figure~\ref{picture:PositiveRootTypeD}.

\begin{figure}[h]
\begin{center}
\begin{picture}(400,100)

\put(295,85){$\cdots$}
\put(160,80){\dashbox(40,20){\tiny $x_1-x_n$}} 
\put(200,80){\dashbox(40,20){\tiny $x_1+x_n$}} 
\put(360,80){\framebox(40,20){\tiny $x_1+x_2$}} 
\put(95,85){$\cdots$}
\put(0,80){\framebox(40,20){\tiny $x_1-x_2$}} 
\put(55,65){$\ddots$}
\put(178,65){$\vdots$}
\put(218,65){$\vdots$}
\put(335,65){$\cdot$} 
\put(340,67.5){$\cdot$}
\put(345,70){$\cdot$}

\put(160,40){\dashbox(40,20){\tiny $x_i-x_n$}}
\put(255,45){$\cdots$}
\put(200,40){\dashbox(40,20){\tiny $x_i+x_n$}}
\put(135,45){$\cdots$}
\put(280,40){\framebox(40,20){\tiny $x_i+x_{i+1}$}}
\put(80,40){\framebox(40,20){\tiny $x_i-x_{i+1}$}}
\put(135,25){$\ddots$}
\put(178,25){$\vdots$}
\put(218,25){$\vdots$}
\put(255,25){$\cdot$} 
\put(260,27.5){$\cdot$}
\put(265,30){$\cdot$}
\put(200,0){\dashbox(40,20){\tiny $x_{n-1}+x_n$}}
\put(160,0){\dashbox(40,20){\tiny $x_{n-1}-x_n$}}

\put(160,80){\line(1,0){80}}
\put(160,100){\line(1,0){80}}
\put(160,80){\line(0,1){20}}
\put(240,80){\line(0,1){20}}

\put(160,40){\line(1,0){80}}
\put(160,60){\line(1,0){80}}
\put(160,40){\line(0,1){20}}
\put(240,40){\line(0,1){20}}

\put(160,0){\line(1,0){80}}
\put(160,20){\line(1,0){80}}
\put(160,0){\line(0,1){20}}
\put(240,0){\line(0,1){20}}

\end{picture}
\end{center}
\vspace{-10pt}
\caption{The arrangement of all positive roots for type $D_n$.}
\label{picture:PositiveRootTypeD}
\end{figure} 

In Figure~\ref{picture:PositiveRootTypeD} the partial order $\preceq$ on $\Phi^+_{D_n}$ is defined as follows:
\begin{enumerate}
\item if a root $\alpha$ is left-adjacent to a root $\beta$, then $\alpha \lessdot \beta$, except for $(\alpha,\beta)=(x_i-x_n,x_i+x_n)$ $(1 \leq i \leq n-1)$ which are pictorially divided by a dotted line;
\item $x_i-x_{n-1} \lessdot x_i+x_n$ and $x_i-x_n \lessdot x_i+x_{n+1}$ for $1 \leq i \leq n-1$;
\item if a root $\alpha$ is lower-adjacent to a root $\beta$, then $\alpha \lessdot \beta$.
\end{enumerate}
For two positive roots $\alpha,\beta$, we define $\alpha \preceq \beta$ if there exist positive roots $\gamma_0,\ldots,\gamma_N$ such that $\alpha=\gamma_0 \lessdot \gamma_1 \lessdot \cdots \lessdot \gamma_N=\beta$.
We denote positive roots in $\Phi^{+}_{D_n}$ by 
\begin{equation*} 
\alpha_{i,j}=
\begin{cases}
x_i-x_j  \ \ \ \ \ &{\rm if} \ i+1 \leq j \leq n, \\ 
x_i+x_{2n-j}  \ \ \ \ \ &{\rm if} \ n+1 \leq j \leq 2n-i-1 
\end{cases}
\end{equation*}
for each $i=1,\ldots,n-1$, and 
\begin{equation*} 
\alpha_{n,j}=x_{2n-j}+x_n  \ \ \ \ \ {\rm if} \ n+1 \leq j \leq 2n-1. 
\end{equation*}

Note that $\Phi^+_i$ is the set of all positive roots in the $i$-th row except for the root $x_i+x_n$ for $i=1,\ldots,n-1$, and $\Phi^+_n$ is the set of all positive roots in the $(n+1)$-th column in Figure~\ref{picture:PositiveRootTypeD}.
Motivated by this, we define the coordinate in type $D_n$ as shown in Figure~\ref{picture:CoordinateTypeD}.

\begin{figure}[h]
\begin{center}
\begin{picture}(470,120)
\put(380,105){$\cdots$}
\put(155,100){\framebox(50,20){\tiny $(1,n-1)$}} 
\put(245,100){\framebox(40,20){\tiny $(n,2n-1)$}} 
\put(430,100){\framebox(40,20){\tiny $(1,2n-2)$}} 
\put(90,105){$\cdots$}
\put(25,100){\framebox(25,20){\tiny $(1,2)$}} 
\put(50,85){$\ddots$}
\put(180,85){$\vdots$}
\put(263,85){$\vdots$}
\put(412,85){$\cdot$} 
\put(417,87.5){$\cdot$}
\put(422,90){$\cdot$}

\put(155,60){\framebox(50,20){\tiny $(i,n-1)$}}
\put(340,65){$\cdots$}
\put(245,60){\framebox(40,20){\tiny $(n,2n-i)$}}
\put(128,65){$\cdots$}
\put(360,60){\framebox(50,20){\tiny $(i,2n-1-i)$}}
\put(80,60){\framebox(35,20){\tiny $(i,i+1)$}}
\put(125,45){$\ddots$}
\put(180,45){$\vdots$}
\put(263,45){$\vdots$}
\put(340,45){$\cdot$} 
\put(345,47.5){$\cdot$}
\put(350,50){$\cdot$}
\put(245,20){\framebox(40,20){\tiny $(n,n+2)$}}
\put(155,20){\framebox(50,20){\tiny $(n-2,n-1)$}}
\put(155,0){\framebox(50,20){\tiny $(n-1,n-1)$}}

\put(205,100){\framebox(40,20){\tiny $(1,n)$}} 
\put(223,85){$\vdots$}
\put(205,60){\framebox(40,20){\tiny $(i,n)$}}
\put(223,45){$\vdots$}
\put(205,20){\framebox(40,20){\tiny $(n-2,n)$}}
\put(205,0){\framebox(40,20){\tiny $(n-1,n)$}}

\put(285,100){\framebox(50,20){\tiny $(1,n+1)$}} 
\put(310,85){$\vdots$} 
\put(285,60){\framebox(50,20){\tiny $(i,n+1)$}} 
\put(310,45){$\vdots$} 
\put(285,20){\framebox(50,20){\tiny $(n-2,n+1)$}} 

\put(0,100){\tiny \framebox(25,20){$(1,1)$}} 
\put(55,60){\tiny \framebox(25,20){$(i,i)$}} 
\put(105,20){\tiny \framebox(50,20){$(n-2,n-2)$}} 
\put(245,0){\tiny \framebox(40,20){$(n,n+1)$}} 
\put(245,-20){\tiny \framebox(40,20){$(n,n)$}} 

\put(20,85){$\ddots$}
\put(80,45){$\ddots$}
\end{picture}
\end{center}
\vspace{10pt}
\caption{The coordinate in type $D_n$.}
\label{picture:CoordinateTypeD}
\end{figure} 

We define a \textbf{Hessenberg function for type $D_n$} as a function $h: [n] \to [2n-1]$ satisfying the following conditions
\begin{enumerate}
\item $i \leq h(i) \leq 2n-1-i$ for $i=1,\ldots, n-1$, 
\item $n \leq h(n) \leq 2n-1$, 
\item if $h(i) \neq 2n-1-i$, then $h(i) \leq h(i+1)$ for $i=1,\ldots, n-2$, 
\item if $h(i) = 2n-1-i$, then $h(i+1) = 2n-1-(i+1)$ for $i=1,\ldots, n-2$,
\item if $h(i) \geq n+1$, then $h(n) \geq 2n-i$ for $i=1,\ldots, n-2$,\footnote{The condition $(5)$ is true for $i=n-1$ because $h(n-1) = n-1$ or $h(n-1)=n$.}
\item if $h(n) \geq 2n-i$, then $h(i) \geq n-1$ for $i=1,\ldots, n-2$.
\end{enumerate}
Under the decomposition $\Phi^+_{D_n}=\coprod_{i=1}^{n} \ \Phi^+_i$, one can see that the set of lower ideals $I \subset \Phi^+_{D_n}$ and the set of Hessenberg functions $h$ for type $D_n$ are in one-to-one correspondence which sends $I$ to $h_I$ in \eqref{eq:Hessft}.
We write a Hessenberg function by listing its values in sequence, namely, $h = (h(1), h(2), \ldots , h(n))$.

It is useful to express a Hessenberg function $h$ pictorially by drawing a configuration of boxes on a square grid with the coordinate in Figure~\ref{picture:CoordinateTypeD} whose shaded boxes correspond to the roots of the lower ideal $I$ associated with $h$ and $(i,i)$-boxes (see for example \cite{AHHM, AHMMS} for type $A$). 
For example the picture of the Hessenberg function $h=(3,5,4,7)$ is shown in Figure~\ref{picture:The picture for h=(3,5,4,7)}.

\begin{figure}[h]
\begin{center}
\begin{picture}(210,50)
\put(0,48){\colorbox{gray}}
\put(5,48){\colorbox{gray}}
\put(10,48){\colorbox{gray}}
\put(15,48){\colorbox{gray}}
\put(20,48){\colorbox{gray}}
\put(24,48){\colorbox{gray}}
\put(0,53){\colorbox{gray}}
\put(5,53){\colorbox{gray}}
\put(10,53){\colorbox{gray}}
\put(15,53){\colorbox{gray}}
\put(20,53){\colorbox{gray}}
\put(24,53){\colorbox{gray}}
\put(0,57){\colorbox{gray}}
\put(5,57){\colorbox{gray}}
\put(10,57){\colorbox{gray}}
\put(15,57){\colorbox{gray}}
\put(20,57){\colorbox{gray}}
\put(24,57){\colorbox{gray}}

\put(30,48){\colorbox{gray}}
\put(35,48){\colorbox{gray}}
\put(40,48){\colorbox{gray}}
\put(45,48){\colorbox{gray}}
\put(50,48){\colorbox{gray}}
\put(54,48){\colorbox{gray}}
\put(30,53){\colorbox{gray}}
\put(35,53){\colorbox{gray}}
\put(40,53){\colorbox{gray}}
\put(45,53){\colorbox{gray}}
\put(50,53){\colorbox{gray}}
\put(54,53){\colorbox{gray}}
\put(30,57){\colorbox{gray}}
\put(35,57){\colorbox{gray}}
\put(40,57){\colorbox{gray}}
\put(45,57){\colorbox{gray}}
\put(50,57){\colorbox{gray}}
\put(54,57){\colorbox{gray}}

\put(60,48){\colorbox{gray}}
\put(65,48){\colorbox{gray}}
\put(70,48){\colorbox{gray}}
\put(75,48){\colorbox{gray}}
\put(80,48){\colorbox{gray}}
\put(84,48){\colorbox{gray}}
\put(60,53){\colorbox{gray}}
\put(65,53){\colorbox{gray}}
\put(70,53){\colorbox{gray}}
\put(75,53){\colorbox{gray}}
\put(80,53){\colorbox{gray}}
\put(84,53){\colorbox{gray}}
\put(60,57){\colorbox{gray}}
\put(65,57){\colorbox{gray}}
\put(70,57){\colorbox{gray}}
\put(75,57){\colorbox{gray}}
\put(80,57){\colorbox{gray}}
\put(84,57){\colorbox{gray}}

\put(120,48){\colorbox{gray}}
\put(125,48){\colorbox{gray}}
\put(130,48){\colorbox{gray}}
\put(135,48){\colorbox{gray}}
\put(140,48){\colorbox{gray}}
\put(144,48){\colorbox{gray}}
\put(120,53){\colorbox{gray}}
\put(125,53){\colorbox{gray}}
\put(130,53){\colorbox{gray}}
\put(135,53){\colorbox{gray}}
\put(140,53){\colorbox{gray}}
\put(144,53){\colorbox{gray}}
\put(120,57){\colorbox{gray}}
\put(125,57){\colorbox{gray}}
\put(130,57){\colorbox{gray}}
\put(135,57){\colorbox{gray}}
\put(140,57){\colorbox{gray}}
\put(144,57){\colorbox{gray}}

\put(30,33){\colorbox{gray}}
\put(35,33){\colorbox{gray}}
\put(40,33){\colorbox{gray}}
\put(45,33){\colorbox{gray}}
\put(50,33){\colorbox{gray}}
\put(54,33){\colorbox{gray}}
\put(30,38){\colorbox{gray}}
\put(35,38){\colorbox{gray}}
\put(40,38){\colorbox{gray}}
\put(45,38){\colorbox{gray}}
\put(50,38){\colorbox{gray}}
\put(54,38){\colorbox{gray}}
\put(30,42){\colorbox{gray}}
\put(35,42){\colorbox{gray}}
\put(40,42){\colorbox{gray}}
\put(45,42){\colorbox{gray}}
\put(50,42){\colorbox{gray}}
\put(54,42){\colorbox{gray}}

\put(60,33){\colorbox{gray}}
\put(65,33){\colorbox{gray}}
\put(70,33){\colorbox{gray}}
\put(75,33){\colorbox{gray}}
\put(80,33){\colorbox{gray}}
\put(84,33){\colorbox{gray}}
\put(60,38){\colorbox{gray}}
\put(65,38){\colorbox{gray}}
\put(70,38){\colorbox{gray}}
\put(75,38){\colorbox{gray}}
\put(80,38){\colorbox{gray}}
\put(84,38){\colorbox{gray}}
\put(60,42){\colorbox{gray}}
\put(65,42){\colorbox{gray}}
\put(70,42){\colorbox{gray}}
\put(75,42){\colorbox{gray}}
\put(80,42){\colorbox{gray}}
\put(84,42){\colorbox{gray}}

\put(90,33){\colorbox{gray}}
\put(95,33){\colorbox{gray}}
\put(100,33){\colorbox{gray}}
\put(105,33){\colorbox{gray}}
\put(110,33){\colorbox{gray}}
\put(114,33){\colorbox{gray}}
\put(90,38){\colorbox{gray}}
\put(95,38){\colorbox{gray}}
\put(100,38){\colorbox{gray}}
\put(105,38){\colorbox{gray}}
\put(110,38){\colorbox{gray}}
\put(114,38){\colorbox{gray}}
\put(90,42){\colorbox{gray}}
\put(95,42){\colorbox{gray}}
\put(100,42){\colorbox{gray}}
\put(105,42){\colorbox{gray}}
\put(110,42){\colorbox{gray}}
\put(114,42){\colorbox{gray}}

\put(120,33){\colorbox{gray}}
\put(125,33){\colorbox{gray}}
\put(130,33){\colorbox{gray}}
\put(135,33){\colorbox{gray}}
\put(140,33){\colorbox{gray}}
\put(144,33){\colorbox{gray}}
\put(120,38){\colorbox{gray}}
\put(125,38){\colorbox{gray}}
\put(130,38){\colorbox{gray}}
\put(135,38){\colorbox{gray}}
\put(140,38){\colorbox{gray}}
\put(144,38){\colorbox{gray}}
\put(120,42){\colorbox{gray}}
\put(125,42){\colorbox{gray}}
\put(130,42){\colorbox{gray}}
\put(135,42){\colorbox{gray}}
\put(140,42){\colorbox{gray}}
\put(144,42){\colorbox{gray}}

\put(150,33){\colorbox{gray}}
\put(155,33){\colorbox{gray}}
\put(160,33){\colorbox{gray}}
\put(165,33){\colorbox{gray}}
\put(170,33){\colorbox{gray}}
\put(174,33){\colorbox{gray}}
\put(150,38){\colorbox{gray}}
\put(155,38){\colorbox{gray}}
\put(160,38){\colorbox{gray}}
\put(165,38){\colorbox{gray}}
\put(170,38){\colorbox{gray}}
\put(174,38){\colorbox{gray}}
\put(150,42){\colorbox{gray}}
\put(155,42){\colorbox{gray}}
\put(160,42){\colorbox{gray}}
\put(165,42){\colorbox{gray}}
\put(170,42){\colorbox{gray}}
\put(174,42){\colorbox{gray}}

\put(60,18){\colorbox{gray}}
\put(65,18){\colorbox{gray}}
\put(70,18){\colorbox{gray}}
\put(75,18){\colorbox{gray}}
\put(80,18){\colorbox{gray}}
\put(84,18){\colorbox{gray}}
\put(60,23){\colorbox{gray}}
\put(65,23){\colorbox{gray}}
\put(70,23){\colorbox{gray}}
\put(75,23){\colorbox{gray}}
\put(80,23){\colorbox{gray}}
\put(84,23){\colorbox{gray}}
\put(60,27){\colorbox{gray}}
\put(65,27){\colorbox{gray}}
\put(70,27){\colorbox{gray}}
\put(75,27){\colorbox{gray}}
\put(80,27){\colorbox{gray}}
\put(84,27){\colorbox{gray}}

\put(90,18){\colorbox{gray}}
\put(95,18){\colorbox{gray}}
\put(100,18){\colorbox{gray}}
\put(105,18){\colorbox{gray}}
\put(110,18){\colorbox{gray}}
\put(114,18){\colorbox{gray}}
\put(90,23){\colorbox{gray}}
\put(95,23){\colorbox{gray}}
\put(100,23){\colorbox{gray}}
\put(105,23){\colorbox{gray}}
\put(110,23){\colorbox{gray}}
\put(114,23){\colorbox{gray}}
\put(90,27){\colorbox{gray}}
\put(95,27){\colorbox{gray}}
\put(100,27){\colorbox{gray}}
\put(105,27){\colorbox{gray}}
\put(110,27){\colorbox{gray}}
\put(114,27){\colorbox{gray}}

\put(120,18){\colorbox{gray}}
\put(125,18){\colorbox{gray}}
\put(130,18){\colorbox{gray}}
\put(135,18){\colorbox{gray}}
\put(140,18){\colorbox{gray}}
\put(144,18){\colorbox{gray}}
\put(120,23){\colorbox{gray}}
\put(125,23){\colorbox{gray}}
\put(130,23){\colorbox{gray}}
\put(135,23){\colorbox{gray}}
\put(140,23){\colorbox{gray}}
\put(144,23){\colorbox{gray}}
\put(120,27){\colorbox{gray}}
\put(125,27){\colorbox{gray}}
\put(130,27){\colorbox{gray}}
\put(135,27){\colorbox{gray}}
\put(140,27){\colorbox{gray}}
\put(144,27){\colorbox{gray}}

\put(120,3){\colorbox{gray}}
\put(125,3){\colorbox{gray}}
\put(130,3){\colorbox{gray}}
\put(135,3){\colorbox{gray}}
\put(140,3){\colorbox{gray}}
\put(144,3){\colorbox{gray}}
\put(120,8){\colorbox{gray}}
\put(125,8){\colorbox{gray}}
\put(130,8){\colorbox{gray}}
\put(135,8){\colorbox{gray}}
\put(140,8){\colorbox{gray}}
\put(144,8){\colorbox{gray}}
\put(120,12){\colorbox{gray}}
\put(125,12){\colorbox{gray}}
\put(130,12){\colorbox{gray}}
\put(135,12){\colorbox{gray}}
\put(140,12){\colorbox{gray}}
\put(144,12){\colorbox{gray}}

\put(0,00){\framebox(30,15)}
\put(0,15){\framebox(30,15)} 
\put(0,30){\framebox(30,15)}
\put(0,45){\framebox(30,15){\tiny $(1,1)$}}

\put(30,00){\framebox(30,15)}
\put(30,15){\framebox(30,15)}
\put(30,30){\framebox(30,15){\tiny $(2,2)$}}
\put(30,45){\framebox(30,15){\tiny $(1,2)$}} 

\put(60,00){\framebox(30,15)} 
\put(60,15){\framebox(30,15){\tiny $(3,3)$}} 
\put(60,30){\framebox(30,15){\tiny $(2,3)$}} 
\put(60,45){\framebox(30,15){\tiny $(1,3)$}}

\put(90,00){\framebox(30,15)} 
\put(90,15){\framebox(30,15){\tiny $(3,4)$}} 
\put(90,30){\framebox(30,15){\tiny $(2,4)$}} 
\put(90,45){\framebox(30,15){\tiny $(1,4)$}}

\put(120,00){\framebox(30,15){\tiny $(4,4)$}} 
\put(120,15){\framebox(30,15){\tiny $(4,5)$}} 
\put(120,30){\framebox(30,15){\tiny $(4,6)$}} 
\put(120,45){\framebox(30,15){\tiny $(4,7)$}}

\put(150,00){\framebox(30,15)} 
\put(150,15){\framebox(30,15)} 
\put(150,30){\framebox(30,15){\tiny $(2,5)$}} 
\put(150,45){\framebox(30,15){\tiny $(1,5)$}}

\put(180,0){\framebox(30,15)} 
\put(180,15){\framebox(30,15)} 
\put(180,30){\framebox(30,15)} 
\put(180,45){\framebox(30,15){\tiny $(1,6)$}}
\end{picture}
\end{center}
\caption{The picture for the Hessenberg function $h=(3,5,4,7)$.}
\label{picture:The picture for h=(3,5,4,7)}
\end{figure} 

Now we define the derivations $\{\psi^{D_n}_{i,j} \mid 1 \leq i \leq n-1, i \leq j \leq 2n-i-1\} \cup \{\psi^{D_n}_{n,j} \mid n \leq j \leq 2n-1 \}$ by the following recursive formula.
We begin with the case when $j=i$. 
In this case we make the following definition
\begin{align*}
\psi^{D_n}_{i,i}&= \partial_1+\cdots+\partial_i, \ \ \ {\rm for} \ 1 \leq i \leq n-2,\\
\psi^{D_n}_{n-1,n-1}&= \partial_1+\cdots+\partial_{n-1}-\partial_n, \\
\psi^{D_n}_{n,n}&= \partial_1+\cdots+\partial_{n-1}+\partial_n,
\end{align*}
where $\partial_i=\frac{\partial}{\partial_{x_i}}$ denotes the partial derivatives for all $i=1,\ldots,n$.
Now we proceed inductively for the rest of the $\psi^{D_n}_{i,j}$ as follows: 
\begin{align}
\psi^{D_n}_{i,j}&= \psi^{D_n}_{i-1,j-1} + \alpha_{i,j}\psi^{D_n}_{i,j-1}+\delta_{j,n-1}\xi^{D_n}_{i} \ \ {\rm for} \ 1 \leq i \leq n-1 {\rm,} \ i+1 \leq j \leq 2n-1-i, \label{eq:psiDi}\\
\psi^{D_n}_{n,j}&= \alpha_{n,j}\psi^{D_n}_{n,j-1}+(-1)^{j-n}\psi^{D_n}_{2n-j,n} \ \ \ \ \ \ \ {\rm for} \ n+1 \leq j \leq 2n-1, \label{eq:psiDn}
\end{align}
where $\psi^{D_n}_{0,j}$ is defined to be the following
\begin{equation} \label{eq:psi0D}
\psi^{D_n}_{0,j}=
\begin{cases}
0 \ \ \ {\rm for} \ 1 \leq j \leq n-2, \\
(-1)^n \left(\sum_{k=1}^{n} x_1 \cdots \widehat x_k \cdots x_n \partial_k \right) \ \ \ {\rm for} \ j=n-1, \\
-x_{2n-j}x_{2n-j+1}\cdots x_{n}\psi^{D_n}_{0,n-1} \ \ \ {\rm for} \ n \leq j \leq 2n-3,
\end{cases}
\end{equation}
and $\xi^{D_n}_{i}$ is defined as 
\begin{align} \label{eq:xiD}
\xi^{D_n}_{i}=&\sum_{k=1}^i \left( (x_k-x_{i+1})\cdots(x_k-x_{n-1})x_n + (-1)^{n-i} x_{i+1} \cdots x_{n} \right) x_k^{-1} \partial_k \\
&+ (-1)^{n-i} \sum_{k=i+1}^{n} x_{i+1} \cdots \widehat{x_k} \cdots x_n \partial_k, \notag
\end{align}
where the caret sign \ $\widehat{}$ \ over $x_j$ means that the entry $x_j$ is to be omitted.
Note that $\big( (x_j-x_{i+1})\cdots(x_j-x_{n-1})x_n + (-1)^{n-i} x_{i+1} \cdots x_{n-1}x_{n} \big)$ is divisible by $x_j$.
Thus, $\xi^{D_n}_{i}$ is an element of $\Der \CR=\CR \otimes \mathfrak{t}$, so is $\psi^{D_n}_{i,j}$.
Also, we define $\xi^{D_n}_{i}$ for $i=0$ as 
$$
\xi^{D_n}_{0}:=\psi^{D_n}_{0,n-1}.
$$
Note that we need the derivations $\xi^{D_n}_{i}$ only for $0 \leq i \leq n-2$.
In fact, if we take $i=n-1$ in \eqref{eq:psiDi}, then one can see that $j=n$ and $\xi^{D_n}_{n-1}$ does not appear in the right hand side of \eqref{eq:psiDi}.

\begin{example}
In type $D_4$, the positive roots $\alpha_{i,j}$ are given by
\begin{align*}
\alpha_{1,2}&=x_1-x_2, \ \alpha_{1,3}=x_1-x_3, \ \alpha_{1,4}=x_1-x_4, \ \alpha_{1,5}=x_1+x_3, \  \alpha_{1,6}=x_1+x_2, \ \alpha_{2,3}=x_2-x_3, \\ 
\alpha_{2,4}&=x_2-x_4, \ \alpha_{2,5}=x_2+x_3, \ \alpha_{3,4}=x_3-x_4, \ \alpha_{4,5}=x_3+x_4, \ \alpha_{4,6}=x_2+x_4, \ \alpha_{4,7}=x_1+x_4. 
\end{align*}
We arrange the derivations $\psi^{D_4}_{i,j}$ as shown in Figure~\ref{picture:The arrangement of derivations for type D4}.
{\tiny
\begin{figure}[h]
\begin{center}
\begin{picture}(210,50)
\put(0,00){\framebox(30,15)}
\put(0,15){\framebox(30,15)} 
\put(0,30){\framebox(30,15)}
\put(0,45){\framebox(30,15){$\psi^{D_4}_{1,1}$}}
\put(30,00){\framebox(30,15)}
\put(30,15){\framebox(30,15)}
\put(30,30){\framebox(30,15){$\psi^{D_4}_{2,2}$}}
\put(30,45){\framebox(30,15){$\psi^{D_4}_{1,2}$}} 
\put(60,00){\framebox(30,15)} 
\put(60,15){\framebox(30,15){$\psi^{D_4}_{3,3}$}} 
\put(60,30){\framebox(30,15){$\psi^{D_4}_{2,3}$}} 
\put(60,45){\framebox(30,15){$\psi^{D_4}_{1,3}$}}
\put(90,00){\framebox(30,15)} 
\put(90,15){\framebox(30,15){$\psi^{D_4}_{3,4}$}} 
\put(90,30){\framebox(30,15){$\psi^{D_4}_{2,4}$}} 
\put(90,45){\framebox(30,15){$\psi^{D_4}_{1,4}$}}
\put(120,00){\framebox(30,15){$\psi^{D_4}_{4,4}$}} 
\put(120,15){\framebox(30,15){$\psi^{D_4}_{4,5}$}} 
\put(120,30){\framebox(30,15){$\psi^{D_4}_{4,6}$}} 
\put(120,45){\framebox(30,15){$\psi^{D_4}_{4,7}$}}
\put(150,00){\framebox(30,15)} 
\put(150,15){\framebox(30,15)} 
\put(150,30){\framebox(30,15){$\psi^{D_4}_{2,5}$}} 
\put(150,45){\framebox(30,15){$\psi^{D_4}_{1,5}$}}
\put(180,0){\framebox(30,15)} 
\put(180,15){\framebox(30,15)} 
\put(180,30){\framebox(30,15)} 
\put(180,45){\framebox(30,15){$\psi^{D_4}_{1,6}$}}
\end{picture}
\end{center}
\caption{The arrangement of derivations for type $D_4$.}
\label{picture:The arrangement of derivations for type D4}
\end{figure} }

The derivations $\psi^{D_4}_{i,j} \in \Der \CR=\CR \otimes \mathfrak{t}$ are explicitly described as follows:
{\tiny
\begin{align*}
\psi^{D_4}_{1,1} =&\partial_{1},\ \  
\psi^{D_4}_{1,2} = (x_1-x_2)\partial_{1},\ \ 
\psi^{D_4}_{1,3} = \frac{\big((x_1-x_2)(x_1-x_3)(x_1+x_4)-x_2x_3x_4\big)}{x_1}\partial_1-x_3x_4\partial_2-x_2x_4\partial_3-x_2x_3\partial_4,\\
\psi^{D_4}_{1,4} =& \frac{\big((x_1-x_2)(x_1-x_3)(x_1-x_4)(x_1+x_4)+x_2x_3x_4^2\big)}{x_1}\partial_1+x_3x_4^2\partial_2+x_2x_4^2\partial_3+x_2x_3x_4\partial_4,\\
\psi^{D_4}_{1,5} =& \frac{\big((x_1-x_2)(x_1-x_3)(x_1-x_4)(x_1+x_4)(x_1+x_3)+x_2x_3^2x_4^2\big)}{x_1}\partial_1 +x_3^2x_4^2\partial_2+x_2x_3x_4^2\partial_3+x_2x_3^2x_4\partial_4,\\
\psi^{D_4}_{1,6} =& \frac{\big((x_1-x_2)(x_1-x_3)(x_1-x_4)(x_1+x_4)(x_1+x_3)(x_1+x_2)+x_2^2x_3^2x_4^2\big)}{x_1}\partial_1 +x_2x_3^2x_4^2\partial_2+x_2^2x_3x_4^2\partial_3+x_2^2x_3^2x_4\partial_4,\\
\psi^{D_4}_{2,2} =& \partial_{1}+\partial_{2},\ \ 
\psi^{D_4}_{2,3} = \frac{\big((x_1-x_3)(x_1+x_4)+x_3x_4\big)}{x_1}\partial_1+\frac{\big((x_2-x_3)(x_2+x_4)+x_3x_4\big)}{x_2}\partial_2+x_4\partial_3+x_3\partial_4,\\
\psi^{D_4}_{2,4} =& \frac{\big((x_1-x_3)(x_1-x_4)(x_1+x_4)-x_3x_4^2\big)}{x_1}\partial_1+\frac{\big((x_2-x_3)(x_2-x_4)(x_2+x_4)-x_3x_4^2\big)}{x_2}\partial_2 -x_4^2\partial_3-x_3x_4\partial_4,\\
\psi^{D_4}_{2,5} =& \frac{\big((x_1-x_3)(x_1-x_4)(x_1+x_4)(x_1+x_3)-x_3^2x_4^2\big)}{x_1}\partial_1 +\frac{\big((x_2-x_3)(x_2-x_4)(x_2+x_4)(x_2+x_3)-x_3^2x_4^2\big)}{x_2}\partial_2 -x_3x_4^2\partial_3-x_3^3x_4\partial_4,\\
\psi^{D_4}_{3,3} =& \partial_{1}+\partial_{2}+\partial_{3}-\partial_{4},\ \ 
\psi^{D_4}_{3,4} = x_1\partial_1 + x_2\partial_2+x_3\partial_3+x_4\partial_4, \ \ 
\psi^{D_4}_{4,4} = \partial_{1}+\partial_{2}+\partial_{3}+\partial_{4},\\
\psi^{D_4}_{4,5} =& -\frac{\big((x_1-x_3)(x_1-x_4)-x_3x_4 \big)}{x_1}\partial_{1}-\frac{\big((x_2-x_3)(x_2-x_4)-x_3x_4 \big)}{x_2}\partial_{2}+x_4\partial_{3}+x_3\partial_{4},\\
\psi^{D_4}_{4,6} =& \frac{\big((x_1-x_2)(x_1-x_3)(x_1-x_4)+x_2x_3x_4\big)}{x_1}\partial_1+x_3x_4\partial_2+x_2x_4\partial_3+x_2x_3\partial_4,\\
\psi^{D_4}_{4,7} =& x_2x_3x_4\partial_1+x_1x_3x_4\partial_2+x_1x_2x_4\partial_3+x_1x_2x_3\partial_4.
\end{align*}}
\end{example}

It is straightforward from the definition of $\psi^{D_n}_{i,j}$ to see the following explicit formula for $\psi^{D_n}_{i,j}$.
For $1 \leq i \leq n-1$ we obtain
\begin{align}
\psi^{D_n}_{i,j}=&\sum_{k=1}^{i} (x_k-x_{i+1})\cdots(x_k-x_j) \partial_k \ \ \ \ \ \ {\rm for} \ i \leq j \leq n-2 \ (i \neq n-1), \label{eq:psi1} \\
\psi^{D_n}_{i,n-1}=&\sum_{k=1}^{i} \left((x_k-x_{i+1})\cdots(x_k-x_{n-1})(x_k+x_n) +(-1)^{n-i} x_{i+1}\cdots x_n \right) x_k^{-1}\partial_k \label{eq:psi2} \\
&+(-1)^{n-i} \sum_{k=i+1}^n x_{i+1} \cdots \widehat{x_k} \cdots x_n \partial_k, \notag \\
\psi^{D_n}_{i,n+j}=&\sum_{k=1}^{i} \big((x_k-x_{i+1})\cdots(x_k-x_{n})(x_k+x_n)\cdots(x_k+x_{n-j}) \label{eq:psi3} \\
&+(-1)^{n-i+1} x_{i+1}\cdots x_{n-1-j} x_{n-j}^2\cdots x_n^2 \big) x_k^{-1}\partial_k \notag \\
&+(-1)^{n-i+1} x_{n-j} \cdots x_n \sum_{k=i+1}^n x_{i+1} \cdots \widehat{x_k} \cdots x_n \partial_k \ \ \ \ \ \ {\rm for} \ 0 \leq j \leq n-1-i, \notag 
\end{align}
and
\begin{align}
\psi^{D_n}_{n,2n-1-r}=&\sum_{k=1}^{r} \left((-1)^{n-r+1} (x_k-x_{r+1})\cdots(x_k-x_{n-1})(x_k-x_n) + x_{r+1}\cdots x_n \right) x_k^{-1}\partial_k \label{eq:psi4} \\
&+ \sum_{k=r+1}^n x_{r+1} \cdots \widehat{x_k} \cdots x_n \partial_k \ \ \ \ \ \ {\rm for} \ 0 \leq r \leq n-1 \notag.
\end{align}

\begin{proposition} \label{proposition:psiD}
Let $I$ be a lower ideal in $\Phi^{+}_{D_n}$ and $h_I$ the associated Hessenberg function in \eqref{eq:Hessft}.
Then, $\psi^{D_n}_{i,h_I(i)}$ is an element of $D(\A_I)$ for $i=1,\ldots,n$.
\end{proposition}

\begin{proof}
We fix $i$ and put $j=h_I(i)$.
We consider a Hessenberg function $h$ such that $h(i)=j$, and choose the maximal Hessenberg function $h^{(i,j)}$ from such Hessenberg functions.
The lower ideal associated with $h^{(i,j)}$ is denoted by $I^{(i,j)}$. 
Note that the lower ideal $I^{(i,j)}$ is the maximal lower ideal among lower ideals $I$ such that $\alpha_{i,j+1} \notin I$ with respect to inclusion.
It is enough to prove that $\psi^{D_n}_{i,j}$ belong to $D(\A_{I^{(i,j)}})$ because $D(\A_{I^{(i,j)}}) \subset D(\A_I)$. \\ \ 
\textit{Case~1} Suppose that $1 \leq i \leq n-2$ and $i \leq j \leq n-2$. Then the Hessenberg function $h^{(i,j)}$ is given by
\begin{align*}
h^{(i,j)}(t)=\begin{cases}
j \ \ \ &{\rm if} \ 1 \leq t \leq i, \\
2n-1-t \ \ \ &{\rm if} \ i+1 \leq t \leq n-1, \\
2n-1-i \ \ \ &{\rm if} \ t = n. 
\end{cases}
\end{align*}
The picture of the Hessenberg function $h^{(i,j)}$ is shown in Figure~\ref{picture:h^{(i,j)}Case1}.
Here, $d_t=t+e_t=2n-1-t$ for $1 \leq t \leq n-1$.

\begin{figure}[h]
\begin{center}
\begin{picture}(490,165)
\hspace{-20pt}

\put(0,153){\colorbox{gray}}
\put(5,153){\colorbox{gray}}
\put(10,153){\colorbox{gray}}
\put(14,153){\colorbox{gray}}
\put(0,158){\colorbox{gray}}
\put(5,158){\colorbox{gray}}
\put(10,158){\colorbox{gray}}
\put(14,158){\colorbox{gray}}
\put(0,162){\colorbox{gray}}
\put(5,162){\colorbox{gray}}
\put(10,162){\colorbox{gray}}
\put(14,162){\colorbox{gray}}

\put(40,153){\colorbox{gray}}
\put(45,153){\colorbox{gray}}
\put(50,153){\colorbox{gray}}
\put(54,153){\colorbox{gray}}
\put(40,158){\colorbox{gray}}
\put(45,158){\colorbox{gray}}
\put(50,158){\colorbox{gray}}
\put(54,158){\colorbox{gray}}
\put(40,162){\colorbox{gray}}
\put(45,162){\colorbox{gray}}
\put(50,162){\colorbox{gray}}
\put(54,162){\colorbox{gray}}

\put(80,153){\colorbox{gray}}
\put(85,153){\colorbox{gray}}
\put(90,153){\colorbox{gray}}
\put(94,153){\colorbox{gray}}
\put(80,158){\colorbox{gray}}
\put(85,158){\colorbox{gray}}
\put(90,158){\colorbox{gray}}
\put(94,158){\colorbox{gray}}
\put(80,162){\colorbox{gray}}
\put(85,162){\colorbox{gray}}
\put(90,162){\colorbox{gray}}
\put(94,162){\colorbox{gray}}

\put(40,118){\colorbox{gray}}
\put(45,118){\colorbox{gray}}
\put(50,118){\colorbox{gray}}
\put(54,118){\colorbox{gray}}
\put(40,123){\colorbox{gray}}
\put(45,123){\colorbox{gray}}
\put(50,123){\colorbox{gray}}
\put(54,123){\colorbox{gray}}
\put(40,127){\colorbox{gray}}
\put(45,127){\colorbox{gray}}
\put(50,127){\colorbox{gray}}
\put(54,127){\colorbox{gray}}

\put(80,118){\colorbox{gray}}
\put(85,118){\colorbox{gray}}
\put(90,118){\colorbox{gray}}
\put(94,118){\colorbox{gray}}
\put(80,123){\colorbox{gray}}
\put(85,123){\colorbox{gray}}
\put(90,123){\colorbox{gray}}
\put(94,123){\colorbox{gray}}
\put(80,127){\colorbox{gray}}
\put(85,127){\colorbox{gray}}
\put(90,127){\colorbox{gray}}
\put(94,127){\colorbox{gray}}

\put(100,103){\colorbox{gray}}
\put(105,103){\colorbox{gray}}
\put(110,103){\colorbox{gray}}
\put(115,103){\colorbox{gray}}
\put(120,103){\colorbox{gray}}
\put(125,103){\colorbox{gray}}
\put(130,103){\colorbox{gray}}
\put(135,103){\colorbox{gray}}
\put(140,103){\colorbox{gray}}
\put(144,103){\colorbox{gray}}
\put(100,108){\colorbox{gray}}
\put(105,108){\colorbox{gray}}
\put(110,108){\colorbox{gray}}
\put(115,108){\colorbox{gray}}
\put(120,108){\colorbox{gray}}
\put(125,108){\colorbox{gray}}
\put(130,108){\colorbox{gray}}
\put(135,108){\colorbox{gray}}
\put(140,108){\colorbox{gray}}
\put(144,108){\colorbox{gray}}
\put(100,112){\colorbox{gray}}
\put(105,112){\colorbox{gray}}
\put(110,112){\colorbox{gray}}
\put(115,112){\colorbox{gray}}
\put(120,112){\colorbox{gray}}
\put(125,112){\colorbox{gray}}
\put(130,112){\colorbox{gray}}
\put(135,112){\colorbox{gray}}
\put(140,112){\colorbox{gray}}
\put(144,112){\colorbox{gray}}

\put(170,103){\colorbox{gray}}
\put(175,103){\colorbox{gray}}
\put(180,103){\colorbox{gray}}
\put(185,103){\colorbox{gray}}
\put(190,103){\colorbox{gray}}
\put(195,103){\colorbox{gray}}
\put(200,103){\colorbox{gray}}
\put(205,103){\colorbox{gray}}
\put(210,103){\colorbox{gray}}
\put(214,103){\colorbox{gray}}
\put(170,108){\colorbox{gray}}
\put(175,108){\colorbox{gray}}
\put(180,108){\colorbox{gray}}
\put(185,108){\colorbox{gray}}
\put(190,108){\colorbox{gray}}
\put(195,108){\colorbox{gray}}
\put(200,108){\colorbox{gray}}
\put(205,108){\colorbox{gray}}
\put(210,108){\colorbox{gray}}
\put(214,108){\colorbox{gray}}
\put(170,112){\colorbox{gray}}
\put(175,112){\colorbox{gray}}
\put(180,112){\colorbox{gray}}
\put(185,112){\colorbox{gray}}
\put(190,112){\colorbox{gray}}
\put(195,112){\colorbox{gray}}
\put(200,112){\colorbox{gray}}
\put(205,112){\colorbox{gray}}
\put(210,112){\colorbox{gray}}
\put(214,112){\colorbox{gray}}

\put(220,103){\colorbox{gray}}
\put(225,103){\colorbox{gray}}
\put(230,103){\colorbox{gray}}
\put(235,103){\colorbox{gray}}
\put(240,103){\colorbox{gray}}
\put(245,103){\colorbox{gray}}
\put(249,103){\colorbox{gray}}
\put(220,108){\colorbox{gray}}
\put(225,108){\colorbox{gray}}
\put(230,108){\colorbox{gray}}
\put(235,108){\colorbox{gray}}
\put(240,108){\colorbox{gray}}
\put(245,108){\colorbox{gray}}
\put(249,108){\colorbox{gray}}
\put(220,112){\colorbox{gray}}
\put(225,112){\colorbox{gray}}
\put(230,112){\colorbox{gray}}
\put(235,112){\colorbox{gray}}
\put(240,112){\colorbox{gray}}
\put(245,112){\colorbox{gray}}
\put(249,112){\colorbox{gray}}

\put(255,103){\colorbox{gray}}
\put(260,103){\colorbox{gray}}
\put(265,103){\colorbox{gray}}
\put(270,103){\colorbox{gray}}
\put(275,103){\colorbox{gray}}
\put(280,103){\colorbox{gray}}
\put(285,103){\colorbox{gray}}
\put(290,103){\colorbox{gray}}
\put(295,103){\colorbox{gray}}
\put(300,103){\colorbox{gray}}
\put(304,103){\colorbox{gray}}
\put(255,108){\colorbox{gray}}
\put(260,108){\colorbox{gray}}
\put(265,108){\colorbox{gray}}
\put(270,108){\colorbox{gray}}
\put(275,108){\colorbox{gray}}
\put(280,108){\colorbox{gray}}
\put(285,108){\colorbox{gray}}
\put(290,108){\colorbox{gray}}
\put(295,108){\colorbox{gray}}
\put(300,108){\colorbox{gray}}
\put(304,108){\colorbox{gray}}
\put(255,112){\colorbox{gray}}
\put(260,112){\colorbox{gray}}
\put(265,112){\colorbox{gray}}
\put(270,112){\colorbox{gray}}
\put(275,112){\colorbox{gray}}
\put(280,112){\colorbox{gray}}
\put(285,112){\colorbox{gray}}
\put(290,112){\colorbox{gray}}
\put(295,112){\colorbox{gray}}
\put(300,112){\colorbox{gray}}
\put(304,112){\colorbox{gray}}

\put(385,103){\colorbox{gray}}
\put(390,103){\colorbox{gray}}
\put(395,103){\colorbox{gray}}
\put(400,103){\colorbox{gray}}
\put(405,103){\colorbox{gray}}
\put(410,103){\colorbox{gray}}
\put(415,103){\colorbox{gray}}
\put(420,103){\colorbox{gray}}
\put(424,103){\colorbox{gray}}
\put(385,108){\colorbox{gray}}
\put(390,108){\colorbox{gray}}
\put(395,108){\colorbox{gray}}
\put(400,108){\colorbox{gray}}
\put(405,108){\colorbox{gray}}
\put(410,108){\colorbox{gray}}
\put(415,108){\colorbox{gray}}
\put(420,108){\colorbox{gray}}
\put(424,108){\colorbox{gray}}
\put(385,112){\colorbox{gray}}
\put(390,112){\colorbox{gray}}
\put(395,112){\colorbox{gray}}
\put(400,112){\colorbox{gray}}
\put(405,112){\colorbox{gray}}
\put(410,112){\colorbox{gray}}
\put(415,112){\colorbox{gray}}
\put(420,112){\colorbox{gray}}
\put(424,112){\colorbox{gray}}

\put(100,68){\colorbox{gray}}
\put(105,68){\colorbox{gray}}
\put(110,68){\colorbox{gray}}
\put(115,68){\colorbox{gray}}
\put(120,68){\colorbox{gray}}
\put(125,68){\colorbox{gray}}
\put(130,68){\colorbox{gray}}
\put(135,68){\colorbox{gray}}
\put(140,68){\colorbox{gray}}
\put(144,68){\colorbox{gray}}
\put(100,73){\colorbox{gray}}
\put(105,73){\colorbox{gray}}
\put(110,73){\colorbox{gray}}
\put(115,73){\colorbox{gray}}
\put(120,73){\colorbox{gray}}
\put(125,73){\colorbox{gray}}
\put(130,73){\colorbox{gray}}
\put(135,73){\colorbox{gray}}
\put(140,73){\colorbox{gray}}
\put(144,73){\colorbox{gray}}
\put(100,77){\colorbox{gray}}
\put(105,77){\colorbox{gray}}
\put(110,77){\colorbox{gray}}
\put(115,77){\colorbox{gray}}
\put(120,77){\colorbox{gray}}
\put(125,77){\colorbox{gray}}
\put(130,77){\colorbox{gray}}
\put(135,77){\colorbox{gray}}
\put(140,77){\colorbox{gray}}
\put(144,77){\colorbox{gray}}

\put(170,68){\colorbox{gray}}
\put(175,68){\colorbox{gray}}
\put(180,68){\colorbox{gray}}
\put(185,68){\colorbox{gray}}
\put(190,68){\colorbox{gray}}
\put(195,68){\colorbox{gray}}
\put(200,68){\colorbox{gray}}
\put(205,68){\colorbox{gray}}
\put(210,68){\colorbox{gray}}
\put(214,68){\colorbox{gray}}
\put(170,73){\colorbox{gray}}
\put(175,73){\colorbox{gray}}
\put(180,73){\colorbox{gray}}
\put(185,73){\colorbox{gray}}
\put(190,73){\colorbox{gray}}
\put(195,73){\colorbox{gray}}
\put(200,73){\colorbox{gray}}
\put(205,73){\colorbox{gray}}
\put(210,73){\colorbox{gray}}
\put(214,73){\colorbox{gray}}
\put(170,77){\colorbox{gray}}
\put(175,77){\colorbox{gray}}
\put(180,77){\colorbox{gray}}
\put(185,77){\colorbox{gray}}
\put(190,77){\colorbox{gray}}
\put(195,77){\colorbox{gray}}
\put(200,77){\colorbox{gray}}
\put(205,77){\colorbox{gray}}
\put(210,77){\colorbox{gray}}
\put(214,77){\colorbox{gray}}

\put(220,68){\colorbox{gray}}
\put(225,68){\colorbox{gray}}
\put(230,68){\colorbox{gray}}
\put(235,68){\colorbox{gray}}
\put(240,68){\colorbox{gray}}
\put(245,68){\colorbox{gray}}
\put(249,68){\colorbox{gray}}
\put(220,73){\colorbox{gray}}
\put(225,73){\colorbox{gray}}
\put(230,73){\colorbox{gray}}
\put(235,73){\colorbox{gray}}
\put(240,73){\colorbox{gray}}
\put(245,73){\colorbox{gray}}
\put(249,73){\colorbox{gray}}
\put(220,77){\colorbox{gray}}
\put(225,77){\colorbox{gray}}
\put(230,77){\colorbox{gray}}
\put(235,77){\colorbox{gray}}
\put(240,77){\colorbox{gray}}
\put(245,77){\colorbox{gray}}
\put(249,77){\colorbox{gray}}

\put(255,68){\colorbox{gray}}
\put(260,68){\colorbox{gray}}
\put(265,68){\colorbox{gray}}
\put(270,68){\colorbox{gray}}
\put(275,68){\colorbox{gray}}
\put(280,68){\colorbox{gray}}
\put(285,68){\colorbox{gray}}
\put(290,68){\colorbox{gray}}
\put(295,68){\colorbox{gray}}
\put(300,68){\colorbox{gray}}
\put(304,68){\colorbox{gray}}
\put(255,73){\colorbox{gray}}
\put(260,73){\colorbox{gray}}
\put(265,73){\colorbox{gray}}
\put(270,73){\colorbox{gray}}
\put(275,73){\colorbox{gray}}
\put(280,73){\colorbox{gray}}
\put(285,73){\colorbox{gray}}
\put(290,73){\colorbox{gray}}
\put(295,73){\colorbox{gray}}
\put(300,73){\colorbox{gray}}
\put(304,73){\colorbox{gray}}
\put(255,77){\colorbox{gray}}
\put(260,77){\colorbox{gray}}
\put(265,77){\colorbox{gray}}
\put(270,77){\colorbox{gray}}
\put(275,77){\colorbox{gray}}
\put(280,77){\colorbox{gray}}
\put(285,77){\colorbox{gray}}
\put(290,77){\colorbox{gray}}
\put(295,77){\colorbox{gray}}
\put(300,77){\colorbox{gray}}
\put(304,77){\colorbox{gray}}

\put(330,68){\colorbox{gray}}
\put(335,68){\colorbox{gray}}
\put(340,68){\colorbox{gray}}
\put(345,68){\colorbox{gray}}
\put(350,68){\colorbox{gray}}
\put(355,68){\colorbox{gray}}
\put(360,68){\colorbox{gray}}
\put(365,68){\colorbox{gray}}
\put(369,68){\colorbox{gray}}
\put(330,73){\colorbox{gray}}
\put(335,73){\colorbox{gray}}
\put(340,73){\colorbox{gray}}
\put(345,73){\colorbox{gray}}
\put(350,73){\colorbox{gray}}
\put(355,73){\colorbox{gray}}
\put(360,73){\colorbox{gray}}
\put(365,73){\colorbox{gray}}
\put(369,73){\colorbox{gray}}
\put(330,77){\colorbox{gray}}
\put(335,77){\colorbox{gray}}
\put(340,77){\colorbox{gray}}
\put(345,77){\colorbox{gray}}
\put(350,77){\colorbox{gray}}
\put(355,77){\colorbox{gray}}
\put(360,77){\colorbox{gray}}
\put(365,77){\colorbox{gray}}
\put(369,77){\colorbox{gray}}

\put(170,33){\colorbox{gray}}
\put(175,33){\colorbox{gray}}
\put(180,33){\colorbox{gray}}
\put(185,33){\colorbox{gray}}
\put(190,33){\colorbox{gray}}
\put(195,33){\colorbox{gray}}
\put(200,33){\colorbox{gray}}
\put(205,33){\colorbox{gray}}
\put(210,33){\colorbox{gray}}
\put(214,33){\colorbox{gray}}
\put(170,38){\colorbox{gray}}
\put(175,38){\colorbox{gray}}
\put(180,38){\colorbox{gray}}
\put(185,38){\colorbox{gray}}
\put(190,38){\colorbox{gray}}
\put(195,38){\colorbox{gray}}
\put(200,38){\colorbox{gray}}
\put(205,38){\colorbox{gray}}
\put(210,38){\colorbox{gray}}
\put(214,38){\colorbox{gray}}
\put(170,42){\colorbox{gray}}
\put(175,42){\colorbox{gray}}
\put(180,42){\colorbox{gray}}
\put(185,42){\colorbox{gray}}
\put(190,42){\colorbox{gray}}
\put(195,42){\colorbox{gray}}
\put(200,42){\colorbox{gray}}
\put(205,42){\colorbox{gray}}
\put(210,42){\colorbox{gray}}
\put(214,42){\colorbox{gray}}

\put(220,33){\colorbox{gray}}
\put(225,33){\colorbox{gray}}
\put(230,33){\colorbox{gray}}
\put(235,33){\colorbox{gray}}
\put(240,33){\colorbox{gray}}
\put(245,33){\colorbox{gray}}
\put(249,33){\colorbox{gray}}
\put(220,38){\colorbox{gray}}
\put(225,38){\colorbox{gray}}
\put(230,38){\colorbox{gray}}
\put(235,38){\colorbox{gray}}
\put(240,38){\colorbox{gray}}
\put(245,38){\colorbox{gray}}
\put(249,38){\colorbox{gray}}
\put(220,42){\colorbox{gray}}
\put(225,42){\colorbox{gray}}
\put(230,42){\colorbox{gray}}
\put(235,42){\colorbox{gray}}
\put(240,42){\colorbox{gray}}
\put(245,42){\colorbox{gray}}
\put(249,42){\colorbox{gray}}

\put(255,33){\colorbox{gray}}
\put(260,33){\colorbox{gray}}
\put(265,33){\colorbox{gray}}
\put(270,33){\colorbox{gray}}
\put(275,33){\colorbox{gray}}
\put(280,33){\colorbox{gray}}
\put(285,33){\colorbox{gray}}
\put(290,33){\colorbox{gray}}
\put(295,33){\colorbox{gray}}
\put(300,33){\colorbox{gray}}
\put(304,33){\colorbox{gray}}
\put(255,38){\colorbox{gray}}
\put(260,38){\colorbox{gray}}
\put(265,38){\colorbox{gray}}
\put(270,38){\colorbox{gray}}
\put(275,38){\colorbox{gray}}
\put(280,38){\colorbox{gray}}
\put(285,38){\colorbox{gray}}
\put(290,38){\colorbox{gray}}
\put(295,38){\colorbox{gray}}
\put(300,38){\colorbox{gray}}
\put(304,38){\colorbox{gray}}
\put(255,42){\colorbox{gray}}
\put(260,42){\colorbox{gray}}
\put(265,42){\colorbox{gray}}
\put(270,42){\colorbox{gray}}
\put(275,42){\colorbox{gray}}
\put(280,42){\colorbox{gray}}
\put(285,42){\colorbox{gray}}
\put(290,42){\colorbox{gray}}
\put(295,42){\colorbox{gray}}
\put(300,42){\colorbox{gray}}
\put(304,42){\colorbox{gray}}

\put(255,18){\colorbox{gray}}
\put(260,18){\colorbox{gray}}
\put(265,18){\colorbox{gray}}
\put(270,18){\colorbox{gray}}
\put(275,18){\colorbox{gray}}
\put(280,18){\colorbox{gray}}
\put(285,18){\colorbox{gray}}
\put(290,18){\colorbox{gray}}
\put(295,18){\colorbox{gray}}
\put(300,18){\colorbox{gray}}
\put(304,18){\colorbox{gray}}
\put(255,23){\colorbox{gray}}
\put(260,23){\colorbox{gray}}
\put(265,23){\colorbox{gray}}
\put(270,23){\colorbox{gray}}
\put(275,23){\colorbox{gray}}
\put(280,23){\colorbox{gray}}
\put(285,23){\colorbox{gray}}
\put(290,23){\colorbox{gray}}
\put(295,23){\colorbox{gray}}
\put(300,23){\colorbox{gray}}
\put(304,23){\colorbox{gray}}
\put(255,27){\colorbox{gray}}
\put(260,27){\colorbox{gray}}
\put(265,27){\colorbox{gray}}
\put(270,27){\colorbox{gray}}
\put(275,27){\colorbox{gray}}
\put(280,27){\colorbox{gray}}
\put(285,27){\colorbox{gray}}
\put(290,27){\colorbox{gray}}
\put(295,27){\colorbox{gray}}
\put(300,27){\colorbox{gray}}
\put(304,27){\colorbox{gray}}

\put(0,150){\framebox(20,15){\tiny $(1,1)$}} 

\put(23,155){$\cdots$}
\put(23,135){$\ddots$}

\put(40,150){\framebox(20,15){\tiny $(1,i)$}}
\put(48,135){$\vdots$}
\put(40,115){\framebox(20,15){\tiny $(i,i)$}}

\put(63,155){$\cdots$}
\put(63,119){$\cdots$}
\put(63,94){$\ddots$}

\put(80,150){\framebox(20,15){\tiny $(1,j)$}}
\put(88,135){$\vdots$}
\put(80,115){\framebox(20,15){\tiny $(i,j)$}}
\put(88,94){$\vdots$}

\put(100,150){\framebox(50,15){\tiny $(1,j+1)$}}
\put(123,135){$\vdots$}
\put(100,115){\framebox(50,15){\tiny $(i,j+1)$}}
\put(100,100){\framebox(50,15){\tiny $(i+1,j+1)$}}
\put(123,85){$\vdots$}
\put(100,65){\framebox(50,15){\tiny $(j+1,j+1)$}}

\put(153,155){$\cdots$}
\put(153,120){$\cdots$}
\put(153,104){$\cdots$}
\put(153,68){$\cdots$}
\put(153,50){$\ddots$}

\put(170,150){\framebox(50,15){\tiny $(1,n-1)$}}
\put(194,135){$\vdots$}
\put(170,115){\framebox(50,15){\tiny $(i,n-1)$}}
\put(170,100){\framebox(50,15){\tiny $(i+1,n-1)$}}
\put(194,85){$\vdots$}
\put(170,65){\framebox(50,15){\tiny $(j+1,n-1)$}}
\put(194,50){$\vdots$}
\put(170,30){\framebox(50,15){\tiny $(n-1,n-1)$}}

\put(220,150){\framebox(35,15){\tiny $(1,n)$}}
\put(236,135){$\vdots$}
\put(220,115){\framebox(35,15){\tiny $(i,n)$}}
\put(220,100){\framebox(35,15){\tiny $(i+1,n)$}}
\put(236,85){$\vdots$}
\put(220,65){\framebox(35,15){\tiny $(j+1,n)$}}
\put(236,50){$\vdots$}
\put(220,30){\framebox(35,15){\tiny $(n-1,n)$}}

\put(255,150){\framebox(55,15){\tiny $(n,2n-1)$}}
\put(279,135){$\vdots$}
\put(255,115){\framebox(55,15){\tiny $(n,2n-i)$}}
\put(255,100){\framebox(55,15){\tiny $(n,2n-1-i)$}}
\put(279,85){$\vdots$}
\put(255,65){\framebox(55,15){\tiny $(n,2n-1-j)$}}
\put(279,50){$\vdots$}
\put(255,30){\framebox(55,15){\tiny $(n,n+1)$}}
\put(255,15){\framebox(55,15){\tiny $(n,n)$}}

\put(383,155){$\cdots$}
\put(363,120){$\cdots$}
\put(343,104){$\cdots$}
\put(313,68){$\cdots$}
\put(313,50){$\cdot$} 
\put(318,52.5){$\cdot$}
\put(323,55){$\cdot$}

\put(330,65){\framebox(45,15){\tiny $(j+1,d_{j+1})$}}

\put(376,85){$\cdot$} 
\put(378,87.5){$\cdot$}
\put(380,90){$\cdot$}

\put(385,100){\framebox(45,15){\tiny $(i+1,d_{i+1})$}}

\put(430,115){\framebox(25,15){\tiny $(i,d_i)$}}

\put(456,135){$\cdot$} 
\put(458,137.5){$\cdot$}
\put(460,140){$\cdot$}

\put(465,150){\framebox(25,15){\tiny $(1,d_1)$}}
\end{picture}
\end{center}
\vspace{-20pt}
\caption{The Hessenberg function $h^{(i,j)}$ for $1 \leq i \leq n-2$ and $i \leq j \leq n-2$.}
\label{picture:h^{(i,j)}Case1}
\end{figure} 

Let $\alpha \in I^{(i,j)}$. 
Since all shaded boxes except for the boxes $(r,r)$ with $1 \leq r \leq n$ in Figure~\ref{picture:h^{(i,j)}Case1} correspond to coordinates of all positive roots in $I^{(i,j)}$, $\alpha$ is one of the following forms
\begin{align}
x_k-x_\ell \ \ \ &(1 \leq k < \ell \leq i), \label{eq:Case1(1)} \\
x_k-x_\ell \ \ \ &(1 \leq k \leq i < \ell \leq j), \label{eq:Case1(2)} \\
x_k \pm x_\ell \ \ \ &(i < k < \ell \leq n). \label{eq:Case1(3)} 
\end{align}
From the formula $\eqref{eq:psi1}$, we have
\begin{align*} 
\psi^{D_n}_{i,j} (\alpha)= \begin{cases}
(x_k-x_{i+1})\cdots(x_k-x_j) - (x_\ell-x_{i+1})\cdots(x_\ell-x_j) &{\rm if} \ \alpha \ {\rm is \ of \ the \ form} \ \eqref{eq:Case1(1)}, \\
(x_k-x_{i+1})\cdots(x_k-x_j) &{\rm if} \ \alpha \ {\rm is \ of \ the \ form} \ \eqref{eq:Case1(2)}, \\
0  &{\rm if} \ \alpha \ {\rm is \ of \ the \ form} \ \eqref{eq:Case1(3)}. 
\end{cases}
\end{align*}
One can see that $\psi^{D_n}_{i,j} (\alpha) \equiv 0$ mod $\alpha$
in both of cases. \\ \ 
\textit{Case~2} Suppose that $1 \leq i \leq n-1$ and $j = n-1$. In this case, the Hessenberg function $h^{(i,j)}$ is defined by
\begin{align*}
h^{(i,j)}(t)=\begin{cases}
n-1 \ \ \ &{\rm if} \ 1 \leq t \leq i, \\
2n-1-t \ \ \ &{\rm if} \ i+1 \leq t \leq n-1, \\
2n-1 \ \ \ &{\rm if} \ t = n. 
\end{cases}
\end{align*}
The picture of the Hessenberg function $h^{(i,j)}$ is shown in Figure~\ref{picture:h^{(i,j)}Case2}.

\begin{figure}[h]
\begin{center}
\begin{picture}(405,165)
\put(0,153){\colorbox{gray}}
\put(5,153){\colorbox{gray}}
\put(10,153){\colorbox{gray}}
\put(14,153){\colorbox{gray}}
\put(0,158){\colorbox{gray}}
\put(5,158){\colorbox{gray}}
\put(10,158){\colorbox{gray}}
\put(14,158){\colorbox{gray}}
\put(0,162){\colorbox{gray}}
\put(5,162){\colorbox{gray}}
\put(10,162){\colorbox{gray}}
\put(14,162){\colorbox{gray}}

\put(40,153){\colorbox{gray}}
\put(45,153){\colorbox{gray}}
\put(50,153){\colorbox{gray}}
\put(54,153){\colorbox{gray}}
\put(40,158){\colorbox{gray}}
\put(45,158){\colorbox{gray}}
\put(50,158){\colorbox{gray}}
\put(54,158){\colorbox{gray}}
\put(40,162){\colorbox{gray}}
\put(45,162){\colorbox{gray}}
\put(50,162){\colorbox{gray}}
\put(54,162){\colorbox{gray}}

\put(80,153){\colorbox{gray}}
\put(85,153){\colorbox{gray}}
\put(90,153){\colorbox{gray}}
\put(95,153){\colorbox{gray}}
\put(100,153){\colorbox{gray}}
\put(105,153){\colorbox{gray}}
\put(110,153){\colorbox{gray}}
\put(115,153){\colorbox{gray}}
\put(120,153){\colorbox{gray}}
\put(124,153){\colorbox{gray}}
\put(80,158){\colorbox{gray}}
\put(85,158){\colorbox{gray}}
\put(90,158){\colorbox{gray}}
\put(95,158){\colorbox{gray}}
\put(100,158){\colorbox{gray}}
\put(105,158){\colorbox{gray}}
\put(110,158){\colorbox{gray}}
\put(115,158){\colorbox{gray}}
\put(120,158){\colorbox{gray}}
\put(124,158){\colorbox{gray}}
\put(80,162){\colorbox{gray}}
\put(85,162){\colorbox{gray}}
\put(90,162){\colorbox{gray}}
\put(95,162){\colorbox{gray}}
\put(100,162){\colorbox{gray}}
\put(105,162){\colorbox{gray}}
\put(110,162){\colorbox{gray}}
\put(115,162){\colorbox{gray}}
\put(120,162){\colorbox{gray}}
\put(124,162){\colorbox{gray}}

\put(165,153){\colorbox{gray}}
\put(170,153){\colorbox{gray}}
\put(175,153){\colorbox{gray}}
\put(180,153){\colorbox{gray}}
\put(185,153){\colorbox{gray}}
\put(190,153){\colorbox{gray}}
\put(195,153){\colorbox{gray}}
\put(200,153){\colorbox{gray}}
\put(205,153){\colorbox{gray}}
\put(210,153){\colorbox{gray}}
\put(214,153){\colorbox{gray}}
\put(165,158){\colorbox{gray}}
\put(170,158){\colorbox{gray}}
\put(175,158){\colorbox{gray}}
\put(180,158){\colorbox{gray}}
\put(185,158){\colorbox{gray}}
\put(190,158){\colorbox{gray}}
\put(195,158){\colorbox{gray}}
\put(200,158){\colorbox{gray}}
\put(205,158){\colorbox{gray}}
\put(210,158){\colorbox{gray}}
\put(214,158){\colorbox{gray}}
\put(165,162){\colorbox{gray}}
\put(170,162){\colorbox{gray}}
\put(175,162){\colorbox{gray}}
\put(180,162){\colorbox{gray}}
\put(185,162){\colorbox{gray}}
\put(190,162){\colorbox{gray}}
\put(195,162){\colorbox{gray}}
\put(200,162){\colorbox{gray}}
\put(205,162){\colorbox{gray}}
\put(210,162){\colorbox{gray}}
\put(214,162){\colorbox{gray}}

\put(40,118){\colorbox{gray}}
\put(45,118){\colorbox{gray}}
\put(50,118){\colorbox{gray}}
\put(54,118){\colorbox{gray}}
\put(40,123){\colorbox{gray}}
\put(45,123){\colorbox{gray}}
\put(50,123){\colorbox{gray}}
\put(54,123){\colorbox{gray}}
\put(40,127){\colorbox{gray}}
\put(45,127){\colorbox{gray}}
\put(50,127){\colorbox{gray}}
\put(54,127){\colorbox{gray}}

\put(80,118){\colorbox{gray}}
\put(85,118){\colorbox{gray}}
\put(90,118){\colorbox{gray}}
\put(95,118){\colorbox{gray}}
\put(100,118){\colorbox{gray}}
\put(105,118){\colorbox{gray}}
\put(110,118){\colorbox{gray}}
\put(115,118){\colorbox{gray}}
\put(120,118){\colorbox{gray}}
\put(124,118){\colorbox{gray}}
\put(80,123){\colorbox{gray}}
\put(85,123){\colorbox{gray}}
\put(90,123){\colorbox{gray}}
\put(95,123){\colorbox{gray}}
\put(100,123){\colorbox{gray}}
\put(105,123){\colorbox{gray}}
\put(110,123){\colorbox{gray}}
\put(115,123){\colorbox{gray}}
\put(120,123){\colorbox{gray}}
\put(124,123){\colorbox{gray}}
\put(80,127){\colorbox{gray}}
\put(85,127){\colorbox{gray}}
\put(90,127){\colorbox{gray}}
\put(95,127){\colorbox{gray}}
\put(100,127){\colorbox{gray}}
\put(105,127){\colorbox{gray}}
\put(110,127){\colorbox{gray}}
\put(115,127){\colorbox{gray}}
\put(120,127){\colorbox{gray}}
\put(124,127){\colorbox{gray}}

\put(165,118){\colorbox{gray}}
\put(170,118){\colorbox{gray}}
\put(175,118){\colorbox{gray}}
\put(180,118){\colorbox{gray}}
\put(185,118){\colorbox{gray}}
\put(190,118){\colorbox{gray}}
\put(195,118){\colorbox{gray}}
\put(200,118){\colorbox{gray}}
\put(205,118){\colorbox{gray}}
\put(210,118){\colorbox{gray}}
\put(214,118){\colorbox{gray}}
\put(165,123){\colorbox{gray}}
\put(170,123){\colorbox{gray}}
\put(175,123){\colorbox{gray}}
\put(180,123){\colorbox{gray}}
\put(185,123){\colorbox{gray}}
\put(190,123){\colorbox{gray}}
\put(195,123){\colorbox{gray}}
\put(200,123){\colorbox{gray}}
\put(205,123){\colorbox{gray}}
\put(210,123){\colorbox{gray}}
\put(214,123){\colorbox{gray}}
\put(165,127){\colorbox{gray}}
\put(170,127){\colorbox{gray}}
\put(175,127){\colorbox{gray}}
\put(180,127){\colorbox{gray}}
\put(185,127){\colorbox{gray}}
\put(190,127){\colorbox{gray}}
\put(195,127){\colorbox{gray}}
\put(200,127){\colorbox{gray}}
\put(205,127){\colorbox{gray}}
\put(210,127){\colorbox{gray}}
\put(214,127){\colorbox{gray}}

\put(80,103){\colorbox{gray}}
\put(85,103){\colorbox{gray}}
\put(90,103){\colorbox{gray}}
\put(95,103){\colorbox{gray}}
\put(100,103){\colorbox{gray}}
\put(105,103){\colorbox{gray}}
\put(110,103){\colorbox{gray}}
\put(115,103){\colorbox{gray}}
\put(120,103){\colorbox{gray}}
\put(124,103){\colorbox{gray}}
\put(80,108){\colorbox{gray}}
\put(85,108){\colorbox{gray}}
\put(90,108){\colorbox{gray}}
\put(95,108){\colorbox{gray}}
\put(100,108){\colorbox{gray}}
\put(105,108){\colorbox{gray}}
\put(110,108){\colorbox{gray}}
\put(115,108){\colorbox{gray}}
\put(120,108){\colorbox{gray}}
\put(124,108){\colorbox{gray}}
\put(80,112){\colorbox{gray}}
\put(85,112){\colorbox{gray}}
\put(90,112){\colorbox{gray}}
\put(95,112){\colorbox{gray}}
\put(100,112){\colorbox{gray}}
\put(105,112){\colorbox{gray}}
\put(110,112){\colorbox{gray}}
\put(115,112){\colorbox{gray}}
\put(120,112){\colorbox{gray}}
\put(124,112){\colorbox{gray}}

\put(130,103){\colorbox{gray}}
\put(135,103){\colorbox{gray}}
\put(140,103){\colorbox{gray}}
\put(145,103){\colorbox{gray}}
\put(150,103){\colorbox{gray}}
\put(155,103){\colorbox{gray}}
\put(159,103){\colorbox{gray}}
\put(130,108){\colorbox{gray}}
\put(135,108){\colorbox{gray}}
\put(140,108){\colorbox{gray}}
\put(145,108){\colorbox{gray}}
\put(150,108){\colorbox{gray}}
\put(155,108){\colorbox{gray}}
\put(159,108){\colorbox{gray}}
\put(130,112){\colorbox{gray}}
\put(135,112){\colorbox{gray}}
\put(140,112){\colorbox{gray}}
\put(145,112){\colorbox{gray}}
\put(150,112){\colorbox{gray}}
\put(155,112){\colorbox{gray}}
\put(159,112){\colorbox{gray}}

\put(165,103){\colorbox{gray}}
\put(170,103){\colorbox{gray}}
\put(175,103){\colorbox{gray}}
\put(180,103){\colorbox{gray}}
\put(185,103){\colorbox{gray}}
\put(190,103){\colorbox{gray}}
\put(195,103){\colorbox{gray}}
\put(200,103){\colorbox{gray}}
\put(205,103){\colorbox{gray}}
\put(210,103){\colorbox{gray}}
\put(214,103){\colorbox{gray}}
\put(165,108){\colorbox{gray}}
\put(170,108){\colorbox{gray}}
\put(175,108){\colorbox{gray}}
\put(180,108){\colorbox{gray}}
\put(185,108){\colorbox{gray}}
\put(190,108){\colorbox{gray}}
\put(195,108){\colorbox{gray}}
\put(200,108){\colorbox{gray}}
\put(205,108){\colorbox{gray}}
\put(210,108){\colorbox{gray}}
\put(214,108){\colorbox{gray}}
\put(165,112){\colorbox{gray}}
\put(170,112){\colorbox{gray}}
\put(175,112){\colorbox{gray}}
\put(180,112){\colorbox{gray}}
\put(185,112){\colorbox{gray}}
\put(190,112){\colorbox{gray}}
\put(195,112){\colorbox{gray}}
\put(200,112){\colorbox{gray}}
\put(205,112){\colorbox{gray}}
\put(210,112){\colorbox{gray}}
\put(214,112){\colorbox{gray}}

\put(220,103){\colorbox{gray}}
\put(225,103){\colorbox{gray}}
\put(230,103){\colorbox{gray}}
\put(235,103){\colorbox{gray}}
\put(240,103){\colorbox{gray}}
\put(245,103){\colorbox{gray}}
\put(250,103){\colorbox{gray}}
\put(255,103){\colorbox{gray}}
\put(260,103){\colorbox{gray}}
\put(264,103){\colorbox{gray}}
\put(220,108){\colorbox{gray}}
\put(225,108){\colorbox{gray}}
\put(230,108){\colorbox{gray}}
\put(235,108){\colorbox{gray}}
\put(240,108){\colorbox{gray}}
\put(245,108){\colorbox{gray}}
\put(250,108){\colorbox{gray}}
\put(255,108){\colorbox{gray}}
\put(260,108){\colorbox{gray}}
\put(264,108){\colorbox{gray}}
\put(220,112){\colorbox{gray}}
\put(225,112){\colorbox{gray}}
\put(230,112){\colorbox{gray}}
\put(235,112){\colorbox{gray}}
\put(240,112){\colorbox{gray}}
\put(245,112){\colorbox{gray}}
\put(250,112){\colorbox{gray}}
\put(255,112){\colorbox{gray}}
\put(260,112){\colorbox{gray}}
\put(264,112){\colorbox{gray}}

\put(290,103){\colorbox{gray}}
\put(295,103){\colorbox{gray}}
\put(300,103){\colorbox{gray}}
\put(305,103){\colorbox{gray}}
\put(310,103){\colorbox{gray}}
\put(315,103){\colorbox{gray}}
\put(320,103){\colorbox{gray}}
\put(325,103){\colorbox{gray}}
\put(329,103){\colorbox{gray}}
\put(290,108){\colorbox{gray}}
\put(295,108){\colorbox{gray}}
\put(300,108){\colorbox{gray}}
\put(305,108){\colorbox{gray}}
\put(310,108){\colorbox{gray}}
\put(315,108){\colorbox{gray}}
\put(320,108){\colorbox{gray}}
\put(325,108){\colorbox{gray}}
\put(329,108){\colorbox{gray}}
\put(290,112){\colorbox{gray}}
\put(295,112){\colorbox{gray}}
\put(300,112){\colorbox{gray}}
\put(305,112){\colorbox{gray}}
\put(310,112){\colorbox{gray}}
\put(315,112){\colorbox{gray}}
\put(320,112){\colorbox{gray}}
\put(325,112){\colorbox{gray}}
\put(329,112){\colorbox{gray}}

\put(80,63){\colorbox{gray}}
\put(85,63){\colorbox{gray}}
\put(90,63){\colorbox{gray}}
\put(95,63){\colorbox{gray}}
\put(100,63){\colorbox{gray}}
\put(105,63){\colorbox{gray}}
\put(110,63){\colorbox{gray}}
\put(115,63){\colorbox{gray}}
\put(120,63){\colorbox{gray}}
\put(124,63){\colorbox{gray}}
\put(80,68){\colorbox{gray}}
\put(85,68){\colorbox{gray}}
\put(90,68){\colorbox{gray}}
\put(95,68){\colorbox{gray}}
\put(100,68){\colorbox{gray}}
\put(105,68){\colorbox{gray}}
\put(110,68){\colorbox{gray}}
\put(115,68){\colorbox{gray}}
\put(120,68){\colorbox{gray}}
\put(124,68){\colorbox{gray}}
\put(80,72){\colorbox{gray}}
\put(85,72){\colorbox{gray}}
\put(90,72){\colorbox{gray}}
\put(95,72){\colorbox{gray}}
\put(100,72){\colorbox{gray}}
\put(105,72){\colorbox{gray}}
\put(110,72){\colorbox{gray}}
\put(115,72){\colorbox{gray}}
\put(120,72){\colorbox{gray}}
\put(124,72){\colorbox{gray}}

\put(130,63){\colorbox{gray}}
\put(135,63){\colorbox{gray}}
\put(140,63){\colorbox{gray}}
\put(145,63){\colorbox{gray}}
\put(150,63){\colorbox{gray}}
\put(155,63){\colorbox{gray}}
\put(159,63){\colorbox{gray}}
\put(130,68){\colorbox{gray}}
\put(135,68){\colorbox{gray}}
\put(140,68){\colorbox{gray}}
\put(145,68){\colorbox{gray}}
\put(150,68){\colorbox{gray}}
\put(155,68){\colorbox{gray}}
\put(159,68){\colorbox{gray}}
\put(130,72){\colorbox{gray}}
\put(135,72){\colorbox{gray}}
\put(140,72){\colorbox{gray}}
\put(145,72){\colorbox{gray}}
\put(150,72){\colorbox{gray}}
\put(155,72){\colorbox{gray}}
\put(159,72){\colorbox{gray}}

\put(165,63){\colorbox{gray}}
\put(170,63){\colorbox{gray}}
\put(175,63){\colorbox{gray}}
\put(180,63){\colorbox{gray}}
\put(185,63){\colorbox{gray}}
\put(190,63){\colorbox{gray}}
\put(195,63){\colorbox{gray}}
\put(200,63){\colorbox{gray}}
\put(205,63){\colorbox{gray}}
\put(210,63){\colorbox{gray}}
\put(214,63){\colorbox{gray}}
\put(165,68){\colorbox{gray}}
\put(170,68){\colorbox{gray}}
\put(175,68){\colorbox{gray}}
\put(180,68){\colorbox{gray}}
\put(185,68){\colorbox{gray}}
\put(190,68){\colorbox{gray}}
\put(195,68){\colorbox{gray}}
\put(200,68){\colorbox{gray}}
\put(205,68){\colorbox{gray}}
\put(210,68){\colorbox{gray}}
\put(214,68){\colorbox{gray}}
\put(165,72){\colorbox{gray}}
\put(170,72){\colorbox{gray}}
\put(175,72){\colorbox{gray}}
\put(180,72){\colorbox{gray}}
\put(185,72){\colorbox{gray}}
\put(190,72){\colorbox{gray}}
\put(195,72){\colorbox{gray}}
\put(200,72){\colorbox{gray}}
\put(205,72){\colorbox{gray}}
\put(210,72){\colorbox{gray}}
\put(214,72){\colorbox{gray}}

\put(165,48){\colorbox{gray}}
\put(170,48){\colorbox{gray}}
\put(175,48){\colorbox{gray}}
\put(180,48){\colorbox{gray}}
\put(185,48){\colorbox{gray}}
\put(190,48){\colorbox{gray}}
\put(195,48){\colorbox{gray}}
\put(200,48){\colorbox{gray}}
\put(205,48){\colorbox{gray}}
\put(210,48){\colorbox{gray}}
\put(214,48){\colorbox{gray}}
\put(165,53){\colorbox{gray}}
\put(170,53){\colorbox{gray}}
\put(175,53){\colorbox{gray}}
\put(180,53){\colorbox{gray}}
\put(185,53){\colorbox{gray}}
\put(190,53){\colorbox{gray}}
\put(195,53){\colorbox{gray}}
\put(200,53){\colorbox{gray}}
\put(205,53){\colorbox{gray}}
\put(210,53){\colorbox{gray}}
\put(214,53){\colorbox{gray}}
\put(165,57){\colorbox{gray}}
\put(170,57){\colorbox{gray}}
\put(175,57){\colorbox{gray}}
\put(180,57){\colorbox{gray}}
\put(185,57){\colorbox{gray}}
\put(190,57){\colorbox{gray}}
\put(195,57){\colorbox{gray}}
\put(200,57){\colorbox{gray}}
\put(205,57){\colorbox{gray}}
\put(210,57){\colorbox{gray}}
\put(214,57){\colorbox{gray}}

\put(0,150){\framebox(20,15){\tiny $(1,1)$}} 

\put(23,155){$\cdots$}
\put(23,135){$\ddots$}

\put(40,150){\framebox(20,15){\tiny $(1,i)$}}
\put(48,135){$\vdots$}
\put(40,115){\framebox(20,15){\tiny $(i,i)$}}

\put(63,155){$\cdots$}
\put(63,119){$\cdots$}
\put(63,89){$\ddots$}

\put(80,150){\framebox(50,15){\tiny $(1,n-1)$}}
\put(104,135){$\vdots$}
\put(80,115){\framebox(50,15){\tiny $(i,n-1)$}}
\put(80,100){\framebox(50,15){\tiny $(i+1,n-1)$}}
\put(104,82){$\vdots$}
\put(80,60){\framebox(50,15){\tiny $(n-1,n-1)$}}

\put(130,150){\framebox(35,15){\tiny $(1,n)$}}
\put(146,135){$\vdots$}
\put(130,115){\framebox(35,15){\tiny $(i,n)$}}
\put(130,100){\framebox(35,15){\tiny $(i+1,n)$}}
\put(146,82){$\vdots$}
\put(130,60){\framebox(35,15){\tiny $(n-1,n)$}}

\put(165,150){\framebox(55,15){\tiny $(n,2n-1)$}}
\put(189,135){$\vdots$}
\put(165,115){\framebox(55,15){\tiny $(n,2n-i)$}}
\put(165,100){\framebox(55,15){\tiny $(n,2n-1-i)$}}
\put(189,82){$\vdots$}
\put(165,60){\framebox(55,15){\tiny $(n,n+1)$}}
\put(165,45){\framebox(55,15){\tiny $(n,n)$}}

\put(220,150){\framebox(50,15){\tiny $(1,n+1)$}}
\put(244,135){$\vdots$}
\put(220,115){\framebox(50,15){\tiny $(i,n+1)$}}
\put(220,100){\framebox(50,15){\tiny $(i+1,n+1)$}}
\put(244,88){$\vdots$}

\put(318,155){$\cdots$}
\put(298,120){$\cdots$}
\put(273,104){$\cdots$}
\put(273,87){$\cdot$} 
\put(278,89.5){$\cdot$}
\put(283,92){$\cdot$}

\put(290,100){\framebox(45,15){\tiny $(i+1,d_{i+1})$}}

\put(335,115){\framebox(25,15){\tiny $(i,d_i)$}}

\put(363,135){$\cdot$} 
\put(368,137.5){$\cdot$}
\put(373,140){$\cdot$}

\put(380,150){\framebox(25,15){\tiny $(1,d_1)$}}
\end{picture}
\end{center}
\vspace{-50pt}
\caption{The Hessenberg function $h^{(i,n-1)}$ with $1 \leq i \leq n-1$.}
\label{picture:h^{(i,j)}Case2}
\end{figure} 

In this case, $\alpha \in I^{(i,j)}$ is one of the following forms
\begin{align}
x_k-x_\ell \ \ \ &(1 \leq k < \ell \leq i) \label{eq:Case2(1)}, \\
x_k-x_\ell \ \ \ &(1 \leq k \leq i < \ell \leq n-1) \label{eq:Case2(2)}, \\
x_k \pm x_\ell \ \ \ &(i < k < \ell \leq n) \label{eq:Case2(3)}, \\
x_k+x_n \ \ \ &(1 \leq k \leq i) \label{eq:Case2(4)}. 
\end{align}
Using the formula $\eqref{eq:psi2}$, we show that $\psi^{D_n}_{i,j} (\alpha) \equiv 0$ mod $\alpha$.
If $\alpha$ is of the form \eqref{eq:Case2(1)}, then we have
\begin{align*}
\psi^{D_n}_{i,n-1}(\alpha)=& \left((x_k-x_{i+1})\cdots(x_k-x_{n-1})(x_k+x_n) +(-1)^{n-i} x_{i+1}\cdots x_n \right) x_k^{-1} \\
&- \left((x_\ell-x_{i+1})\cdots(x_\ell-x_{n-1})(x_\ell+x_n) +(-1)^{n-i} x_{i+1}\cdots x_n \right) x_\ell^{-1} \\
\equiv& 0 \ \ \ \ \ {\rm mod} \ x_k-x_\ell.
\end{align*}
If $\alpha$ is of the form \eqref{eq:Case2(2)}, then 
\begin{align*}
\psi^{D_n}_{i,n-1}(\alpha)=& \left((x_k-x_{i+1})\cdots(x_k-x_{n-1})(x_k+x_n) +(-1)^{n-i} x_{i+1}\cdots x_n \right) x_k^{-1} \\
&-(-1)^{n-i} x_{i+1} \cdots \widehat{x_\ell} \cdots x_n \\
\equiv& (-1)^{n-i} x_{i+1}\cdots x_n x_\ell^{-1}-(-1)^{n-i} x_{i+1} \cdots \widehat{x_\ell} \cdots x_n \ \ \ \ \ {\rm mod} \ x_k-x_\ell \\
=& 0. 
\end{align*}
If $\alpha$ is of the form \eqref{eq:Case2(3)}, then
\begin{align*}
\psi^{D_n}_{i,n-1}(\alpha)=&(-1)^{n-i} x_{i+1} \cdots \widehat{x_k} \cdots x_n \pm (-1)^{n-i} x_{i+1} \cdots \widehat{x_\ell} \cdots x_n \\
\equiv& 0 \ \ \ \ \ {\rm mod} \ x_k \pm x_\ell. 
\end{align*}
If $\alpha$ is of the form \eqref{eq:Case2(4)}, then
\begin{align*}
\psi^{D_n}_{i,n-1}(\alpha)=&\left((x_k-x_{i+1})\cdots(x_k-x_{n-1})(x_k+x_n) +(-1)^{n-i} x_{i+1}\cdots x_n \right) x_k^{-1} \\
&+ (-1)^{n-i} x_{i+1} \cdots x_{n-1} \\
\equiv& (-1)^{n-i} x_{i+1}\cdots x_n (-x_n^{-1}) + (-1)^{n-i} x_{i+1} \cdots x_{n-1} \ \ \ \ \ {\rm mod} \ x_k+x_n \\
=& 0. 
\end{align*}
\noindent
\textit{Case~3} Suppose that $1 \leq i \leq n-1$ and $n \leq j \leq 2n-1-i$. Then the Hessenberg function $h^{(i,j)}$ is given by
\begin{align*}
h^{(i,j)}(t)=\begin{cases}
j \ \ \ &{\rm if} \ 1 \leq t \leq i, \\
2n-1-t \ \ \ &{\rm if} \ i+1 \leq t \leq n-1, \\
2n-1 \ \ \ &{\rm if} \ t = n, 
\end{cases}
\end{align*}
and the picture of the Hessenberg function $h^{(i,j)}$ is shown in Figure~\ref{picture:h^{(i,j)}Case3}.

\begin{figure}[h]
\begin{center}
\begin{picture}(435,165)
\put(0,153){\colorbox{gray}}
\put(5,153){\colorbox{gray}}
\put(10,153){\colorbox{gray}}
\put(14,153){\colorbox{gray}}
\put(0,158){\colorbox{gray}}
\put(5,158){\colorbox{gray}}
\put(10,158){\colorbox{gray}}
\put(14,158){\colorbox{gray}}
\put(0,162){\colorbox{gray}}
\put(5,162){\colorbox{gray}}
\put(10,162){\colorbox{gray}}
\put(14,162){\colorbox{gray}}

\put(40,153){\colorbox{gray}}
\put(45,153){\colorbox{gray}}
\put(50,153){\colorbox{gray}}
\put(54,153){\colorbox{gray}}
\put(40,158){\colorbox{gray}}
\put(45,158){\colorbox{gray}}
\put(50,158){\colorbox{gray}}
\put(54,158){\colorbox{gray}}
\put(40,162){\colorbox{gray}}
\put(45,162){\colorbox{gray}}
\put(50,162){\colorbox{gray}}
\put(54,162){\colorbox{gray}}

\put(80,153){\colorbox{gray}}
\put(85,153){\colorbox{gray}}
\put(90,153){\colorbox{gray}}
\put(95,153){\colorbox{gray}}
\put(100,153){\colorbox{gray}}
\put(105,153){\colorbox{gray}}
\put(110,153){\colorbox{gray}}
\put(115,153){\colorbox{gray}}
\put(120,153){\colorbox{gray}}
\put(124,153){\colorbox{gray}}
\put(80,158){\colorbox{gray}}
\put(85,158){\colorbox{gray}}
\put(90,158){\colorbox{gray}}
\put(95,158){\colorbox{gray}}
\put(100,158){\colorbox{gray}}
\put(105,158){\colorbox{gray}}
\put(110,158){\colorbox{gray}}
\put(115,158){\colorbox{gray}}
\put(120,158){\colorbox{gray}}
\put(124,158){\colorbox{gray}}
\put(80,162){\colorbox{gray}}
\put(85,162){\colorbox{gray}}
\put(90,162){\colorbox{gray}}
\put(95,162){\colorbox{gray}}
\put(100,162){\colorbox{gray}}
\put(105,162){\colorbox{gray}}
\put(110,162){\colorbox{gray}}
\put(115,162){\colorbox{gray}}
\put(120,162){\colorbox{gray}}
\put(124,162){\colorbox{gray}}

\put(130,153){\colorbox{gray}}
\put(135,153){\colorbox{gray}}
\put(140,153){\colorbox{gray}}
\put(145,153){\colorbox{gray}}
\put(150,153){\colorbox{gray}}
\put(155,153){\colorbox{gray}}
\put(159,153){\colorbox{gray}}
\put(130,158){\colorbox{gray}}
\put(135,158){\colorbox{gray}}
\put(140,158){\colorbox{gray}}
\put(145,158){\colorbox{gray}}
\put(150,158){\colorbox{gray}}
\put(155,158){\colorbox{gray}}
\put(159,158){\colorbox{gray}}
\put(130,162){\colorbox{gray}}
\put(135,162){\colorbox{gray}}
\put(140,162){\colorbox{gray}}
\put(145,162){\colorbox{gray}}
\put(150,162){\colorbox{gray}}
\put(155,162){\colorbox{gray}}
\put(159,162){\colorbox{gray}}

\put(165,153){\colorbox{gray}}
\put(170,153){\colorbox{gray}}
\put(175,153){\colorbox{gray}}
\put(180,153){\colorbox{gray}}
\put(185,153){\colorbox{gray}}
\put(190,153){\colorbox{gray}}
\put(195,153){\colorbox{gray}}
\put(200,153){\colorbox{gray}}
\put(205,153){\colorbox{gray}}
\put(210,153){\colorbox{gray}}
\put(214,153){\colorbox{gray}}
\put(165,158){\colorbox{gray}}
\put(170,158){\colorbox{gray}}
\put(175,158){\colorbox{gray}}
\put(180,158){\colorbox{gray}}
\put(185,158){\colorbox{gray}}
\put(190,158){\colorbox{gray}}
\put(195,158){\colorbox{gray}}
\put(200,158){\colorbox{gray}}
\put(205,158){\colorbox{gray}}
\put(210,158){\colorbox{gray}}
\put(214,158){\colorbox{gray}}
\put(165,162){\colorbox{gray}}
\put(170,162){\colorbox{gray}}
\put(175,162){\colorbox{gray}}
\put(180,162){\colorbox{gray}}
\put(185,162){\colorbox{gray}}
\put(190,162){\colorbox{gray}}
\put(195,162){\colorbox{gray}}
\put(200,162){\colorbox{gray}}
\put(205,162){\colorbox{gray}}
\put(210,162){\colorbox{gray}}
\put(214,162){\colorbox{gray}}

\put(240,153){\colorbox{gray}}
\put(245,153){\colorbox{gray}}
\put(250,153){\colorbox{gray}}
\put(254,153){\colorbox{gray}}
\put(240,158){\colorbox{gray}}
\put(245,158){\colorbox{gray}}
\put(250,158){\colorbox{gray}}
\put(254,158){\colorbox{gray}}
\put(240,162){\colorbox{gray}}
\put(245,162){\colorbox{gray}}
\put(250,162){\colorbox{gray}}
\put(254,162){\colorbox{gray}}

\put(40,118){\colorbox{gray}}
\put(45,118){\colorbox{gray}}
\put(50,118){\colorbox{gray}}
\put(54,118){\colorbox{gray}}
\put(40,123){\colorbox{gray}}
\put(45,123){\colorbox{gray}}
\put(50,123){\colorbox{gray}}
\put(54,123){\colorbox{gray}}
\put(40,127){\colorbox{gray}}
\put(45,127){\colorbox{gray}}
\put(50,127){\colorbox{gray}}
\put(54,127){\colorbox{gray}}

\put(80,118){\colorbox{gray}}
\put(85,118){\colorbox{gray}}
\put(90,118){\colorbox{gray}}
\put(95,118){\colorbox{gray}}
\put(100,118){\colorbox{gray}}
\put(105,118){\colorbox{gray}}
\put(110,118){\colorbox{gray}}
\put(115,118){\colorbox{gray}}
\put(120,118){\colorbox{gray}}
\put(124,118){\colorbox{gray}}
\put(80,123){\colorbox{gray}}
\put(85,123){\colorbox{gray}}
\put(90,123){\colorbox{gray}}
\put(95,123){\colorbox{gray}}
\put(100,123){\colorbox{gray}}
\put(105,123){\colorbox{gray}}
\put(110,123){\colorbox{gray}}
\put(115,123){\colorbox{gray}}
\put(120,123){\colorbox{gray}}
\put(124,123){\colorbox{gray}}
\put(80,127){\colorbox{gray}}
\put(85,127){\colorbox{gray}}
\put(90,127){\colorbox{gray}}
\put(95,127){\colorbox{gray}}
\put(100,127){\colorbox{gray}}
\put(105,127){\colorbox{gray}}
\put(110,127){\colorbox{gray}}
\put(115,127){\colorbox{gray}}
\put(120,127){\colorbox{gray}}
\put(124,127){\colorbox{gray}}

\put(130,118){\colorbox{gray}}
\put(135,118){\colorbox{gray}}
\put(140,118){\colorbox{gray}}
\put(145,118){\colorbox{gray}}
\put(150,118){\colorbox{gray}}
\put(155,118){\colorbox{gray}}
\put(159,118){\colorbox{gray}}
\put(130,123){\colorbox{gray}}
\put(135,123){\colorbox{gray}}
\put(140,123){\colorbox{gray}}
\put(145,123){\colorbox{gray}}
\put(150,123){\colorbox{gray}}
\put(155,123){\colorbox{gray}}
\put(159,123){\colorbox{gray}}
\put(130,127){\colorbox{gray}}
\put(135,127){\colorbox{gray}}
\put(140,127){\colorbox{gray}}
\put(145,127){\colorbox{gray}}
\put(150,127){\colorbox{gray}}
\put(155,127){\colorbox{gray}}
\put(159,127){\colorbox{gray}}

\put(165,118){\colorbox{gray}}
\put(170,118){\colorbox{gray}}
\put(175,118){\colorbox{gray}}
\put(180,118){\colorbox{gray}}
\put(185,118){\colorbox{gray}}
\put(190,118){\colorbox{gray}}
\put(195,118){\colorbox{gray}}
\put(200,118){\colorbox{gray}}
\put(205,118){\colorbox{gray}}
\put(210,118){\colorbox{gray}}
\put(214,118){\colorbox{gray}}
\put(165,123){\colorbox{gray}}
\put(170,123){\colorbox{gray}}
\put(175,123){\colorbox{gray}}
\put(180,123){\colorbox{gray}}
\put(185,123){\colorbox{gray}}
\put(190,123){\colorbox{gray}}
\put(195,123){\colorbox{gray}}
\put(200,123){\colorbox{gray}}
\put(205,123){\colorbox{gray}}
\put(210,123){\colorbox{gray}}
\put(214,123){\colorbox{gray}}
\put(165,127){\colorbox{gray}}
\put(170,127){\colorbox{gray}}
\put(175,127){\colorbox{gray}}
\put(180,127){\colorbox{gray}}
\put(185,127){\colorbox{gray}}
\put(190,127){\colorbox{gray}}
\put(195,127){\colorbox{gray}}
\put(200,127){\colorbox{gray}}
\put(205,127){\colorbox{gray}}
\put(210,127){\colorbox{gray}}
\put(214,127){\colorbox{gray}}

\put(240,118){\colorbox{gray}}
\put(245,118){\colorbox{gray}}
\put(250,118){\colorbox{gray}}
\put(254,118){\colorbox{gray}}
\put(240,123){\colorbox{gray}}
\put(245,123){\colorbox{gray}}
\put(250,123){\colorbox{gray}}
\put(254,123){\colorbox{gray}}
\put(240,127){\colorbox{gray}}
\put(245,127){\colorbox{gray}}
\put(250,127){\colorbox{gray}}
\put(254,127){\colorbox{gray}}

\put(80,103){\colorbox{gray}}
\put(85,103){\colorbox{gray}}
\put(90,103){\colorbox{gray}}
\put(95,103){\colorbox{gray}}
\put(100,103){\colorbox{gray}}
\put(105,103){\colorbox{gray}}
\put(110,103){\colorbox{gray}}
\put(115,103){\colorbox{gray}}
\put(120,103){\colorbox{gray}}
\put(124,103){\colorbox{gray}}
\put(80,108){\colorbox{gray}}
\put(85,108){\colorbox{gray}}
\put(90,108){\colorbox{gray}}
\put(95,108){\colorbox{gray}}
\put(100,108){\colorbox{gray}}
\put(105,108){\colorbox{gray}}
\put(110,108){\colorbox{gray}}
\put(115,108){\colorbox{gray}}
\put(120,108){\colorbox{gray}}
\put(124,108){\colorbox{gray}}
\put(80,112){\colorbox{gray}}
\put(85,112){\colorbox{gray}}
\put(90,112){\colorbox{gray}}
\put(95,112){\colorbox{gray}}
\put(100,112){\colorbox{gray}}
\put(105,112){\colorbox{gray}}
\put(110,112){\colorbox{gray}}
\put(115,112){\colorbox{gray}}
\put(120,112){\colorbox{gray}}
\put(124,112){\colorbox{gray}}

\put(130,103){\colorbox{gray}}
\put(135,103){\colorbox{gray}}
\put(140,103){\colorbox{gray}}
\put(145,103){\colorbox{gray}}
\put(150,103){\colorbox{gray}}
\put(155,103){\colorbox{gray}}
\put(159,103){\colorbox{gray}}
\put(130,108){\colorbox{gray}}
\put(135,108){\colorbox{gray}}
\put(140,108){\colorbox{gray}}
\put(145,108){\colorbox{gray}}
\put(150,108){\colorbox{gray}}
\put(155,108){\colorbox{gray}}
\put(159,108){\colorbox{gray}}
\put(130,112){\colorbox{gray}}
\put(135,112){\colorbox{gray}}
\put(140,112){\colorbox{gray}}
\put(145,112){\colorbox{gray}}
\put(150,112){\colorbox{gray}}
\put(155,112){\colorbox{gray}}
\put(159,112){\colorbox{gray}}

\put(165,103){\colorbox{gray}}
\put(170,103){\colorbox{gray}}
\put(175,103){\colorbox{gray}}
\put(180,103){\colorbox{gray}}
\put(185,103){\colorbox{gray}}
\put(190,103){\colorbox{gray}}
\put(195,103){\colorbox{gray}}
\put(200,103){\colorbox{gray}}
\put(205,103){\colorbox{gray}}
\put(210,103){\colorbox{gray}}
\put(214,103){\colorbox{gray}}
\put(165,108){\colorbox{gray}}
\put(170,108){\colorbox{gray}}
\put(175,108){\colorbox{gray}}
\put(180,108){\colorbox{gray}}
\put(185,108){\colorbox{gray}}
\put(190,108){\colorbox{gray}}
\put(195,108){\colorbox{gray}}
\put(200,108){\colorbox{gray}}
\put(205,108){\colorbox{gray}}
\put(210,108){\colorbox{gray}}
\put(214,108){\colorbox{gray}}
\put(165,112){\colorbox{gray}}
\put(170,112){\colorbox{gray}}
\put(175,112){\colorbox{gray}}
\put(180,112){\colorbox{gray}}
\put(185,112){\colorbox{gray}}
\put(190,112){\colorbox{gray}}
\put(195,112){\colorbox{gray}}
\put(200,112){\colorbox{gray}}
\put(205,112){\colorbox{gray}}
\put(210,112){\colorbox{gray}}
\put(214,112){\colorbox{gray}}

\put(260,103){\colorbox{gray}}
\put(265,103){\colorbox{gray}}
\put(270,103){\colorbox{gray}}
\put(275,103){\colorbox{gray}}
\put(280,103){\colorbox{gray}}
\put(285,103){\colorbox{gray}}
\put(290,103){\colorbox{gray}}
\put(295,103){\colorbox{gray}}
\put(300,103){\colorbox{gray}}
\put(304,103){\colorbox{gray}}
\put(260,108){\colorbox{gray}}
\put(265,108){\colorbox{gray}}
\put(270,108){\colorbox{gray}}
\put(275,108){\colorbox{gray}}
\put(280,108){\colorbox{gray}}
\put(285,108){\colorbox{gray}}
\put(290,108){\colorbox{gray}}
\put(295,108){\colorbox{gray}}
\put(300,108){\colorbox{gray}}
\put(304,108){\colorbox{gray}}
\put(260,112){\colorbox{gray}}
\put(265,112){\colorbox{gray}}
\put(270,112){\colorbox{gray}}
\put(275,112){\colorbox{gray}}
\put(280,112){\colorbox{gray}}
\put(285,112){\colorbox{gray}}
\put(290,112){\colorbox{gray}}
\put(295,112){\colorbox{gray}}
\put(300,112){\colorbox{gray}}
\put(304,112){\colorbox{gray}}

\put(330,103){\colorbox{gray}}
\put(335,103){\colorbox{gray}}
\put(340,103){\colorbox{gray}}
\put(345,103){\colorbox{gray}}
\put(350,103){\colorbox{gray}}
\put(355,103){\colorbox{gray}}
\put(360,103){\colorbox{gray}}
\put(365,103){\colorbox{gray}}
\put(369,103){\colorbox{gray}}
\put(330,108){\colorbox{gray}}
\put(335,108){\colorbox{gray}}
\put(340,108){\colorbox{gray}}
\put(345,108){\colorbox{gray}}
\put(350,108){\colorbox{gray}}
\put(355,108){\colorbox{gray}}
\put(360,108){\colorbox{gray}}
\put(365,108){\colorbox{gray}}
\put(369,108){\colorbox{gray}}
\put(330,112){\colorbox{gray}}
\put(335,112){\colorbox{gray}}
\put(340,112){\colorbox{gray}}
\put(345,112){\colorbox{gray}}
\put(350,112){\colorbox{gray}}
\put(355,112){\colorbox{gray}}
\put(360,112){\colorbox{gray}}
\put(365,112){\colorbox{gray}}
\put(369,112){\colorbox{gray}}

\put(80,58){\colorbox{gray}}
\put(85,58){\colorbox{gray}}
\put(90,58){\colorbox{gray}}
\put(95,58){\colorbox{gray}}
\put(100,58){\colorbox{gray}}
\put(105,58){\colorbox{gray}}
\put(110,58){\colorbox{gray}}
\put(115,58){\colorbox{gray}}
\put(120,58){\colorbox{gray}}
\put(124,58){\colorbox{gray}}
\put(80,63){\colorbox{gray}}
\put(85,63){\colorbox{gray}}
\put(90,63){\colorbox{gray}}
\put(95,63){\colorbox{gray}}
\put(100,63){\colorbox{gray}}
\put(105,63){\colorbox{gray}}
\put(110,63){\colorbox{gray}}
\put(115,63){\colorbox{gray}}
\put(120,63){\colorbox{gray}}
\put(124,63){\colorbox{gray}}
\put(80,67){\colorbox{gray}}
\put(85,67){\colorbox{gray}}
\put(90,67){\colorbox{gray}}
\put(95,67){\colorbox{gray}}
\put(100,67){\colorbox{gray}}
\put(105,67){\colorbox{gray}}
\put(110,67){\colorbox{gray}}
\put(115,67){\colorbox{gray}}
\put(120,67){\colorbox{gray}}
\put(124,67){\colorbox{gray}}

\put(130,58){\colorbox{gray}}
\put(135,58){\colorbox{gray}}
\put(140,58){\colorbox{gray}}
\put(145,58){\colorbox{gray}}
\put(150,58){\colorbox{gray}}
\put(155,58){\colorbox{gray}}
\put(159,58){\colorbox{gray}}
\put(130,63){\colorbox{gray}}
\put(135,63){\colorbox{gray}}
\put(140,63){\colorbox{gray}}
\put(145,63){\colorbox{gray}}
\put(150,63){\colorbox{gray}}
\put(155,63){\colorbox{gray}}
\put(159,63){\colorbox{gray}}
\put(130,67){\colorbox{gray}}
\put(135,67){\colorbox{gray}}
\put(140,67){\colorbox{gray}}
\put(145,67){\colorbox{gray}}
\put(150,67){\colorbox{gray}}
\put(155,67){\colorbox{gray}}
\put(159,67){\colorbox{gray}}

\put(165,58){\colorbox{gray}}
\put(170,58){\colorbox{gray}}
\put(175,58){\colorbox{gray}}
\put(180,58){\colorbox{gray}}
\put(185,58){\colorbox{gray}}
\put(190,58){\colorbox{gray}}
\put(195,58){\colorbox{gray}}
\put(200,58){\colorbox{gray}}
\put(205,58){\colorbox{gray}}
\put(210,58){\colorbox{gray}}
\put(214,58){\colorbox{gray}}
\put(165,63){\colorbox{gray}}
\put(170,63){\colorbox{gray}}
\put(175,63){\colorbox{gray}}
\put(180,63){\colorbox{gray}}
\put(185,63){\colorbox{gray}}
\put(190,63){\colorbox{gray}}
\put(195,63){\colorbox{gray}}
\put(200,63){\colorbox{gray}}
\put(205,63){\colorbox{gray}}
\put(210,63){\colorbox{gray}}
\put(214,63){\colorbox{gray}}
\put(165,67){\colorbox{gray}}
\put(170,67){\colorbox{gray}}
\put(175,67){\colorbox{gray}}
\put(180,67){\colorbox{gray}}
\put(185,67){\colorbox{gray}}
\put(190,67){\colorbox{gray}}
\put(195,67){\colorbox{gray}}
\put(200,67){\colorbox{gray}}
\put(205,67){\colorbox{gray}}
\put(210,67){\colorbox{gray}}
\put(214,67){\colorbox{gray}}

\put(165,43){\colorbox{gray}}
\put(170,43){\colorbox{gray}}
\put(175,43){\colorbox{gray}}
\put(180,43){\colorbox{gray}}
\put(185,43){\colorbox{gray}}
\put(190,43){\colorbox{gray}}
\put(195,43){\colorbox{gray}}
\put(200,43){\colorbox{gray}}
\put(205,43){\colorbox{gray}}
\put(210,43){\colorbox{gray}}
\put(214,43){\colorbox{gray}}
\put(165,48){\colorbox{gray}}
\put(170,48){\colorbox{gray}}
\put(175,48){\colorbox{gray}}
\put(180,48){\colorbox{gray}}
\put(185,48){\colorbox{gray}}
\put(190,48){\colorbox{gray}}
\put(195,48){\colorbox{gray}}
\put(200,48){\colorbox{gray}}
\put(205,48){\colorbox{gray}}
\put(210,48){\colorbox{gray}}
\put(214,48){\colorbox{gray}}
\put(165,52){\colorbox{gray}}
\put(170,52){\colorbox{gray}}
\put(175,52){\colorbox{gray}}
\put(180,52){\colorbox{gray}}
\put(185,52){\colorbox{gray}}
\put(190,52){\colorbox{gray}}
\put(195,52){\colorbox{gray}}
\put(200,52){\colorbox{gray}}
\put(205,52){\colorbox{gray}}
\put(210,52){\colorbox{gray}}
\put(214,52){\colorbox{gray}}

\put(0,150){\framebox(20,15){\tiny $(1,1)$}} 

\put(23,155){$\cdots$}
\put(23,135){$\ddots$}

\put(40,150){\framebox(20,15){\tiny $(1,i)$}}
\put(48,135){$\vdots$}
\put(40,115){\framebox(20,15){\tiny $(i,i)$}}

\put(63,155){$\cdots$}
\put(63,119){$\cdots$}
\put(63,90){$\ddots$}

\put(80,150){\framebox(50,15){\tiny $(1,n-1)$}}
\put(104,135){$\vdots$}
\put(80,115){\framebox(50,15){\tiny $(i,n-1)$}}
\put(80,100){\framebox(50,15){\tiny $(i+1,n-1)$}}
\put(104,80){$\vdots$}
\put(80,55){\framebox(50,15){\tiny $(n-1,n-1)$}}

\put(130,150){\framebox(35,15){\tiny $(1,n)$}}
\put(146,135){$\vdots$}
\put(130,115){\framebox(35,15){\tiny $(i,n)$}}
\put(130,100){\framebox(35,15){\tiny $(i+1,n)$}}
\put(146,80){$\vdots$}
\put(130,55){\framebox(35,15){\tiny $(n-1,n)$}}

\put(165,150){\framebox(55,15){\tiny $(n,2n-1)$}}
\put(189,135){$\vdots$}
\put(165,115){\framebox(55,15){\tiny $(n,2n-i)$}}
\put(165,100){\framebox(55,15){\tiny $(n,2n-1-i)$}}
\put(189,80){$\vdots$}
\put(165,55){\framebox(55,15){\tiny $(n,n+1)$}}
\put(165,40){\framebox(55,15){\tiny $(n,n)$}}

\put(223,155){$\cdots$}
\put(223,120){$\cdots$}
\put(223,104){$\cdots$}
\put(223,66){$\cdot$} 
\put(228,67.5){$\cdot$}
\put(233,69){$\cdot$}

\put(240,150){\framebox(20,15){\tiny $(1,j)$}}
\put(248,135){$\vdots$}
\put(240,115){\framebox(20,15){\tiny $(i,j)$}}
\put(248,102){$\vdots$}

\put(260,150){\framebox(50,15){\tiny $(1,j+1)$}}
\put(283,135){$\vdots$}
\put(260,115){\framebox(50,15){\tiny $(i,j+1)$}}
\put(260,100){\framebox(50,15){\tiny $(i+1,j+1)$}}
\put(283,88){$\vdots$}

\put(353,155){$\cdots$}
\put(338,120){$\cdots$}
\put(313,104){$\cdots$}
\put(313,91){$\cdot$} 
\put(318,92.5){$\cdot$}
\put(323,94){$\cdot$}

\put(330,100){\framebox(45,15){\tiny $(i+1,d_{i+1})$}}

\put(375,115){\framebox(25,15){\tiny $(i,d_i)$}}

\put(401,135){$\cdot$} 
\put(403,137.5){$\cdot$}
\put(405,140){$\cdot$}

\put(410,150){\framebox(25,15){\tiny $(1,d_1)$}}
\end{picture}
\end{center}
\vspace{-50pt}
\caption{The Hessenberg function $h^{(i,j)}$ for $1 \leq i \leq n-1, n \leq j \leq 2n-1-i$.}
\label{picture:h^{(i,j)}Case3}
\end{figure} 

Then, $\alpha \in I^{(i,j)}$ is one of the following forms
\begin{align}
x_k-x_\ell \ \ \ &(1 \leq k < \ell \leq i) \label{eq:Case3(1)}, \\
x_k-x_\ell \ \ \ &(1 \leq k \leq i < \ell \leq n) \label{eq:Case3(2)}, \\
x_k+x_\ell \ \ \ &(1 \leq k \leq i, 2n-j \leq \ell \leq n) \label{eq:Case3(3)}, \\
x_k \pm x_\ell \ \ \ &(i < k < \ell \leq n) \label{eq:Case3(4)}.
\end{align}
Using the formula $\eqref{eq:psi3}$, we show that $\psi^{D_n}_{i,j} (\alpha) \equiv 0$ mod $\alpha$.
We put $s=j-n$. Then $0 \leq s \leq n-1-i$ and $j=n+s$.
If $\alpha$ is of the form \eqref{eq:Case3(1)}, then we have
\begin{align*}
\psi^{D_n}_{i,n+s}(\alpha)=& \big((x_k-x_{i+1})\cdots(x_k-x_{n})(x_k+x_n)\cdots(x_k+x_{n-s}) \\
& \ \ \ +(-1)^{n-i+1} x_{i+1}\cdots x_{n-1-s} x_{n-s}^2\cdots x_n^2 \big)x_k^{-1} \\
&-\big((x_\ell-x_{i+1})\cdots(x_\ell-x_{n})(x_\ell+x_n)\cdots(x_\ell+x_{n-s}) \\
& \ \ \ +(-1)^{n-i+1} x_{i+1}\cdots x_{n-1-s} x_{n-s}^2\cdots x_n^2 \big)x_\ell^{-1} \\
\equiv& 0 \ \ \ \ \ {\rm mod} \ x_k-x_\ell
\end{align*}
If $\alpha$ is of the form \eqref{eq:Case3(2)}, then 
\begin{align*}
\psi^{D_n}_{i,n+s}(\alpha)=& \big((x_k-x_{i+1})\cdots(x_k-x_{n})(x_k+x_n)\cdots(x_k+x_{n-s}) \\
&+(-1)^{n-i+1} x_{i+1}\cdots x_{n-1-s} x_{n-s}^2\cdots x_n^2 \big)x_k^{-1} \\
&-(-1)^{n-i+1} x_{n-s} \cdots x_n (x_{i+1} \cdots \widehat{x_\ell} \cdots x_n) \\
=& \big((x_k-x_{i+1})\cdots(x_k-x_{n})(x_k+x_n)\cdots(x_k+x_{n-s}) \\
&+(-1)^{n-i+1} x_{i+1}\cdots x_{n-1-s} x_{n-s}^2\cdots x_n^2 \\
&-(-1)^{n-i+1} x_{i+1}\cdots x_{n-1-s} x_{n-s}^2\cdots x_n^2 (x_k x_\ell^{-1}) \big)x_k^{-1} \\
\equiv& 0 \ \ \ \ \ {\rm mod} \ x_k-x_\ell.  
\end{align*}
If $\alpha$ is of the form \eqref{eq:Case3(3)}, then $i< 2n-j \leq \ell$ because of the condition for Case~3.
So, we have
\begin{align*}
\psi^{D_n}_{i,n+s}(\alpha)=& \big((x_k-x_{i+1})\cdots(x_k-x_{n})(x_k+x_n)\cdots(x_k+x_{n-s}) \\
&+(-1)^{n-i+1} x_{i+1}\cdots x_{n-1-s} x_{n-s}^2\cdots x_n^2 \big)x_k^{-1} \\
&+(-1)^{n-i+1} x_{n-s} \cdots x_n (x_{i+1} \cdots \widehat{x_\ell} \cdots x_n) \\
\equiv& 0 \ \ \ \ \ {\rm mod} \ x_k+x_\ell. 
\end{align*}
If $\alpha$ is of the form \eqref{eq:Case3(4)}, then
\begin{align*}
\psi^{D_n}_{i,n+s}(\alpha)=&(-1)^{n-i+1} x_{n-s} \cdots x_n (x_{i+1} \cdots \widehat{x_k} \cdots x_n) \pm (-1)^{n-i+1} x_{n-s} \cdots x_n (x_{i+1} \cdots \widehat{x_\ell} \cdots x_n) \\
\equiv& 0 \ \ \ \ \ {\rm mod} \ x_k \pm x_\ell. 
\end{align*}
\textit{Case~4} Suppose that $i = n$ and $n \leq j \leq 2n-1$. 
We put $r=2n-1-j$. 
Then $0 \leq r \leq n-1$ and the Hessenberg function $h^{(i,j)}$ is given as follows
\begin{align*}
h^{(i,j)}(t)=\begin{cases}
n \ \ \ &{\rm if} \ 1 \leq t \leq r, \\
2n-1-t \ \ \ &{\rm if} \ r+1 \leq t \leq n-1, \\
2n-1-r \ \ \ &{\rm if} \ t = n. 
\end{cases}
\end{align*}
The picture of the Hessenberg function $h^{(i,j)}$ is shown in Figure~\ref{picture:h^{(i,j)}Case4}.

\begin{figure}[h]
\begin{center}
\begin{picture}(405,165)
\put(0,153){\colorbox{gray}}
\put(5,153){\colorbox{gray}}
\put(10,153){\colorbox{gray}}
\put(14,153){\colorbox{gray}}
\put(0,158){\colorbox{gray}}
\put(5,158){\colorbox{gray}}
\put(10,158){\colorbox{gray}}
\put(14,158){\colorbox{gray}}
\put(0,162){\colorbox{gray}}
\put(5,162){\colorbox{gray}}
\put(10,162){\colorbox{gray}}
\put(14,162){\colorbox{gray}}

\put(40,153){\colorbox{gray}}
\put(45,153){\colorbox{gray}}
\put(50,153){\colorbox{gray}}
\put(54,153){\colorbox{gray}}
\put(40,158){\colorbox{gray}}
\put(45,158){\colorbox{gray}}
\put(50,158){\colorbox{gray}}
\put(54,158){\colorbox{gray}}
\put(40,162){\colorbox{gray}}
\put(45,162){\colorbox{gray}}
\put(50,162){\colorbox{gray}}
\put(54,162){\colorbox{gray}}

\put(80,153){\colorbox{gray}}
\put(85,153){\colorbox{gray}}
\put(90,153){\colorbox{gray}}
\put(95,153){\colorbox{gray}}
\put(100,153){\colorbox{gray}}
\put(105,153){\colorbox{gray}}
\put(110,153){\colorbox{gray}}
\put(115,153){\colorbox{gray}}
\put(120,153){\colorbox{gray}}
\put(124,153){\colorbox{gray}}
\put(80,158){\colorbox{gray}}
\put(85,158){\colorbox{gray}}
\put(90,158){\colorbox{gray}}
\put(95,158){\colorbox{gray}}
\put(100,158){\colorbox{gray}}
\put(105,158){\colorbox{gray}}
\put(110,158){\colorbox{gray}}
\put(115,158){\colorbox{gray}}
\put(120,158){\colorbox{gray}}
\put(124,158){\colorbox{gray}}
\put(80,162){\colorbox{gray}}
\put(85,162){\colorbox{gray}}
\put(90,162){\colorbox{gray}}
\put(95,162){\colorbox{gray}}
\put(100,162){\colorbox{gray}}
\put(105,162){\colorbox{gray}}
\put(110,162){\colorbox{gray}}
\put(115,162){\colorbox{gray}}
\put(120,162){\colorbox{gray}}
\put(124,162){\colorbox{gray}}

\put(130,153){\colorbox{gray}}
\put(135,153){\colorbox{gray}}
\put(140,153){\colorbox{gray}}
\put(145,153){\colorbox{gray}}
\put(150,153){\colorbox{gray}}
\put(155,153){\colorbox{gray}}
\put(159,153){\colorbox{gray}}
\put(130,158){\colorbox{gray}}
\put(135,158){\colorbox{gray}}
\put(140,158){\colorbox{gray}}
\put(145,158){\colorbox{gray}}
\put(150,158){\colorbox{gray}}
\put(155,158){\colorbox{gray}}
\put(159,158){\colorbox{gray}}
\put(130,162){\colorbox{gray}}
\put(135,162){\colorbox{gray}}
\put(140,162){\colorbox{gray}}
\put(145,162){\colorbox{gray}}
\put(150,162){\colorbox{gray}}
\put(155,162){\colorbox{gray}}
\put(159,162){\colorbox{gray}}

\put(40,118){\colorbox{gray}}
\put(45,118){\colorbox{gray}}
\put(50,118){\colorbox{gray}}
\put(54,118){\colorbox{gray}}
\put(40,123){\colorbox{gray}}
\put(45,123){\colorbox{gray}}
\put(50,123){\colorbox{gray}}
\put(54,123){\colorbox{gray}}
\put(40,127){\colorbox{gray}}
\put(45,127){\colorbox{gray}}
\put(50,127){\colorbox{gray}}
\put(54,127){\colorbox{gray}}

\put(80,118){\colorbox{gray}}
\put(85,118){\colorbox{gray}}
\put(90,118){\colorbox{gray}}
\put(95,118){\colorbox{gray}}
\put(100,118){\colorbox{gray}}
\put(105,118){\colorbox{gray}}
\put(110,118){\colorbox{gray}}
\put(115,118){\colorbox{gray}}
\put(120,118){\colorbox{gray}}
\put(124,118){\colorbox{gray}}
\put(80,123){\colorbox{gray}}
\put(85,123){\colorbox{gray}}
\put(90,123){\colorbox{gray}}
\put(95,123){\colorbox{gray}}
\put(100,123){\colorbox{gray}}
\put(105,123){\colorbox{gray}}
\put(110,123){\colorbox{gray}}
\put(115,123){\colorbox{gray}}
\put(120,123){\colorbox{gray}}
\put(124,123){\colorbox{gray}}
\put(80,127){\colorbox{gray}}
\put(85,127){\colorbox{gray}}
\put(90,127){\colorbox{gray}}
\put(95,127){\colorbox{gray}}
\put(100,127){\colorbox{gray}}
\put(105,127){\colorbox{gray}}
\put(110,127){\colorbox{gray}}
\put(115,127){\colorbox{gray}}
\put(120,127){\colorbox{gray}}
\put(124,127){\colorbox{gray}}

\put(130,118){\colorbox{gray}}
\put(135,118){\colorbox{gray}}
\put(140,118){\colorbox{gray}}
\put(145,118){\colorbox{gray}}
\put(150,118){\colorbox{gray}}
\put(155,118){\colorbox{gray}}
\put(159,118){\colorbox{gray}}
\put(130,123){\colorbox{gray}}
\put(135,123){\colorbox{gray}}
\put(140,123){\colorbox{gray}}
\put(145,123){\colorbox{gray}}
\put(150,123){\colorbox{gray}}
\put(155,123){\colorbox{gray}}
\put(159,123){\colorbox{gray}}
\put(130,127){\colorbox{gray}}
\put(135,127){\colorbox{gray}}
\put(140,127){\colorbox{gray}}
\put(145,127){\colorbox{gray}}
\put(150,127){\colorbox{gray}}
\put(155,127){\colorbox{gray}}
\put(159,127){\colorbox{gray}}

\put(80,103){\colorbox{gray}}
\put(85,103){\colorbox{gray}}
\put(90,103){\colorbox{gray}}
\put(95,103){\colorbox{gray}}
\put(100,103){\colorbox{gray}}
\put(105,103){\colorbox{gray}}
\put(110,103){\colorbox{gray}}
\put(115,103){\colorbox{gray}}
\put(120,103){\colorbox{gray}}
\put(124,103){\colorbox{gray}}
\put(80,108){\colorbox{gray}}
\put(85,108){\colorbox{gray}}
\put(90,108){\colorbox{gray}}
\put(95,108){\colorbox{gray}}
\put(100,108){\colorbox{gray}}
\put(105,108){\colorbox{gray}}
\put(110,108){\colorbox{gray}}
\put(115,108){\colorbox{gray}}
\put(120,108){\colorbox{gray}}
\put(124,108){\colorbox{gray}}
\put(80,112){\colorbox{gray}}
\put(85,112){\colorbox{gray}}
\put(90,112){\colorbox{gray}}
\put(95,112){\colorbox{gray}}
\put(100,112){\colorbox{gray}}
\put(105,112){\colorbox{gray}}
\put(110,112){\colorbox{gray}}
\put(115,112){\colorbox{gray}}
\put(120,112){\colorbox{gray}}
\put(124,112){\colorbox{gray}}

\put(130,103){\colorbox{gray}}
\put(135,103){\colorbox{gray}}
\put(140,103){\colorbox{gray}}
\put(145,103){\colorbox{gray}}
\put(150,103){\colorbox{gray}}
\put(155,103){\colorbox{gray}}
\put(159,103){\colorbox{gray}}
\put(130,108){\colorbox{gray}}
\put(135,108){\colorbox{gray}}
\put(140,108){\colorbox{gray}}
\put(145,108){\colorbox{gray}}
\put(150,108){\colorbox{gray}}
\put(155,108){\colorbox{gray}}
\put(159,108){\colorbox{gray}}
\put(130,112){\colorbox{gray}}
\put(135,112){\colorbox{gray}}
\put(140,112){\colorbox{gray}}
\put(145,112){\colorbox{gray}}
\put(150,112){\colorbox{gray}}
\put(155,112){\colorbox{gray}}
\put(159,112){\colorbox{gray}}

\put(165,103){\colorbox{gray}}
\put(170,103){\colorbox{gray}}
\put(175,103){\colorbox{gray}}
\put(180,103){\colorbox{gray}}
\put(185,103){\colorbox{gray}}
\put(190,103){\colorbox{gray}}
\put(195,103){\colorbox{gray}}
\put(200,103){\colorbox{gray}}
\put(205,103){\colorbox{gray}}
\put(210,103){\colorbox{gray}}
\put(214,103){\colorbox{gray}}
\put(165,108){\colorbox{gray}}
\put(170,108){\colorbox{gray}}
\put(175,108){\colorbox{gray}}
\put(180,108){\colorbox{gray}}
\put(185,108){\colorbox{gray}}
\put(190,108){\colorbox{gray}}
\put(195,108){\colorbox{gray}}
\put(200,108){\colorbox{gray}}
\put(205,108){\colorbox{gray}}
\put(210,108){\colorbox{gray}}
\put(214,108){\colorbox{gray}}
\put(165,112){\colorbox{gray}}
\put(170,112){\colorbox{gray}}
\put(175,112){\colorbox{gray}}
\put(180,112){\colorbox{gray}}
\put(185,112){\colorbox{gray}}
\put(190,112){\colorbox{gray}}
\put(195,112){\colorbox{gray}}
\put(200,112){\colorbox{gray}}
\put(205,112){\colorbox{gray}}
\put(210,112){\colorbox{gray}}
\put(214,112){\colorbox{gray}}

\put(220,103){\colorbox{gray}}
\put(225,103){\colorbox{gray}}
\put(230,103){\colorbox{gray}}
\put(235,103){\colorbox{gray}}
\put(240,103){\colorbox{gray}}
\put(245,103){\colorbox{gray}}
\put(250,103){\colorbox{gray}}
\put(255,103){\colorbox{gray}}
\put(260,103){\colorbox{gray}}
\put(264,103){\colorbox{gray}}
\put(220,108){\colorbox{gray}}
\put(225,108){\colorbox{gray}}
\put(230,108){\colorbox{gray}}
\put(235,108){\colorbox{gray}}
\put(240,108){\colorbox{gray}}
\put(245,108){\colorbox{gray}}
\put(250,108){\colorbox{gray}}
\put(255,108){\colorbox{gray}}
\put(260,108){\colorbox{gray}}
\put(264,108){\colorbox{gray}}
\put(220,112){\colorbox{gray}}
\put(225,112){\colorbox{gray}}
\put(230,112){\colorbox{gray}}
\put(235,112){\colorbox{gray}}
\put(240,112){\colorbox{gray}}
\put(245,112){\colorbox{gray}}
\put(250,112){\colorbox{gray}}
\put(255,112){\colorbox{gray}}
\put(260,112){\colorbox{gray}}
\put(264,112){\colorbox{gray}}

\put(290,103){\colorbox{gray}}
\put(295,103){\colorbox{gray}}
\put(300,103){\colorbox{gray}}
\put(305,103){\colorbox{gray}}
\put(310,103){\colorbox{gray}}
\put(315,103){\colorbox{gray}}
\put(320,103){\colorbox{gray}}
\put(325,103){\colorbox{gray}}
\put(329,103){\colorbox{gray}}
\put(290,108){\colorbox{gray}}
\put(295,108){\colorbox{gray}}
\put(300,108){\colorbox{gray}}
\put(305,108){\colorbox{gray}}
\put(310,108){\colorbox{gray}}
\put(315,108){\colorbox{gray}}
\put(320,108){\colorbox{gray}}
\put(325,108){\colorbox{gray}}
\put(329,108){\colorbox{gray}}
\put(290,112){\colorbox{gray}}
\put(295,112){\colorbox{gray}}
\put(300,112){\colorbox{gray}}
\put(305,112){\colorbox{gray}}
\put(310,112){\colorbox{gray}}
\put(315,112){\colorbox{gray}}
\put(320,112){\colorbox{gray}}
\put(325,112){\colorbox{gray}}
\put(329,112){\colorbox{gray}}

\put(80,63){\colorbox{gray}}
\put(85,63){\colorbox{gray}}
\put(90,63){\colorbox{gray}}
\put(95,63){\colorbox{gray}}
\put(100,63){\colorbox{gray}}
\put(105,63){\colorbox{gray}}
\put(110,63){\colorbox{gray}}
\put(115,63){\colorbox{gray}}
\put(120,63){\colorbox{gray}}
\put(124,63){\colorbox{gray}}
\put(80,68){\colorbox{gray}}
\put(85,68){\colorbox{gray}}
\put(90,68){\colorbox{gray}}
\put(95,68){\colorbox{gray}}
\put(100,68){\colorbox{gray}}
\put(105,68){\colorbox{gray}}
\put(110,68){\colorbox{gray}}
\put(115,68){\colorbox{gray}}
\put(120,68){\colorbox{gray}}
\put(124,68){\colorbox{gray}}
\put(80,72){\colorbox{gray}}
\put(85,72){\colorbox{gray}}
\put(90,72){\colorbox{gray}}
\put(95,72){\colorbox{gray}}
\put(100,72){\colorbox{gray}}
\put(105,72){\colorbox{gray}}
\put(110,72){\colorbox{gray}}
\put(115,72){\colorbox{gray}}
\put(120,72){\colorbox{gray}}
\put(124,72){\colorbox{gray}}

\put(130,63){\colorbox{gray}}
\put(135,63){\colorbox{gray}}
\put(140,63){\colorbox{gray}}
\put(145,63){\colorbox{gray}}
\put(150,63){\colorbox{gray}}
\put(155,63){\colorbox{gray}}
\put(159,63){\colorbox{gray}}
\put(130,68){\colorbox{gray}}
\put(135,68){\colorbox{gray}}
\put(140,68){\colorbox{gray}}
\put(145,68){\colorbox{gray}}
\put(150,68){\colorbox{gray}}
\put(155,68){\colorbox{gray}}
\put(159,68){\colorbox{gray}}
\put(130,72){\colorbox{gray}}
\put(135,72){\colorbox{gray}}
\put(140,72){\colorbox{gray}}
\put(145,72){\colorbox{gray}}
\put(150,72){\colorbox{gray}}
\put(155,72){\colorbox{gray}}
\put(159,72){\colorbox{gray}}

\put(165,63){\colorbox{gray}}
\put(170,63){\colorbox{gray}}
\put(175,63){\colorbox{gray}}
\put(180,63){\colorbox{gray}}
\put(185,63){\colorbox{gray}}
\put(190,63){\colorbox{gray}}
\put(195,63){\colorbox{gray}}
\put(200,63){\colorbox{gray}}
\put(205,63){\colorbox{gray}}
\put(210,63){\colorbox{gray}}
\put(214,63){\colorbox{gray}}
\put(165,68){\colorbox{gray}}
\put(170,68){\colorbox{gray}}
\put(175,68){\colorbox{gray}}
\put(180,68){\colorbox{gray}}
\put(185,68){\colorbox{gray}}
\put(190,68){\colorbox{gray}}
\put(195,68){\colorbox{gray}}
\put(200,68){\colorbox{gray}}
\put(205,68){\colorbox{gray}}
\put(210,68){\colorbox{gray}}
\put(214,68){\colorbox{gray}}
\put(165,72){\colorbox{gray}}
\put(170,72){\colorbox{gray}}
\put(175,72){\colorbox{gray}}
\put(180,72){\colorbox{gray}}
\put(185,72){\colorbox{gray}}
\put(190,72){\colorbox{gray}}
\put(195,72){\colorbox{gray}}
\put(200,72){\colorbox{gray}}
\put(205,72){\colorbox{gray}}
\put(210,72){\colorbox{gray}}
\put(214,72){\colorbox{gray}}

\put(165,48){\colorbox{gray}}
\put(170,48){\colorbox{gray}}
\put(175,48){\colorbox{gray}}
\put(180,48){\colorbox{gray}}
\put(185,48){\colorbox{gray}}
\put(190,48){\colorbox{gray}}
\put(195,48){\colorbox{gray}}
\put(200,48){\colorbox{gray}}
\put(205,48){\colorbox{gray}}
\put(210,48){\colorbox{gray}}
\put(214,48){\colorbox{gray}}
\put(165,53){\colorbox{gray}}
\put(170,53){\colorbox{gray}}
\put(175,53){\colorbox{gray}}
\put(180,53){\colorbox{gray}}
\put(185,53){\colorbox{gray}}
\put(190,53){\colorbox{gray}}
\put(195,53){\colorbox{gray}}
\put(200,53){\colorbox{gray}}
\put(205,53){\colorbox{gray}}
\put(210,53){\colorbox{gray}}
\put(214,53){\colorbox{gray}}
\put(165,57){\colorbox{gray}}
\put(170,57){\colorbox{gray}}
\put(175,57){\colorbox{gray}}
\put(180,57){\colorbox{gray}}
\put(185,57){\colorbox{gray}}
\put(190,57){\colorbox{gray}}
\put(195,57){\colorbox{gray}}
\put(200,57){\colorbox{gray}}
\put(205,57){\colorbox{gray}}
\put(210,57){\colorbox{gray}}
\put(214,57){\colorbox{gray}}

\put(0,150){\framebox(20,15){\tiny $(1,1)$}} 

\put(23,155){$\cdots$}
\put(23,135){$\ddots$}

\put(40,150){\framebox(20,15){\tiny $(1,i)$}}
\put(48,135){$\vdots$}
\put(40,115){\framebox(20,15){\tiny $(r,r)$}}

\put(63,155){$\cdots$}
\put(63,119){$\cdots$}
\put(63,89){$\ddots$}

\put(80,150){\framebox(50,15){\tiny $(1,n-1)$}}
\put(104,135){$\vdots$}
\put(80,115){\framebox(50,15){\tiny $(r,n-1)$}}
\put(80,100){\framebox(50,15){\tiny $(r+1,n-1)$}}
\put(104,82){$\vdots$}
\put(80,60){\framebox(50,15){\tiny $(n-1,n-1)$}}

\put(130,150){\framebox(35,15){\tiny $(1,n)$}}
\put(146,135){$\vdots$}
\put(130,115){\framebox(35,15){\tiny $(r,n)$}}
\put(130,100){\framebox(35,15){\tiny $(r+1,n)$}}
\put(146,82){$\vdots$}
\put(130,60){\framebox(35,15){\tiny $(n-1,n)$}}

\put(165,150){\framebox(55,15){\tiny $(n,2n-1)$}}
\put(189,135){$\vdots$}
\put(165,115){\framebox(55,15){\tiny $(n,2n-r)$}}
\put(165,100){\framebox(55,15){\tiny $(n,2n-1-r)$}}
\put(189,82){$\vdots$}
\put(165,60){\framebox(55,15){\tiny $(n,n+1)$}}
\put(165,45){\framebox(55,15){\tiny $(n,n)$}}

\put(220,150){\framebox(50,15){\tiny $(1,n+1)$}}
\put(244,135){$\vdots$}
\put(220,115){\framebox(50,15){\tiny $(r,n+1)$}}
\put(220,100){\framebox(50,15){\tiny $(r+1,n+1)$}}
\put(244,88){$\vdots$}

\put(318,155){$\cdots$}
\put(298,120){$\cdots$}
\put(273,104){$\cdots$}
\put(273,87){$\cdot$} 
\put(278,89.5){$\cdot$}
\put(283,92){$\cdot$}

\put(290,100){\framebox(45,15){\tiny $(r+1,d_{r+1})$}}

\put(335,115){\framebox(25,15){\tiny $(r,d_r)$}}

\put(363,135){$\cdot$} 
\put(368,137.5){$\cdot$}
\put(373,140){$\cdot$}

\put(380,150){\framebox(25,15){\tiny $(1,d_1)$}}
\end{picture}
\end{center}
\vspace{-50pt}
\caption{The Hessenberg function $h^{(n,2n-1-r)}$ with $0 \leq r \leq n-1$.}
\label{picture:h^{(i,j)}Case4}
\end{figure} 

Similarly to Case~2, one can see that $\psi^{D_n}_{n,2n-1-r}(\alpha) \equiv 0$ mod $\alpha$ for any $\alpha \in I^{(i,j)}$.

Therefore, it follows from Cases~1,~2,~3,~4 that $\psi^{D_n}_{i,j}$ is an element of $D(\A_{I^{(i,j)}})$, and this completes the proof. 
\end{proof}

\begin{theorem} \label{theorem:psiD}
The derivations $\{\psi^{D_n}_{i,j} \mid 1 \leq i \leq n-1, i \leq j \leq 2n-i-1\} \cup \{\psi^{D_n}_{n,j} \mid n \leq j \leq 2n-1 \}$ form uniform bases for the ideal arrangements of type $D_n$.
\end{theorem}

\begin{proof}
From Proposition~\ref{proposition:psiD} we have $\{\psi^{D_n}_{i,h_I(i)} \mid i \in [n] \} \subset D(\A_I)$ for any lower ideal $I$.
Since $\deg(\psi^{D_n}_{i,j})=j-i$, it is enough to prove that $\psi^{D_n}_{1,h(1)}, \ldots, \psi^{D_n}_{n,h(n)}$ are linearly independent over $\CR$ for any Hessenberg function $h$ for type $D_n$ by Theorem~\ref{theorem:Saito's criterion}. 
We prove this by splitting into two cases.
Let 
\begin{equation} \label{eq:iDnh}
i^{D_n}_h:=\min\{ i \in [n-1] \mid h(i) \geq n-1 \}. 
\end{equation}

\smallskip

\noindent
\textit{Case~1} Suppose that $h(n) \geq 2n-i^{D_n}_h-1$.
We prove the linear independence of $\psi^{D_n}_{1,h(1)}, \ldots, \psi^{D_n}_{n,h(n)}$ over $\CR$ by induction on $n$. 
As the base case $n=2$, the claim is straightforward.
Now we assume that $n > 2$ and the claim holds for $n-1$.
We first note that a Hessenberg function $h$ for type $D_n$ induces the Hessenberg function $h'$ for type $D_{n-1}$ defined by 
\begin{align*} 
h'(i-1)&=h(i)-1 \ \ \ {\rm for} \ 2 \leq i \leq n-1, \\
h'(n-1)&=\min \{h(n)-1, 2n-3 \}.
\end{align*}
The picture of $h'$ is obtained from that of $h$ by removing the $1$-st row.
One see that $i^{D_{n-1}}_{h'} \geq i^{D_n}_h$, so we have $h'(n-1) \geq 2(n-1)-i^{D_{n-1}}_{h'}-1$.
That is, $h'$ is a Hessenberg function for type $D_{n-1}$ satisfying the condition for Case~1.

We consider the following surjection  
\begin{equation} \label{eq:Dmapbar}
\begin{array}{ccc}
\bigoplus_{k=1}^n \R[x_1,\ldots,x_n] \partial_k & \stackrel{}{\longrightarrow} & \bigoplus_{k=2}^n \R[x_2,\ldots,x_n] \partial_k \\
\rotatebox{90}{$\in$} & & \rotatebox{90}{$\in$} \\
\sum_{k=1}^n f_k(x_1,\ldots,x_n)\partial_k & \longmapsto & \sum_{k=2}^n f_k(0,x_2,\ldots,x_n)\partial_k
\end{array}
\end{equation}
and the image of $\psi$ under the map is denoted by $\overline{\psi}$.
Then we have 
\begin{align} 
\overline{\psi^{D_n}_{i,j}}&=\psi^{D_{n-1}}_{i-1,j-1}(x_2,\ldots,x_n) \ \ \ \ \ \ {\rm for} \ i \in [n-1], i \leq j \leq 2n-1-i, \label{eq:barpsiD1} \\
\overline{\psi^{D_n}_{n,j}}&=\begin{cases}
\psi^{D_{n-1}}_{n-1,j-1}(x_2,\ldots,x_n) \ \ \ {\rm if} \ i \in [n-1], i \leq j \leq 2n-1-i, \\
0 \ \ \ \ \ \ \ \ \ \ \ \ \ \ \ \ \ \ \ \ \ \ \ \ \ \ \ {\rm if} \ j=2n-1. \label{eq:barpsiD2}
\end{cases}
\end{align}
Suppose that 
\begin{equation} \label{eq:psiDlinearlyindependent}
\sum_{k=1}^n f_k(x_1,\ldots,x_n) \psi^{D_n}_{k,h(k)}=0
\end{equation} 
for $f_k(x_1,\ldots,x_n) \in \R[x_1,\ldots,x_n]$ such that all $f_k(x_1,\ldots,x_n)$ $(1 \leq k \leq n)$ have no common divisor.  

\noindent
\textit{Case~{1-1}} Assume that $h(1) < n-1$. Then we have $\psi^{D_{n-1}}_{0,h(1)-1}=0$ and $h(n) < 2n-1$ by the definition of the Hessenberg function.
Applying the map in \eqref{eq:Dmapbar} to both sides of \eqref{eq:psiDlinearlyindependent}, we obtain from \eqref{eq:barpsiD1} and \eqref{eq:barpsiD2}
$$
\sum_{k=2}^n f_k(0,x_2\ldots,x_n) \psi^{D_{n-1}}_{k-1,h'(k-1)}(x_2,\ldots,x_n)=0.
$$
By the inductive assumption, we obtain $f_k(0,x_2,\ldots,x_n)=0$ for $k=2,\ldots,n$.
Hence, we can write $f_k(x_1,\ldots,x_n)=x_1 f'_k(x_1,\ldots,x_n)$ for some $f'_k(x_1,\ldots,x_n) \in \R[x_1,\ldots,x_n]$.
Comparing the coefficient of $\partial_1$ in \eqref{eq:psiDlinearlyindependent}, we have 
$$
f_1(x_1,\ldots,x_n)(x_1-x_2)\cdots(x_1-x_{h(1)})+x_1(f'_2(x_1,\ldots,x_n)+\cdots)=0.
$$
Substituting $x_1=0$ into the equation above, we obtain
$$
f_1(0,x_2,\ldots,x_n)(-x_2)\cdots(-x_{h(1)})=0.
$$
Hence, we have $f_1(0,x_2,\ldots,x_n)=0$ which implies $f_1(x_1,\ldots,x_n)$ is divisible by $x_1$.
However, all $f_k(x_1,\ldots,x_n)$ have no common divisor, so $f_k(x_1,\ldots,x_n)=0$ for $k=1,\ldots,n$.

\noindent
\textit{Case~{1-2}} Assume that $h(1) = n-1$. Then we have 
\begin{equation} \label{eq:psiDn-10,h(1)-1}
\psi^{D_{n-1}}_{0,h(1)-1}(x_2,\ldots,x_n)=(-1)^{n-1}\sum_{k=2}^n x_2 \cdots \widehat{x_{k}} \cdots x_n \, \partial_k=(-1)^{n-1} \psi^{D_{n-1}}_{n-1,2n-3}(x_2,\ldots,x_n).
\end{equation}
Since $h(1)=n-1$, we have $i^{D_n}_h=1$.
Thus, $h(n) \geq 2n-i^{D_n}_h-1=2n-2$, i.e. $h(n)=2n-1$ or $h(n)=2n-2$.

If $h(n)=2n-1$, then $\overline{\psi^{D_n}_{n,h(n)}}=0$ from \eqref{eq:barpsiD2} and $h'(n-1)=2n-3$.
Applying the map in \eqref{eq:Dmapbar} to both sides of \eqref{eq:psiDlinearlyindependent}, we obtain from \eqref{eq:psiDn-10,h(1)-1} and \eqref{eq:barpsiD1}
$$
(-1)^{n-1} f_1(0,x_2\ldots,x_n) \psi^{D_{n-1}}_{n-1,h'(n-1)}(x_2,\ldots,x_n)+\sum_{k=2}^{n-1} f_k(0,x_2\ldots,x_n) \psi^{D_{n-1}}_{k-1,h'(k-1)}(x_2,\ldots,x_n)=0.
$$
By the inductive assumption, we obtain $f_k(0,x_2,\ldots,x_n)=0$ for $k=1,\ldots,n-1$.
Hence, we can write $f_k(x_1,\ldots,x_n)=x_1 f'_k(x_1,\ldots,x_n)$ for some $f'_k(x_1,\ldots,x_n) \in \R[x_1,\ldots,x_n]$.
Comparing the coefficient of $\partial_1$ in \eqref{eq:psiDlinearlyindependent}, we have 
$$
x_1(f'_1(x_1,\ldots,x_n)+\cdots)+f_n(x_1,\ldots,x_n) x_2\cdots x_n=0.
$$
Substituting $x_1=0$ into the equation above, we obtain $f_n(0,x_2,\ldots,x_n)=0$ which implies $f_n(x_1,\ldots,x_n)$ is divisible by $x_1$.
Therefore, $f_k(x_1,\ldots,x_n)=0$ for $k=1,\ldots,n$ since all $f_k(x_1,\ldots,x_n)$ have no common divisor. 

If $h(n)=2n-2$, then $h'(n-1)=2n-3$.
Hence by \eqref{eq:barpsiD2}, $\overline{\psi^{D_n}_{n,h(n)}}=\psi^{D_{n-1}}_{n-1,h'(n-1)}(x_2,\ldots,x_n)$. 
Applying the map in \eqref{eq:Dmapbar} to both sides of \eqref{eq:psiDlinearlyindependent}, we obtain from \eqref{eq:psiDn-10,h(1)-1} and \eqref{eq:barpsiD1}
\begin{align*}
& \big((-1)^{n-1} f_1(0,x_2\ldots,x_n)+f_n(0,x_2\ldots,x_n) \big)\psi^{D_{n-1}}_{n-1,h'(n-1)}(x_2,\ldots,x_n) \\
&+\sum_{k=2}^{n-1} f_k(0,x_2\ldots,x_n) \psi^{D_{n-1}}_{k-1,h'(k-1)}(x_2,\ldots,x_n)=0.
\end{align*}
By the inductive assumption, we obtain
\begin{align} 
&f_1(0,x_2,\ldots,x_n)=(-1)^n f_n(0,x_2,\ldots,x_n), \label{eq:f1=fnD} \\
&f_k(0,x_2,\ldots,x_n)=0 \ \ \ {\rm for} \ k=2,\ldots,n-1. \label{eq:fkDcase1-2}
\end{align}
Hence, we can write $f_k(x_1,\ldots,x_n)=x_1 f'_k(x_1,\ldots,x_n)$ for $k=2,\ldots,n-1$.
Comparing the coefficient of $\partial_1$ in \eqref{eq:psiDlinearlyindependent},  
\begin{align*}
&f_1(x_1,\ldots,x_n) \big( (x_1-x_2)\cdots(x_1-x_{n-1})(x_1+x_n)+(-1)^{n-1}x_2\cdots x_n \big)x_1^{-1} \\
&+x_1(f'_2(x_1,\ldots,x_n)+\cdots)+f_n(x_1,\ldots,x_n)\big( (-1)^{n}(x_1-x_2)\cdots(x_1-x_n)+x_2\cdots x_n \big)x_1^{-1}=0.
\end{align*}
Substituting $x_1=0$ into the equation above, we obtain 
\begin{align*}
&f_1(0,x_2,\ldots,x_n) \big( \sum_{k=2}^{n-1} (-x_2)\cdots\widehat{(-x_k)}\cdots(-x_{n-1})(x_n) +(-x_2)\cdots(-x_{n-1}) \big) \\
&+(-1)^n f_n(0,x_2,\ldots,x_n) \big( \sum_{k=2}^n (-x_2)\cdots\widehat{(-x_k)}\cdots(-x_n) \big)=0.
\end{align*}
Hence by \eqref{eq:f1=fnD}, we have $2f_1(0,x_2,\ldots,x_n)(-x_2)\cdots(-x_{n-1})=0$ which implies that $f_1(0,x_2,\ldots,x_n)=f_n(0,x_2,\ldots,x_n)=0$.
From this together with \eqref{eq:fkDcase1-2}, we conclude $f_k(x_1,\ldots,x_n)=0$ for $k=1,\ldots,n$ because all $f_k(x_1,\ldots,x_n)$ have no common divisor. 

\noindent
\textit{Case~{1-3}} Assume that $h(1) \geq n$. 
It follows from similar discussion on Case~{1-2} that $f_k(x_1,\ldots,x_n)=0$ for $k=1,\ldots,n$.

\smallskip

\noindent
\textit{Case~2} Suppose that $h(n) \leq 2n-i^{D_n}_h-1$.
We prove the linear independence of $\psi^{D_n}_{1,h(1)}, \ldots, \psi^{D_n}_{n,h(n)}$ over $\CR$ by  induction on $m_h:=2n-i^{D_n}_h-1-h(n)$. 
As the base case $m_h=0$, namely $h(n)=2n-i^{D_n}_h-1$, the claim follows from Case~1.
Now we assume that $m_h>0$, that is, 
\begin{equation} \label{eq:psiDninductionCase2}
h(n) < 2n-i^{D_n}_h-1
\end{equation}
and the claim holds for any Hessenberg function $\tilde{h}$ for type $D_n$ with $m_{\tilde{h}} = m_h-1$. 

For the given Hessenberg function $h$ with \eqref{eq:psiDninductionCase2}, we can define a Hessenberg function $\tilde{h}$ by 
\begin{equation*}
\tilde{h}(i)=\begin{cases}
h(i) \ \ \ \ \ \ \ \ \ {\rm for} \ i=1,\dots,n-1, \\
h(n)+1 \ \ \ {\rm for} \ i=n.
\end{cases}
\end{equation*}
By the definition we have $i^{D_n}_h=i^{D_n}_{\tilde{h}}$, so $m_{\tilde{h}}=m_h-1$.
Hence, $\psi^{D_n}_{1,\tilde{h}(1)},\ldots,\psi^{D_n}_{n,\tilde{h}(n)}$ are linearly independent over $\CR$ by the inductive assumption. 
Suppose that 
\begin{equation} \label{eq:typeDCase2linearlyindependence1}
\sum_{k=1}^n f_k \psi^{D_n}_{k,h(k)}=0
\end{equation}
for $f_k \in \R[x_1,\ldots,x_n]$.
Since $\psi^{D_n}_{n,\tilde{h}(n)}=\alpha_{n,\tilde{h}(n)} \psi^{D_n}_{n,h(n)}+(-1)^{\tilde{h}(n)-n}\psi^{D_n}_{2n-\tilde{h}(n),n}$ by \eqref{eq:psiDn}, we have 
\begin{equation} \label{eq:typeDCase2linearlyindependence2}
\sum_{k=1}^{n-1} g_k \psi^{D_n}_{k,\tilde{h}(k)} + f_n \psi^{D_n}_{n,\tilde{h}(n)}=(-1)^{\tilde{h}(n)-n} f_n \psi^{D_n}_{2n-\tilde{h}(n),n},
\end{equation}
where $g_k=\alpha_{n,\tilde{h}(n)} f_k$ for $k \in [n-1]$.
We show that $\psi^{D_n}_{2n-\tilde{h}(n),n}$ can be written as a linear combination of $\psi^{D_n}_{1,\tilde{h}(1)},\ldots,\psi^{D_n}_{n-1,\tilde{h}(n-1)}$ over $\CR$. 
We first note that $1 \leq 2n- \tilde{h}(n) \leq n$ because $n \leq \tilde{h}(n) \leq 2n-1$ by the definition $(2)$ of a Hessenberg function. 
If $2n-\tilde{h}(n)=n$, then $n \leq h(n)=\tilde{h}(n)-1=n-1$, yielding a contradiction. 
If $2n-\tilde{h}(n)=1$, then $h(n)=\tilde{h}(n)-1=2n-2$. 
The assumption \eqref{eq:psiDninductionCase2} now implies $i^{D_n}_h<1$, a contradiction. 
Hence, we have 
\begin{align}
&2 \leq 2n- \tilde{h}(n) \leq n-1, \label{eq:inequality2n-h(n)} \\
&1 \leq 2n- \tilde{h}(n)-1 \leq n-2. \label{eq:inequality2n-h(n)-1}
\end{align}
By the definition \eqref{eq:iDnh} and assumption \eqref{eq:psiDninductionCase2}, we have $h(2n- \tilde{h}(n)) \geq h(i^{D_n}_h) \geq n-1$ and $h(2n- \tilde{h}(n)-1) \geq h(i^{D_n}_h) \geq n-1$.
Also, by the definition $(5)$ of a Hessenberg function we have $h(2n- \tilde{h}(n)) \leq n$ and $h(2n- \tilde{h}(n)-1) \leq n$.
Thus, both of the values $h(2n- \tilde{h}(n))$ and $h(2n- \tilde{h}(n)-1)$ are equal to $n-1$ or $n$.

If $h(2n- \tilde{h}(n))=n$, then it follows from \eqref{eq:inequality2n-h(n)} and the definition of $\tilde{h}$ that
$$
\psi^{D_n}_{2n-\tilde{h}(n),n}=\psi^{D_n}_{2n-\tilde{h}(n),h(2n- \tilde{h}(n))}=\psi^{D_n}_{2n-\tilde{h}(n),\tilde{h}(2n- \tilde{h}(n))}.
$$

If $h(2n- \tilde{h}(n))=n-1$, then we have $h(2n- \tilde{h}(n)-1)=n-1$ by the definition $(3)$ of a Hessenberg function.
From \eqref{eq:inequality2n-h(n)}, \eqref{eq:inequality2n-h(n)-1}, and the definition of $\tilde{h}$ together with \eqref{eq:psiDi}, we obtain
\begin{align*}
\psi^{D_n}_{2n-\tilde{h}(n),n} &= \psi^{D_n}_{2n-\tilde{h}(n)-1,n-1}+\alpha_{i,n} \psi^{D_n}_{2n-\tilde{h}(n),n-1}= \psi^{D_n}_{2n-\tilde{h}(n)-1,h(2n- \tilde{h}(n)-1)}+\alpha_{i,n} \psi^{D_n}_{2n-\tilde{h}(n),h(2n- \tilde{h}(n))} \\
&= \psi^{D_n}_{2n-\tilde{h}(n)-1,\tilde{h}(2n- \tilde{h}(n)-1)}+\alpha_{i,n} \psi^{D_n}_{2n-\tilde{h}(n),\tilde{h}(2n- \tilde{h}(n))}.
\end{align*}
Therefore, $\psi^{D_n}_{2n-\tilde{h}(n),n}$ can be written as a linear combination of $\psi^{D_n}_{1,\tilde{h}(1)},\ldots,\psi^{D_n}_{n-1,\tilde{h}(n-1)}$ over $\CR$. 
This together with \eqref{eq:typeDCase2linearlyindependence2} implies $f_n=0$ because $\psi^{D_n}_{1,\tilde{h}(1)},\ldots,\psi^{D_n}_{n,\tilde{h}(n)}$ are linearly independent over $\CR$ by the inductive assumption. 
Hence by \eqref{eq:typeDCase2linearlyindependence1}, we have $\sum_{k=1}^{n-1} f_k \psi^{D_n}_{k,h(k)}=0$. 
However, since $\psi^{D_n}_{k,h(k)}=\psi^{D_n}_{k,\tilde{h}(k)}$ for $1 \leq k \leq n-1$, from the $\CR$-linearly independence of $\psi^{D_n}_{1,\tilde{h}(1)},\ldots,\psi^{D_n}_{n-1,\tilde{h}(n-1)}$ we obtain $f_k=0$ for $1 \leq k \leq n-1$.

Therefore, we proved that $\psi^{D_n}_{1,h(1)},\ldots,\psi^{D_n}_{n,h(n)}$ are linearly independent over $\CR$ for any Hessenberg function $h$ for type $D_n$ by Cases~1 and 2, and this completes the proof. 
\end{proof}

The derivations $\psi^{D_n}_{i,j}$ are uniform bases and have an explicit formula. 
Unfortunately, they are not of the form $\eqref{eq:main1-2}$ because of the form of the derivations $\psi_{0,j}$ for $n \leq j \leq 2n-3$. 
In order to give uniform bases $\tilde\psi^{D_n}_{i,j}$ of the form $\eqref{eq:main1-2}$, we need to modify $\psi^{D_n}_{i,j}$ a little bit.

As the base case, when $j=i$, we define
\begin{align*}
\tilde\psi^{D_n}_{i,i}=\psi^{D_n}_{i,i} \ \ \ {\rm for} \ i \in [n]. 
\end{align*}
Proceeding inductively, we define for $i=1$ 
\begin{align*}
\tilde\psi^{D_n}_{1,j}&=\alpha_{1,j}\tilde\psi^{D_n}_{1,j-1} \ \ \ \ \ \ {\rm if} \ 1< j \leq 2n-2 \ {\rm with} \ j \neq n-1,n, \\
\tilde\psi^{D_n}_{1,n-1}&=\alpha_{1,n-1}\tilde\psi^{D_n}_{1,n-2} + \xi^{D_n}_{1}, \\
\tilde\psi^{D_n}_{1,n}&=\alpha_{1,n}\tilde\psi^{D_n}_{1,n-1} + \xi^{D_n}_{0}, 
\end{align*}
and for $1 < i \leq n-1$
\begin{align*}
\tilde\psi^{D_n}_{i,j}&=\tilde\psi^{D_n}_{i-1,j-1}+\alpha_{i,j}\tilde\psi^{D_n}_{i,j-1} \ \ \ \ \ \ {\rm if} \ i< j \ {\rm with} \ j \neq n-1, \\
\tilde\psi^{D_n}_{i,n-1}&=\tilde\psi^{D_n}_{i-1,n-2}+\alpha_{i,n-1}\tilde\psi^{D_n}_{i,n-2} + \xi^{D_n}_{i} \ \ \ \ \ \ {\rm if} \ i< n-1, 
\end{align*}
and for $i=n$
\begin{align*}
\tilde\psi^{D_n}_{n,j}&=\alpha_{n,j}\tilde\psi^{D_n}_{n,j-1}+(-1)^{j-n}\tilde\psi^{D_n}_{2n-j,n} \ \ \ \ \ \ {\rm if} \ \ n+1 \leq j \leq 2n-1. 
\end{align*}

\begin{lemma} \label{lemma:PsiTildepsi}
For all $(i,j)$, we have 
$$
\tilde\psi^{D_n}_{i,j} \equiv \psi^{D_n}_{i,j} \ \ \ {\rm mod} \ D(\A_{\Phi^+_{D_n}}).
$$ 
In particular, $\{\tilde\psi^{D_n}_{i,h_I(i)} \mid i\in[n] \} \subset D(\A_I)$ for any lower ideal $I \subset \Phi^+_{D_n}$.
\end{lemma}

\begin{proof}
Since $\psi_{0,n-1}$ is an element of $D(\A_{\Phi^+_{D_n}})$, so is $\psi_{0,j}$ for $n \leq j \leq 2n-3$. 
From this together with the recursive formulas for $\{\psi_{i,j} \}$ and $\{\tilde\psi_{i,j} \}$, we obtain that $\tilde\psi_{i,j} \equiv \psi_{i,j}$ mod $D(\A_{\Phi^+_{D_n}})$ for all $(i,j)$.
For the rest, it follows from Proposition~\ref{proposition:psiD} and $D(\A_{\Phi^+_{D_n}}) \subset D(\A_I)$.
\end{proof}

\begin{lemma} \label{lemma:xiD}
The derivation $\xi^{D_n}_{i}$ in \eqref{eq:xiD} holds
\begin{align*}
\xi^{D_n}_{i}&=-\frac{1}{2}\alpha_{i+1,n} \psi^{D_n}_{i+1,n-1} + (-1)^{n-i} \frac{1}{2}\alpha_{n,2n-1-i}\psi^{D_n}_{n,2n-2-i} \\
&=-\frac{1}{2}\alpha_{i+1,n} \tilde\psi^{D_n}_{i+1,n-1} + (-1)^{n-i} \frac{1}{2}\alpha_{n,2n-1-i}\tilde\psi^{D_n}_{n,2n-2-i}
\end{align*}
for $0 \leq i \leq n-2$.
\end{lemma}

\begin{proof}
For the second equality, it follows from the recursive formulas for $\{\psi_{i,j}\}$ and $\{\tilde\psi_{i,j}\}$ that $\psi^{D_n}_{i+1,n-1}=\tilde\psi^{D_n}_{i+1,n-1}$ and $\psi^{D_n}_{n,2n-2-i}=\tilde\psi^{D_n}_{n,2n-2-i}$.

We show the first equality.
From the formulas \eqref{eq:psi2} and \eqref{eq:psi4} we have  
\begin{align*}
\psi^{D_n}_{i+1,n-1}=&\sum_{k=1}^{i+1} \left((x_k-x_{i+2})\cdots(x_k-x_{n-1})(x_k+x_n) +(-1)^{n-i-1} x_{i+2}\cdots x_n \right) x_k^{-1}\partial_k \\
&+(-1)^{n-i-1} \sum_{k=i+2}^n x_{i+2} \cdots \widehat{x_k} \cdots x_n \partial_k, \\
\psi^{D_n}_{n,2n-2-i}=&\sum_{k=1}^{i+1} \left((-1)^{n-i} (x_k-x_{i+2})\cdots(x_k-x_{n-1})(x_k-x_n) + x_{i+2}\cdots x_n \right) x_k^{-1}\partial_k \\
&+ \sum_{k=i+2}^n x_{i+2} \cdots \widehat{x_k} \cdots x_n \partial_k. 
\end{align*}
Noting that $\alpha_{i+1,n}=x_{i+1}-x_n$ and $\alpha_{n,2n-1-i}=x_{i+1}+x_n$, we obtain 
\begin{align*}
&-\alpha_{i+1,n} \psi^{D_n}_{i+1,n-1} + (-1)^{n-i} \alpha_{n,2n-1-i}\psi^{D_n}_{n,2n-2-i} \\
=&\sum_{k=1}^{i} \big((x_k-x_{i+2})\cdots(x_k-x_{n-1})x_k \cdot 2x_n -(x_k-x_{i+2})\cdots(x_k-x_{n-1})x_n \cdot 2x_{i+1} \\
&+(-1)^{n-i} x_{i+2}\cdots x_n \cdot 2x_{i+1} \big) x_k^{-1}\partial_k+\big( (-1)^{n-i} x_{i+2}\cdots x_n \cdot 2x_{i+1} \big)x_{i+1}^{-1}\partial_{i+1} \\
&+(-1)^{n-i} \sum_{k=i+2}^n x_{i+2} \cdots \widehat{x_k} \cdots x_n \cdot 2x_{i+1}\partial_k \\
=&\sum_{k=1}^{i} \big(2(x_k-x_{i+1})(x_k-x_{i+2})\cdots(x_k-x_{n-1})x_n +(-1)^{n-i} 2x_{i+1} x_{i+2}\cdots x_n \big) x_k^{-1}\partial_k \\
&+(-1)^{n-i} 2\big(\sum_{k=i+1}^n x_{i+1} \cdots \widehat{x_k} \cdots x_n \partial_k \big) \\
&=2 \xi^{D_n}_{i}
\end{align*}
as desired.
\end{proof}

For type $D_n$ we have $\height(\Phi^+_{D_n})=2n-3$.
Also, the set $\Lambda_m$ in \eqref{eq:Lambdam} for $m=2k,2k+1$ is given by 
\begin{equation} \label{eq:LambdaD}
\Lambda_m=\begin{cases}
[n-k-1]\cup\{n\} \ \ \ &{\rm if} \ 1 \leq m \leq n-1, \\
[n-k-1] \ \ \ &{\rm if} \ n \leq m \leq 2n-3.
\end{cases}
\end{equation}

\begin{proposition} \label{proposition:tildepsi}
The derivations $\{\tilde\psi^{D_n}_{i,i+m} \mid 0 \leq m \leq 2n-3, i \in \Lambda_m \}$ have the following expressions.
As the base case $m=0$, 
\begin{align*}
\tilde\psi^{D_n}_{i,i}&=\alpha_i^* \ \ \ \ {\rm for} \ i=1,\ldots,n-2, \\
\tilde\psi^{D_n}_{i,i}&=2\alpha_i^* \ \ \ {\rm for} \ i=n-1,n. 
\end{align*}
For $m$ with $1 \leq m \leq n-1$, 
\begin{align*}
\tilde\psi^{D_n}_{i,i+m}=&\sum _{j=1}^{i} \alpha_{j,j+m} \tilde\psi^{D_n}_{j,j+m-1} \ \ \ {\rm for} \ 1 \leq i < n-m-1, \\
\tilde\psi^{D_n}_{i,i+m}=&\sum _{j=1}^{n-m-1} \alpha_{j,j+m} \tilde\psi^{D_n}_{j,j+m-1} -\frac{1}{2}\alpha_{n-m,n}\tilde\psi^{D_n}_{n-m,n-1}+(-1)^{m+1}\frac{1}{2}\alpha_{n,n+m}\tilde\psi^{D_n}_{n,n+m-1} \ \ \ \\
& \ \ \ \ \ \ \ \ \ \ \ \ \ \ \ \ \ \ \ \ \ \ \ \ \ \ \ \ \ \ \ \ \ \ \ \ {\rm for} \ i = n-m-1, \\ 
\tilde\psi^{D_n}_{i,i+m}=&\sum _{j=1 \atop j \neq n-m}^{i} \alpha_{j,j+m} \tilde\psi^{D_n}_{j,j+m-1} +\frac{1}{2}\alpha_{n-m,n}\tilde\psi^{D_n}_{n-m,n-1}+(-1)^{m+1}\frac{1}{2}\alpha_{n,n+m}\tilde\psi^{D_n}_{n,n+m-1} \ \ \ \\ 
& \ \ \ \ \ \ \ \ \ \ \ \ \ \ \ \ \ \ \ \ \ \ \ \ \ \ \ \ \ \ \ \ \ \ \ \ {\rm for} \ n-m \leq i \leq n-1, \\
\tilde\psi^{D_n}_{n,n+m}=&(-1)^{m}\left(\sum _{j=1}^{n-m-1} \alpha_{j,j+m} \tilde\psi^{D_n}_{j,j+m-1} +\frac{1}{2}\alpha_{n-m,n}\tilde\psi^{D_n}_{n-m,n-1}\right)+\frac{1}{2}\alpha_{n,n+m}\tilde\psi^{D_n}_{n,n+m-1}. 
\end{align*}
For $m$ with $n \leq m \leq 2n-3$, we have
\begin{align*}
\tilde\psi^{D_n}_{i,i+m}&=\sum _{j=1}^{i} \alpha_{j,j+m} \tilde\psi^{D_n}_{j,j+m-1}.
\end{align*}
\end{proposition}

\begin{proof}
When $m=0$, it is clear that $\tilde\psi^{D_n}_{i,i}=\alpha_i^*$ for $i \in [n-2]$ and $\tilde\psi^{D_n}_{i,i}=2\alpha_i^*$ for $i=n-1,n$. 
We prove the claim for $m>0$ by induction on $i$.
As the base case $i=1$, if $1 < j \leq 2n-2$ with $j \neq n-1,n$, then  
\begin{align*}
\tilde\psi^{D_n}_{1,j}=\alpha_{1,j}\tilde\psi^{D_n}_{1,j-1}.
\end{align*}
If $j=n-1$, then we have 
\begin{align*}
\tilde\psi^{D_n}_{1,n-1}&=\alpha_{1,n-1}\tilde\psi^{D_n}_{1,n-2}+\xi^{D_n}_{1} =\alpha_{1,n-1}\tilde\psi^{D_n}_{1,n-2}-\frac{1}{2}\alpha_{2,n}\tilde\psi^{D_n}_{2,n-1}+(-1)^{n-1}\frac{1}{2}\alpha_{n,2n-2}\tilde\psi^{D_n}_{n,2n-3}
\end{align*}
by Lemma~\ref{lemma:xiD}.
Similarly, if $j=n$, then 
\begin{align*}
\tilde\psi^{D_n}_{1,n}=\alpha_{1,n}\tilde\psi^{D_n}_{1,n-1}+\xi^{D_n}_{0} =\frac{1}{2}\alpha_{1,n}\tilde\psi^{D_n}_{1,n-1}+(-1)^{n}\frac{1}{2}\alpha_{n,2n-1}\tilde\psi^{D_n}_{n,2n-2}.
\end{align*}
Now we assume that $i > 1$ and the claim holds for $i-1$.
We first consider the case $1 \leq m \leq n-1$. 
If $1 \leq i < n-m-1$, then by the inductive assumption we have
$$
\tilde\psi^{D_n}_{i,i+m}= \tilde\psi^{D_n}_{i-1,i+m-1} + \alpha_{i,j}\tilde\psi^{D_n}_{i,i+m-1}=\sum _{j=1}^{i} \alpha_{j,j+m} \tilde\psi^{D_n}_{j,j+m-1}.
$$ 
If $i = n-m-1$, then from Lemma~\ref{lemma:xiD}
\begin{align*}
\tilde\psi^{D_n}_{i,i+m}&= \tilde\psi^{D_n}_{i-1,i+m-1} + \alpha_{i,i+m}\tilde\psi^{D_n}_{i,i+m-1}+\xi^{D_n}_{i}=\tilde\psi^{D_n}_{n-m-2,n-2} + \alpha_{n-m-1,n-1}\tilde\psi^{D_n}_{n-m-1,n-2}+\xi^{D_n}_{n-m-1} \\
&=\sum _{j=1}^{n-m-1} \alpha_{j,j+m} \tilde\psi^{D_n}_{j,j+m-1}+\xi^{D_n}_{n-m-1} \\
&=\sum _{j=1}^{n-m-1} \alpha_{j,j+m} \tilde\psi^{D_n}_{j,j+m-1}-\frac{1}{2}\alpha_{n-m,n}\tilde\psi_{n-m,n-1}+(-1)^{m+1}\frac{1}{2}\alpha_{n,n+m}\tilde\psi_{n,n+m-1}. 
\end{align*}
If $i=n-m$, then we have
\begin{align*}
\tilde\psi^{D_n}_{i,i+m}&= \tilde\psi^{D_n}_{i-1,i+m-1} + \alpha_{i,i+m}\tilde\psi^{D_n}_{i,i+m-1}=\tilde\psi^{D_n}_{n-m-1,n-1} + \alpha_{n-m,n}\tilde\psi^{D_n}_{n-m,n-1} \\
&=\sum _{j=1}^{n-m-1} \alpha_{j,j+m} \tilde\psi^{D_n}_{j,j+m-1} -\frac{1}{2}\alpha_{n-m,n}\tilde\psi^{D_n}_{n-m,n-1}+(-1)^{m+1}\frac{1}{2}\alpha_{n,n+m}\tilde\psi^{D_n}_{n,n+m-1} + \alpha_{n-m,n}\tilde\psi^{D_n}_{n-m,n-1} \\
&=\sum _{j=1}^{n-m-1} \alpha_{j,j+m} \tilde\psi^{D_n}_{j,j+m-1} +\frac{1}{2}\alpha_{n-m,n}\tilde\psi^{D_n}_{n-m,n-1}+(-1)^{m+1}\frac{1}{2}\alpha_{n,n+m}\tilde\psi^{D_n}_{n,n+m-1}.
\end{align*}
If $n-m < i \leq n-1$, then by the inductive assumption
\begin{align*}
\tilde\psi^{D_n}_{i,i+m}&= \tilde\psi^{D_n}_{i-1,i+m-1} + \alpha_{i,j}\tilde\psi^{D_n}_{i,i+m-1} \\
&=\sum _{j=1 \atop j \neq n-m}^{i-1} \alpha_{j,j+m} \tilde\psi^{D_n}_{j,j+m-1} +\frac{1}{2}\alpha_{n-m,n}\tilde\psi^{D_n}_{n-m,n-1}+(-1)^{m+1}\frac{1}{2}\alpha_{n,n+m}\tilde\psi^{D_n}_{n,n+m-1} + \alpha_{i,j}\tilde\psi^{D_n}_{i,i+m-1} \\
&=\sum _{j=1 \atop j \neq n-m}^{i} \alpha_{j,j+m} \tilde\psi^{D_n}_{j,j+m-1} +\frac{1}{2}\alpha_{n-m,n}\tilde\psi^{D_n}_{n-m,n-1}+(-1)^{m+1}\frac{1}{2}\alpha_{n,n+m}\tilde\psi^{D_n}_{n,n+m-1}.
\end{align*}
If $i=n$, then we have
\begin{align*}
\tilde\psi^{D_n}_{n,n+m}=&\alpha_{n,n+m}\tilde\psi^{D_n}_{n,n+m-1}+(-1)^{m}\tilde\psi^{D_n}_{n-m,n} \\ 
=&\alpha_{n,n+m}\tilde\psi^{D_n}_{n,n+m-1}+(-1)^{m} \big(\sum_{j=1}^{n-m-1} \alpha_{j,j+m}\tilde\psi^{D_n}_{j,j+m-1}+\frac{1}{2}\alpha_{n-m,n}\tilde\psi^{D_n}_{n-m,n-1} \\
&+(-1)^{m+1}\frac{1}{2}\alpha_{n,n+m}\tilde\psi^{D_n}_{n,n+m-1} \big) \\
=&(-1)^{m}\left(\sum _{j=1}^{n-m-1} \alpha_{j,j+m} \tilde\psi^{D_n}_{j,j+m-1} +\frac{1}{2}\alpha_{n-m,n}\tilde\psi^{D_n}_{n-m,n-1}\right)+\frac{1}{2}\alpha_{n,n+m}\tilde\psi^{D_n}_{n,n+m-1}. 
\end{align*}

Next we consider the case $n \leq m \leq 2n-3$.
Then, by the inductive assumption we have
\begin{align*}
\tilde\psi^{D_n}_{i,i+m}&=\tilde\psi^{D_n}_{i-1,i+m-1}+ \alpha_{i,i+m} \tilde\psi^{D_n}_{i,i+m-1} \\
&=\sum _{j=1}^{i} \alpha_{j,j+m} \tilde\psi^{D_n}_{j,j+m-1}.
\end{align*}

Therefore, we proved the desired equalities.
\end{proof}

Motivated by Proposition~\ref{proposition:tildepsi}, we define a matrix $P_m^{D_n}$ of size $|\Lambda_m|=\lambda_m$ as follows.
As the base case, when $m=0$, 
{\footnotesize
$$
P_0^{D_n}=
\begin{pmatrix}
1        &           &    &  & \\
         & \ddots &     &   & \\
         &           &  1 &   & \\
         &           &     & 2 & \\
         &           &     &   & 2
\end{pmatrix}. 
$$}
Let $1 \leq m < n-1$. Depending on whether $m$ is odd or even, we define $P_m^{D_n}$ by
{\footnotesize
$$
P_{2k+1}^{D_n}=
\begin{array}{rcccccccccl}
\ldelim({11}{1ex}[]&1       &           &           &           &              &          &             &            &    & \rdelim){11}{1ex}[] \\
							&\vdots & \ddots &           &           &             &             &             &            &    &\\
							&1        & \cdots & 1         &           &                     &              &             &            &  & \\
							&1        & \cdots & 1         &    1      &  -\frac{1}{2}    &  0          &  \cdots  & 0         &  \frac{1}{2}& \\
							&1        & \cdots & 1         &    1      &    \frac{1}{2}    &  0         &   \cdots &  0         & \frac{1}{2} & \\
							&1        & \cdots & 1         &    1      &    \frac{1}{2}    &  1          &  \ddots &  \vdots  & \vdots & \\
							&\vdots &           & \vdots & \vdots &     \vdots        &  \vdots   &  \ddots  & 0          & \vdots  & \\
							&1        & \cdots & 1        & 1        &     \frac{1}{2}     &  1          &  \cdots  & 1          & \frac{1}{2}&  \\
							&-1       & \cdots & -1      & -1       &    -\frac{1}{2}    &  0         &   \cdots  & 0         & \frac{1}{2}&\\
							&          \multicolumn{3}{c}{\raisebox{2ex}[1ex][0ex]{$\underbrace{\hspace{13ex}}_{{n-2k-3}}$}}	&&&\multicolumn{3}{c}{\raisebox{2ex}[1ex][0ex]{$\underbrace{\hspace{11ex}}_{{k}}$}}&&\\
\end{array}, \ \ 
P_{2k+2}^{D_n}=
\begin{array}{rcccccccccl}
\ldelim({11}{1ex}[]&1        &           &           &           &                      &             &             &            &   & \rdelim){11}{1ex}[] \\
							&\vdots & \ddots &           &           &                      &             &             &            &    &\\
							&1        & \cdots & 1         &           &                     &              &             &            &   &\\
							&1        & \cdots & 1         &    1      &  -\frac{1}{2}    &  0          &  \cdots  & 0         &  -\frac{1}{2}& \\
							&1        & \cdots & 1         &    1      &    \frac{1}{2}    &  0         &   \cdots &  0         & -\frac{1}{2} & \\
							&1        & \cdots & 1         &    1      &    \frac{1}{2}    &  1          &  \ddots &  \vdots  & \vdots  &\\
							&\vdots &           & \vdots & \vdots &     \vdots        &  \vdots   &  \ddots  & 0          & \vdots   &\\
							&1        & \cdots & 1        & 1        &     \frac{1}{2}     &  1          &  \cdots  & 1          & -\frac{1}{2} & \\
							&1       & \cdots & 1      & 1       &    \frac{1}{2}    &  0         &   \cdots  & 0         & \frac{1}{2}&\\
							&         \multicolumn{3}{c}{\raisebox{2ex}[1ex][0ex]{$\underbrace{\hspace{10ex}}_{{n-2k-4}}$}}		     & &&\multicolumn{3}{c}{\raisebox{2ex}[1ex][0ex]{$\underbrace{\hspace{11ex}}_{{k}}$}}&&
\end{array}. \ \ 
$$} 
Let $m=n-1$. Depending on whether $m$ is odd or even, we define 
{\footnotesize
$$
P_{2k+1}^{D_n}=
\begin{array}{rcccccl}
\ldelim({6}{1ex}[]&\frac{1}{2}    &  0         &   \cdots &  0         & \frac{1}{2}& \rdelim){6}{1ex}[]  \\
							&\frac{1}{2}    &  1          &  \ddots &  \vdots  & \vdots  &\\
							&\vdots        &  \vdots   &  \ddots  & 0          & \vdots  & \\
							&\frac{1}{2}     &  1          &  \cdots  & 1          & \frac{1}{2}  &\\
							&-\frac{1}{2}    &  0         &   \cdots  & 0         & \frac{1}{2}&\\
							&                  &\multicolumn{3}{c}{\raisebox{2ex}[1ex][0ex]{$\underbrace{\hspace{11ex}}_{{k}}$}}&&
\end{array}, \ \ 
P_{2k+2}^{D_n}=
\begin{array}{rcccccl}
\ldelim({6}{1ex}[]&\frac{1}{2}    &  0         &   \cdots &  0         & -\frac{1}{2}& \rdelim){6}{1ex}[]  \\
							&\frac{1}{2}    &  1          &  \ddots &  \vdots  & \vdots&  \\
							&\vdots        &  \vdots   &  \ddots  & 0          & \vdots &  \\
							&\frac{1}{2}     &  1          &  \cdots  & 1          & -\frac{1}{2} & \\
							&\frac{1}{2}    &  0         &   \cdots  & 0         & \frac{1}{2}&\\
							&                  &\multicolumn{3}{c}{\raisebox{2ex}[1ex][0ex]{$\underbrace{\hspace{11ex}}_{k}$}}&&
\end{array}. \ \ 
$$}
For $m$ with $n \leq m \leq 2n-3$, we define
{\footnotesize
$$
P_m^{D_n}=
\begin{pmatrix}
1        &           &            \\
\vdots & \ddots &            \\
1        & \cdots & 1            
\end{pmatrix}.
$$}
From Proposition~\ref{proposition:tildepsi} we obtain 
\begin{equation} \label{eq:PsiPmD}
[\tilde\psi^{D_n}_{i,i+m}]_{i \in \Lambda_m}=P_m[\alpha_{i,i+m}\tilde\psi^{D_n}_{i,i+m-1}]_{i \in \Lambda_m}
\end{equation}
for any $m$ with $1 \leq m \leq 2n-3$ where we think of indices for rows and columns of the matrix $P_m$ as the set $\Lambda_m$ in \eqref{eq:LambdaD}, and we arrange them as in increasing order. 
One can see that $\det P_m=1$ if $m \neq n-1$ and $\det P_{n-1}=\frac{1}{2}$, so we obtain $P_m \in \GL(\Lambda_m,\Q)$ for all $1 \leq m \leq 2n-3$.
From this together with Lemma~\ref{lemma:PsiTildepsi}, we obtain the following theorem
by Proposition~\ref{proposition:key} (see also Remark~\ref{remark:key}).

\begin{theorem} \label{theorem:tildepsiD}
The derivations $\{\tilde\psi^{D_n}_{i,i+m} \mid 0 \leq m \leq 2n-3, i \in \Lambda_m \}$ form uniform bases for the lower ideals of type $D_n$.
Furthermore, $P_m^{D_n}$ $(0 \leq m \leq 2n-3)$ are the invertible matrices associated with the uniform bases $\{\tilde\psi^{D_n}_{i,i+m} \mid \ 0 \leq m \leq 2n-1, i \in \Lambda_m \}$. 
\end{theorem}

\section{The cohomology rings of regular nilpotent Hessenberg varieties}
\label{section:Hessenberg varieties}

The logarithmic derivation modules $D(\A_I)$ for the lower ideals $I$ are related with the cohomology rings of regular nilpotent Hessenberg varieties by the work of \cite{AHMMS}. 
In type $A$, an explicit presentation of the cohomology rings of regular nilpotent Hessenberg varieties is given in \cite{AHHM} by using the localization technique. 
By the work of \cite{AHMMS} one can obtain an explicit presentation of the cohomology rings of regular nilpotent Hessenberg varieties in types $A,B,C,G$ by using uniform bases.
In this section we first explain the work of \cite{AHMMS}. 
Then we describe an explicit presentation of the cohomology rings of regular nilpotent Hessenberg varieties in all Lie types.
Throughout this paper, all cohomology rings will be taken with real coefficients.

Let $G$ be a semisimple linear algebraic group of rank $n$.  
We fix a Borel subgroup $B$ of $G$ and a maximal torus $T$ included in $B$. 
The Lie algebras of $G$ and $B$ are denoted by $\mathfrak{g}$ and $\mathfrak{b}$, respectively.
A \textbf{Hessenberg space} $H$ is a $\mathfrak{b}$-submodule of $\mathfrak{g}$ containing $\mathfrak{b}$.  
One can see that there is one-to-one correspondence between the set of lower ideals and the set of Hessenberg spaces which sends $I \subset \Phi^+$ by
$$
H(I):=\mathfrak{b}\oplus \big(\bigoplus_{\alpha \in I} \mathfrak{g}_{-\alpha} \big),
$$ 
where $\mathfrak{g}_{\alpha}$ is the root space for a root $\alpha$. 
The \textbf{Hessenberg variety} $\Hess(X,I)$ associated with an element $X\in \mathfrak{g}$ and a lower ideal $I \subset \Phi^+$ is defined as the following subvariety of the flag variety $G/B$: 
\begin{equation*} 
\Hess(X,I):=\{g B \in G/B \mid \mbox{Ad}(g^{-1})(X) \in H(I)\}.
\end{equation*}

In what follows, we concentrate on Hessenberg varieties $\Hess(N,I)$ for a regular nilpotent element $N$ and we call them \textbf{regular nilpotent Hessenberg varieties}.
Here, we recall that an element $X\in \mathfrak{g}$ is {nilpotent} if $\mbox{ad}(X)$ is nilpotent, i.e.,\ $\mbox{ad}(X)^k=0$ for some $k>0$.
An element $X\in \mathfrak{g}$ is {regular} if its $G$-orbit of the adjoint action has the largest possible dimension. 

Let $\hat T$ be the character group of $T$. 
Since any $\alpha \in \hat T$ extends to a character of $B$, $\alpha$ defines a complex line bundle $L_\alpha:=G \times_B \C_{\alpha}$ where $\C_{\alpha}$ is the one-dimensional $B$-module via $\alpha$.
Here, $L_\alpha$ is the quotient of the product $G \times \C$ by the right $B$-action given by $(g,z) \cdot b  = (gb, \alpha(b)z)$ for $b \in B$ and $(g,z) \in G \times \C$.
We may regard $\hat T$ as an additive group so that $\hat T \otimes_{\Z} \R$ is identified with the dual space $\mathfrak{t}^*$ of the Lie algebra of the maximal compact torus.
To each $\alpha \in \mathfrak{t}^*$ we can assign the Euler class $e(L_{\alpha}) \in H^2(G/B)$. This assignment induces a ring homomorphism 
\begin{equation} \label{eq:varphi}
\varphi\colon \CR=\mbox{Sym}(\mathfrak{t}^*) \to H^*(G/B)
\end{equation}
which doubles the grading on $\CR$.
From the well-known result by Borel \cite{B}, the map $\varphi$ is surjective and its kernel is the ideal generated by $W$-invariants in $\CR$ with zero constant term.
Here, $W$ is the Weyl group.
Composing $\varphi$ with the restriction map $H^*(G/B) \to H^*(\Hess(N,I))$, we have a ring homomorphism
\begin{equation} \label{eq:varphiNilpotent}
\varphi_I\colon \CR\to H^*(\Hess(N,I)).
\end{equation}
The map $\varphi_I$ is surjective from the result of \cite{AHMMS} (Theorem~\ref{theorem:AHMMS} below).
Moreover, its kernel can be described in terms of the logarithmic derivation module $D(\A_I)$.
We can identify $\mathfrak{t}$ and $\mathfrak{t}^*$ via the Killing form which implies the isomorphism $\CR\otimes\mathfrak{t} \cong \CR\otimes\mathfrak{t}^*$.
Composing the isomorphism with the multiplication map $\CR\otimes\mathfrak{t}^*\to \CR$, we obtain an $\CR$-module map
\begin{equation*} 
\Quadraticmap: \Der\CR=\CR\otimes\mathfrak{t} \cong \CR\otimes\mathfrak{t}^*\to \CR.
\end{equation*}
We define an ideal $\mathfrak{a}(I)$ as the image of the logarithmic derivation module $D(\A_I)$ under the map $\Quadraticmap$:
\begin{equation*} 
\mathfrak{a}(I):=\Quadraticmap(D(\A_I)).
\end{equation*}

\begin{theorem}[\cite{AHMMS}] \label{theorem:AHMMS}
The map $\varphi_I$ in \eqref{eq:varphiNilpotent} is surjective and its kernel coincides with the ideal $\mathfrak{a}(I)$.
In particular, $\varphi_I$ induces the ring isomorphism 
$$
H^*(\Hess(N,I)) \cong \CR/\mathfrak{a}(I).
$$
\end{theorem}

From Theorem~\ref{theorem:AHMMS} together with explicit uniform bases for the lower ideals, we obtain an explicit presentation of the cohomology rings of regular nilpotent Hessenberg varieties in all Lie types.
In fact, \cite{AHMMS} derived the explicit presentation for types $A_{n-1},B_n,C_n,G_2$ from uniform bases (\cite[Corollary~10.4, Corollary~10.10, Corollary~10.15, Corollary~10.18]{AHMMS}).
We now give an explicit presentation of the cohomology rings of regular nilpotent Hessenberg varieties for all Lie types in terms of the invertible matrices associated with uniform bases.

Let $\varpi_1, \ldots, \varpi_n$ be the fundamental weights. 
In what follows,  $L$ means either of classical or exceptional Lie types.
We computed the invertible matrices $P^L_m$ for type $L$ in Section~\ref{section:typeD} and \cite{AHMMS}.
(See Appendix and \cite{AHMMS} for the exceptional types.)
Using their results, we define inductively polynomials $f^{L}_{i,i+m} \in \CR$ for $0 \leq m \leq \height(\Phi^+)$ and $i \in \Lambda_m$ as follows: 
\begin{align*}
[f^{L}_{i,i}]_{i \in \Lambda_0}&=\tilde{P_0^{L}}[\varpi_i]_{i \in \Lambda_0},\\
[f^{L}_{i,i+m}]_{i \in \Lambda_m}&=P_m^{L}[\alpha^L_{i,i+m}\psi^{L}_{i,i+m-1}]_{i \in \Lambda_m} \ \ \ {\rm for} \ 1 \leq m \leq \height(\Phi^+),
\end{align*}
where $\tilde{P^L_0}$ is the diagonal matrix with entries $\frac{2p_1}{||\alpha_1||^2},\ldots,\frac{2p_n}{||\alpha_n||^2}$ and $p_1,\ldots,p_n$ are the entries of the diagonal matrix $P^L_0$.
Noting that $\Quadraticmap(\alpha_i^*)=\frac{2}{||\alpha_i||^2} \, \varpi_i$, we have
\begin{equation} \label{fLpsiL}
f^L_{i,i+m}=\Quadraticmap(\psi^L_{i,i+m}) \ \ \ {\rm for} \ 0 \leq m \leq \height(\Phi^+) \ {\rm and} \ i \in \Lambda_m. 
\end{equation}
We remark that the equality \eqref{fLpsiL} is regarded as $f^{D_n}_{i,i+m}=\Quadraticmap(\tilde\psi^{D_n}_{i,i+m})$ for only type $L=D_n$. 

Let $h$ be a Hessenberg function associated to a lower ideal $I \subset \Phi_{L}^+$.
Then we denote a regular nilpotent Hessenberg variety $\Hess(N,I)$ by $\Hess(N,h)$.
Since $\{\psi^L_{i,i+m} \mid 0 \leq m \leq \height(\Phi^+), i \in \Lambda_m \}$ forms uniform bases for the lower ideals of type $L$, we obtain the explicit presentation of the cohomology rings of regular nilpotent Hessenberg varieties in all Lie types from Theorem~\ref{theorem:AHMMS} and \eqref{fLpsiL}.

\begin{corollary} \label{corollary:cohomologyHess}
Let $h$ be a Hessenberg function for type $L$ and $\Hess(N,h)$ the associated regular nilpotent Hessenberg variety for type $L$.
Then, the following ring isomorphism holds 
\begin{equation*} 
H^*(\Hess(N,h)) \cong \CR/(f^L_{1,h(1)},\ldots,f^L_{n,h(n)}).
\end{equation*}
\end{corollary}

\begin{remark}
The regular nilpotent Hessenberg variety $\Hess(N,h)$ for the special case when $h(i)=i+1$ for $i \in [n]$ is called the \textbf{Peterson variety}, denoted by $Pet$. 
An explicit presentation of $H^*(Pet)$ is given by \cite{HHM} as follows: 
\begin{equation*} 
H^*(Pet) \cong \CR/(\alpha_i \varpi_i \mid 1 \leq i \leq n).
\end{equation*}
One can see that Corollary~\ref{corollary:cohomologyHess} generalizes the explicit presentation above. 
\end{remark}

The polynomials $f^L_{i,j}$ have the explicit formula for types $L=A_{n-1},B_n,C_n,G_2$ by \cite{AHHM, AHMMS}.
For type $D_n$ we also have an explicit formula for $g^{D_n}_{i,j}:=\Quadraticmap(\psi^{D_n}_{i,j})$ because $\Quadraticmap$ sends $\partial_i$ to $x_i$.
That is, for $1 \leq i \leq n-1$ we have 
\begin{align*}
g^{D_n}_{i,j}=&\sum_{k=1}^{i} (x_k-x_{i+1})\cdots(x_k-x_j) x_k \ \ \ \ \ \ {\rm for} \ i \leq j \leq n-2 \ (i \neq n-1), \\
g^{D_n}_{i,n-1}=&\sum_{k=1}^{i} \big((x_k-x_{i+1})\cdots(x_k-x_{n-1})(x_k+x_n) \big) +(-1)^{n-i} n \, x_{i+1} \cdots x_n, \\
g^{D_n}_{i,n+j}=&\sum_{k=1}^{i} \big((x_k-x_{i+1})\cdots(x_k-x_{n})(x_k+x_n)\cdots(x_k+x_{n-j}) \big)  \\
&+(-1)^{n-i+1} n \, x_{i+1}\cdots x_{n-1-j} x_{n-j}^2\cdots x_n^2  \ \ \ \ \ \ {\rm for} \ 0 \leq j \leq n-1-i. 
\end{align*}
We also have
\begin{align*}
g^{D_n}_{n,2n-1-r}=&\sum_{k=1}^{r} \left((-1)^{n-r+1} (x_k-x_{r+1})\cdots(x_k-x_{n-1})(x_k-x_n) \right)  + n \, x_{r+1}\cdots x_n 
\end{align*}
for $0 \leq r \leq n-1$.

From Theorems~\ref{theorem:psiD} and \ref{theorem:AHMMS} we obtain the following corollary.

\begin{corollary} \label{corollary:fijD}
Let $h$ be a Hessenberg function for type $D_n$ and $\Hess(N,h)$ the associated regular nilpotent Hessenberg variety for type $D_n$.
Then, the following ring isomorphism holds
\begin{equation*}
H^*(\Hess(N,h)) \cong \R[x_1,\ldots,x_n]/(g^{D_n}_{1,h(1)},\ldots,g^{D_n}_{n,h(n)}).
\end{equation*}
\end{corollary}

\smallskip

\appendix

\section{Uniform bases for type $F$} 
\label{appendix:typeF}

Let $\mathfrak{t}$ be the Euclidean space $V=\R^4$ and we have 
\[
\CR = \Sym(\mathfrak{t}^*)=\R[x_1,x_2,x_3,x_4].
\]
We set the exponents $\e_1,\e_2, \e_3, \e_4$ as 
\[
\e_1=1, \ \e_2=11, \ \e_3=7, \ \e_4=5.
\]
We take the simple roots $\alpha_1=\frac{1}{2}(x_1-x_2-x_3-x_4), \alpha_2=x_2-x_3, \alpha_3=x_3-x_4, \alpha_4=x_4$ so that a labeling of the Dynkin diagram is shown in Figure~\ref{picture:DynkindiagramTypeF4}.
\begin{figure}[h]
\begin{center}
\begin{picture}(90,10)
\put(0,10){\circle{5}} 
\put(30,10){\circle{5}}
\put(60,10){\circle{5}}
\put(90,10){\circle{5}}

\put(2.3,10){\line(1,0){25}}
\put(32.3,11){\line(1,0){25}}
\put(32.3,9){\line(1,0){25}}
\put(62.3,10){\line(1,0){25}}

\qbezier(44,13)(50,10)(44,7)

\put(-2,0){{\tiny 2}} 
\put(28,0){{\tiny 3}}
\put(58,0){{\tiny 4}}
\put(88,0){{\tiny 1}}
\end{picture}
\end{center}
\vspace{-10pt}
\caption{Labeling of Dynkin diagram for type $F_4$.}
\label{picture:DynkindiagramTypeF4}
\end{figure} 

We arrange all positive roots of $F_4$ as shown in Figure~\ref{picture:PositiveRootTypeF}.

\begin{figure}[h]
\begin{center}
\begin{picture}(500,80)
\hspace{-10pt}
\put(0,60){\framebox(30,20){\tiny $x_2-x_3$}} 

\put(30,60){\framebox(30,20){\tiny $x_2-x_4$}}
\put(30,40){\framebox(30,20){\tiny $x_3-x_4$}}

\put(60,60){\framebox(15,20){\tiny $x_2$}}
\put(60,40){\framebox(15,20){\tiny $x_3$}}
\put(60,20){\framebox(15,20){\tiny $x_4$}}

\put(75,60){\dashbox(85,20){\tiny $\frac{1}{2}(x_1+x_2-x_3-x_4)$}}
\put(75,40){\dashbox(85,20){\tiny $\frac{1}{2}(x_1-x_2+x_3-x_4)$}}
\put(75,20){\dashbox(85,20){\tiny $\frac{1}{2}(x_1-x_2-x_3+x_4)$}}
\put(75,0){\framebox(85,20){\tiny $\frac{1}{2}(x_1-x_2-x_3-x_4)$}}

\put(160,60){\dashbox(30,20){\tiny $x_2+x_3$}}
\put(160,40){\dashbox(30,20){\tiny $x_2+x_4$}}
\put(160,20){\dashbox(30,20){\tiny $x_3+x_4$}}

\put(190,60){\framebox(85,20){\tiny $\frac{1}{2}(x_1+x_2+x_3-x_4)$}}
\put(190,40){\framebox(85,20){\tiny $\frac{1}{2}(x_1+x_2-x_3+x_4)$}}
\put(190,20){\framebox(85,20){\tiny $\frac{1}{2}(x_1-x_2+x_3+x_4)$}}

\put(275,60){\framebox(30,20){\tiny $x_1-x_4$}}
\put(275,40){\dashbox(30,20){\tiny $x_1-x_3$}}
\put(275,20){\framebox(30,20){\tiny $x_1-x_2$}}

\put(305,60){\framebox(85,20){\tiny $x_1$}}
\put(305,40){\dashbox(85,20){\tiny $\frac{1}{2}(x_1+x_2+x_3+x_4)$}}

\put(390,60){\framebox(30,20){\tiny $x_1+x_4$}}

\put(420,60){\framebox(30,20){\tiny $x_1+x_3$}}

\put(450,60){\framebox(30,20){\tiny $x_1+x_2$}}

\put(75,80){\line(1,0){115}}
\put(75,60){\line(1,0){115}}
\put(75,40){\line(1,0){115}}
\put(75,20){\line(1,0){115}}
\put(75,20){\line(0,1){60}}
\put(190,20){\line(0,1){60}}

\put(275,60){\line(1,0){115}}
\put(275,40){\line(1,0){115}}
\put(390,40){\line(0,1){20}}

\put(275,40){\line(0,1){20}}


\end{picture}
\end{center}
\vspace{-10pt}
\caption{The arrangement of all positive roots for type $F_4$.}
\label{picture:PositiveRootTypeF}
\end{figure} 

In Figure~\ref{picture:PositiveRootTypeF} the partial order $\preceq$ on $\Phi^+_{F_4}$ is defined as follows:

\begin{enumerate}
\item if a root $\alpha$ is left to a root $\beta$, then $\alpha \lessdot \beta$, except that $\alpha,\beta$ are divided by a dotted line;
\item if a root $\alpha$ is lower-adjacent to a root $\beta$, then $\alpha \lessdot \beta$;
\item if two roots $\alpha,\beta$ are divided by a dotted line, and a root $\gamma$ is immediately to the northwest of $\alpha$ and a root $\delta$ is immediately to the southeast of $\beta$, then $\alpha \lessdot \delta$ and $\gamma \lessdot \beta$. (See Figure~\ref{picture:positionsTypeF}.)
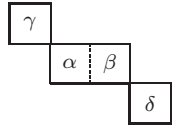
\smallskip
\begin{figure}[h]
\begin{center}
\begin{picture}(60,35)
\put(0,30){\framebox(15,15){\tiny $\gamma$}}
\put(15,15){\dashbox(15,15){\tiny $\alpha$}}
\put(30,15){\dashbox(15,15){\tiny $\beta$}}
\put(45,0){\framebox(15,15){\tiny $\delta$}}

\put(15,30){\line(1,0){30}}
\put(15,15){\line(1,0){30}}
\put(15,15){\line(0,1){15}}
\put(45,15){\line(0,1){15}}
\end{picture}
\end{center}
\vspace{-10pt}
\caption{The positions of $\alpha, \beta, \gamma, \delta$.}
\label{picture:positionsTypeF}
\end{figure}
\end{enumerate}
For two positive roots $\alpha,\beta$, we define $\alpha \preceq \beta$ if there exist positive roots $\gamma_0,\ldots,\gamma_N$ such that $\alpha=\gamma_0 \lessdot \gamma_1 \lessdot \cdots \lessdot \gamma_N=\beta$.

Now we fix a decomposition $\Phi^+_{F_4}=\Phi^+_1 \coprod \Phi^+_2 \coprod \Phi^+_3 \coprod \Phi^+_4$ satisfying \eqref{eq:decomposition1} and \eqref{eq:decomposition2}.
The positive roots $\alpha_{i,j}$ are defined in Table~\ref{table:positiverootsTypeF}.
{\tiny
\begin{table}[h]
\begin{tabular}{|c||c|c|c|c|c|c|c|c|c|c|} \hline
positive roots $\backslash$ $j$ & $2$ & $3$ & $4$ & $5$ & $6$ & $7$  \cr
\hline
\hline $\alpha_{1,j}$ & $\frac{1}{2}(x_1-x_2-x_3-x_4)$ &  &  &  &  &     \cr
\hline $\alpha_{2,j}$ &       & $x_2-x_3$ &  $x_2-x_4$  & $x_2$ & $x_2+x_4$ & $x_2+x_3$          \cr
\hline $\alpha_{3,j}$ &       &       & $x_3-x_4$ &   $x_3$   & $x_3+x_4$ &  $\frac{1}{2}(x_1-x_2+x_3+x_4)$          \cr
\hline $\alpha_{4,j}$ &       &       &       &   $x_4$    & $\frac{1}{2}(x_1-x_2-x_3+x_4)$ &  $\frac{1}{2}(x_1-x_2+x_3-x_4)$          \cr
\hline
\end{tabular} \\ \vspace{1.25ex}
\begin{tabular}{|c||c|c|c|c|c|c|c|c|c|c|} \hline
positive roots $\backslash$ $j$ & $8$ & $9$ & $10$ & $11$ & $12$ & $13$ \cr
\hline
\hline $\alpha_{1,j}$ &       &       &      &       &     &            \cr
\hline $\alpha_{2,j}$ &  $\frac{1}{2}(x_1+x_2+x_3-x_4)$ &  $\frac{1}{2}(x_1+x_2+x_3+x_4)$  &  $x_1$  &  $x_1+x_4$  &  $x_1+x_3$  &  $x_1+x_2$         \cr
\hline $\alpha_{3,j}$ &  $x_1-x_2$     & $x_1-x_3$ &  $x_1-x_4$    &      &    &            \cr
\hline $\alpha_{4,j}$ &   $\frac{1}{2}(x_1+x_2-x_3-x_4)$    & $\frac{1}{2}(x_1+x_2-x_3+x_4)$ &      &      &    &          \cr
\hline
\end{tabular} \vspace{1.25ex}
\caption{The positive roots $\alpha_{i,j}$ in type $F_4$.}
\label{table:positiverootsTypeF}
\end{table}
}

Motivated by this, we define the coordinate in type $F_4$ as shown in Figure~\ref{picture:CoordinateTypeF}.

\begin{figure}[h]
\begin{center}
\begin{picture}(350,70)
\put(0,60){\tiny \framebox(30,20){$(2,2)$}} 
\put(30,60){\tiny \framebox(30,20){$(2,3)$}} 
\put(60,60){\tiny \framebox(30,20){$(2,4)$}} 
\put(90,60){\tiny \framebox(30,20){$(2,5)$}} 
\put(120,60){\tiny \dashbox(30,20){$(4,8)$}} 
\put(150,60){\tiny \dashbox(30,20){$(2,7)$}} 
\put(180,60){\tiny \framebox(30,20){$(2,8)$}} 
\put(210,60){\tiny \framebox(30,20){$(3,10)$}} 
\put(240,60){\tiny \framebox(30,20){$(2,10)$}} 
\put(270,60){\tiny \framebox(30,20){$(2,11)$}} 
\put(300,60){\tiny \framebox(30,20){$(2,12)$}} 
\put(330,60){\tiny \framebox(30,20){$(2,13)$}} 

\put(30,40){\tiny \framebox(30,20){$(3,3)$}} 
\put(60,40){\tiny \framebox(30,20){$(3,4)$}} 
\put(90,40){\tiny \framebox(30,20){$(3,5)$}} 
\put(120,40){\tiny \dashbox(30,20){$(4,7)$}} 
\put(150,40){\tiny \dashbox(30,20){$(2,6)$}} 
\put(180,40){\tiny \framebox(30,20){$(4,9)$}} 
\put(210,40){\tiny \dashbox(30,20){$(3,9)$}} 
\put(240,40){\tiny \dashbox(30,20){$(2,9)$}} 

\put(60,20){\tiny \framebox(30,20){$(4,4)$}} 
\put(90,20){\tiny \framebox(30,20){$(4,5)$}} 
\put(120,20){\tiny \dashbox(30,20){$(4,6)$}} 
\put(150,20){\tiny \dashbox(30,20){$(3,6)$}} 
\put(180,20){\tiny \framebox(30,20){$(3,7)$}} 
\put(210,20){\tiny \framebox(30,20){$(3,8)$}} 

\put(90,0){\tiny \framebox(30,20){$(1,1)$}} 
\put(120,0){\tiny \framebox(30,20){$(1,2)$}} 

\put(120,80){\line(1,0){60}}
\put(120,60){\line(1,0){60}}
\put(120,40){\line(1,0){60}}
\put(120,20){\line(1,0){60}}
\put(120,20){\line(0,1){60}}
\put(180,20){\line(0,1){60}}

\put(210,60){\line(1,0){60}}
\put(210,40){\line(1,0){60}}
\put(210,40){\line(0,1){20}}
\put(270,40){\line(0,1){20}}
\end{picture}
\end{center}
\vspace{-10pt}
\caption{The coordinate in type $F_4$.}
\label{picture:CoordinateTypeF}
\end{figure} 

Note that $\height(\Phi^+_{F_4})=11$ and the set $\Lambda_m$ in \eqref{eq:Lambdam} is given by
Table~\ref{table:LambdaTypeF}.
\begin{table}[h]
\begin{tabular}{|c||c|c|c|c|} \hline
$m$ & $0,1$ & $2,3,4,5$ & $6,7$ & $8,9,10,11$ \cr
\hline
$\Lambda_m$ & $\{1,2,3,4\}$ & $\{2,3,4\}$ & $\{2,3\}$ & $\{2\}$      \cr
\hline 
\end{tabular} 
\vspace{5pt}
\caption{$\Lambda_m$ in type $F_4$.}
\label{table:LambdaTypeF}
\end{table}

For the decomposition $\Phi^+_{F_4}=\Phi^+_1 \coprod \Phi^+_2 \coprod \Phi^+_3 \coprod \Phi^+_4$ above, a \textbf{Hessenberg function for type $F_4$} is a function $h_I:\{1,\ldots,4\} \rightarrow \{1,\ldots,13\}$ defined in \eqref{eq:Hessft} associated to a lower ideal $I \subset \Phi^+_{F_4}$.

\begin{remark} \label{remark:HessenbergFunctionF4}
Note that Hessenberg functions $h$ for type $F_4$ are characterized by the following conditions
\begin{multicols}{2}
\begin{itemize}
\item[(1)] $i \leq h(i) \leq i+e_i$ for $i=1,2,3,4$, 
\item[(2)] if $h(2) \geq k$, then $h(3) \geq k$ for $k=4,5,6,10$, 
\item[(3)] if $h(3) \geq k$, then $h(4) \geq k$ for $k=5,7,9$, 
\item[(4)] if $h(4) \geq k$, then $h(1) =2$ for $k=6$, 
\item[(5)] if $h(4) \geq k$, then $h(3) \geq k-2$ for $k=7,9$, 
\item[(6)] if $h(4) \geq k$, then $h(2) \geq k-3$ for $k=8,9$, 
\item[(7)] if $h(2) \geq 8$, then $h(4) = 9$, 
\item[(8)] if $h(3) \geq 10$, then $h(2) \geq 8$. 
\end{itemize}
\end{multicols}
In fact, one can see that the set of lower ideals $I \subset \Phi^+_{F_4}$ and the set of functions $h$ satisfying the condition above are in one-to-one correspondence which sends $I$ to $h_I$.
\end{remark}

\begin{remark}
A Hessenberg function $h$ such that $h(1)=1, h(2) \leq 7, h(3) \leq 6, h(4) \leq 5$ for type $F_4$ is exactly the Hessenberg function $h$ such that $h(1)=1$ for type $B_4$ which is naturally identified with that of type $B_3$.
\end{remark}

We find uniform bases $\{\psi^{F_4}_{i,i+m} \in \Der \CR \mid 0 \leq m \leq 11, i \in \Lambda_m \}$ inductively.
As the base case, when $m=0$, we define 
\begin{equation*} 
\psi^{F_4}_{i,i}=\alpha_i^* \ \ \ \ \ {\rm for} \ i=1,2,3,4.
\end{equation*}
Explicitly,
\begin{align*} 
\psi^{F_4}_{1,1}=2\partial_1, \ \ 
\psi^{F_4}_{2,2}=\partial_1+\partial_2, \ \ 
\psi^{F_4}_{3,3}=2\partial_1+\partial_2+\partial_3, \ \ 
\psi^{F_4}_{4,4}=3\partial_1+\partial_2+\partial_3+\partial_4. 
\end{align*}
Proceeding inductively, for $m > 0$ and $i \in \Lambda_m$ we write
\begin{equation} \label{eq:psiF} 
\psi^{F_4}_{i,i+m}=\sum_{j \in \Lambda_{m}} p_{ij}^{(m)} \alpha_{j,j+m}\psi^{F_4}_{j,j+m-1}
\end{equation}
for some rational numbers $p_{ij}^{(m)}$.
We determine the rational numbers $p_{ij}^{(m)}$ such that all of the derivations $\psi_{i,i+m}$ form uniform bases.
As the proof of Proposition~\ref{proposition:psiD}, we consider the maximal lower ideal $I^{(i,i+m)}$ containing the root $\alpha_{i,i+m}$ with respect to the inclusion.
Since $\psi^{F_4}_{i,i+m}(\alpha) \in \CR \alpha$ for any $\alpha \in I^{(i,i+m)}$, by \eqref{eq:psiF} the following has to hold
$$
\sum_{j \in \Lambda_{m}} p_{ij}^{(m)} \alpha_{j,j+m}\psi^{F_4}_{j,j+m-1}(\alpha) \equiv 0 \ \ \ \mod \alpha.
$$
Now we know an explicit formula for $\psi^{F_4}_{j,j+m-1}$ ($j \in \Lambda_{m}$) by inductive step, so we obtain a linear equation in $p_{ij}^{(m)}$ ($j \in \Lambda_{m}$) for each $\alpha \in I^{(i,i+m)}$.
We computed a solution of the system of the linear equations in $p_{ij}^{(m)}$ ($j \in \Lambda_{m}$). (We also checked the solution by using Maple\footnote{The program is available at https://researchmap.jp/ehrhart/Database/.}.)
For each $m$ with $0 \leq m \leq 11=\height(\Phi^+_{F_4})$, the matrix $P_m^{F_4}=(p_{ij}^{(m)})_{i,j \in \Lambda_m}$ is described as Figure~\ref{picture:InvertibleMatricesTypeF4}.
{\small
\begin{figure}[h]
\begin{align*}
&P_0^{F_4}=
\begin{pmatrix}
1        & 0   & 0   & 0    \\
0        & 1   & 0   &  0   \\
0        &  0  & 1   & 0    \\
0        &  0  & 0   & 1  \\ 
\end{pmatrix}, \ \ 
P_1^{F_4}=
\begin{pmatrix}
1        & 1    & 1    & 1    \\
0        & 1    & 0    & 0    \\
0        & 1     & 1    & 0    \\
0        & 1     & 1    & 1    \\ 
\end{pmatrix}, \ \ 
P_2^{F_4}=
\begin{pmatrix}
1        & 0   & 0       \\
1        & 1   & 1       \\
0        & 0   & 1       \\
\end{pmatrix}, \ \ 
P_3^{F_4}=
\begin{pmatrix}
1                     & -\frac{1}{2}    & -1     \\
1                     & 1                & -1      \\
\frac{1}{2}         & \frac{1}{2}     & 1       \\
\end{pmatrix}, \\ 
&P_4^{F_4}=
\begin{pmatrix}
1        & 0   & -2      \\
0        & 1   & 2       \\
0        & 0   & 1       \\
\end{pmatrix}, \ \ 
P_5^{F_4}=
\begin{pmatrix}
1                       & 0                 &  0       \\
0                        & 1                 &  0       \\
 -\frac{1}{2}         &   \frac{1}{2}    & 1       \\
\end{pmatrix}, \ \ 
P_6^{F_4}=
\begin{pmatrix}
1        & 1     \\
0        & 1     \\
\end{pmatrix}, \ \  
P_7^{F_4}=
\begin{pmatrix}
1       & 0    \\
2        & 1   \\ 
\end{pmatrix}, \ \ 
P_m^{F_4}=
\begin{pmatrix}
1  \\      
\end{pmatrix} \ \ \ (8 \leq m \leq 11).
\end{align*}
\vspace{-10pt}
\caption{The invertible matrices for type $F_4$.}
\label{picture:InvertibleMatricesTypeF4}
\end{figure} 
}

Here, we arrange indices for rows and columns of the matrix $P^{F_4}_m$ as in increasing order. 
Note that 
\begin{align*}
[\psi^{F_4}_{i,i}]_{i \in \Lambda_0}&=P_0^{F_4}[\alpha_{i}^*]_{i \in \Lambda_0}\\
[\psi^{F_4}_{i,i+m}]_{i \in \Lambda_m}&=P_m^{F_4}[\alpha_{i,i+m}\psi^{F_4}_{i,i+m-1}]_{i \in \Lambda_m} \ \ \ {\rm for} \ 1 \leq m \leq 11.
\end{align*}
One can check that $\det P^{F_4}_m \neq 0$, so we obtain the following theorem
by Proposition~\ref{proposition:key} (see also Remark~\ref{remark:key}).

\begin{theorem} \label{theorem:psiF}
The derivations $\{\psi^{F_4}_{i,i+m} \mid 0 \leq m \leq 11, i \in \Lambda_m \}$ form uniform bases for the ideal arrangements of type $F_4$.
\end{theorem}

\section{Uniform bases for the ideal arrangements in a root subsystem}
\label{appendix:root subsystem}

In this section we prove that uniform bases for the ideal arrangements in a root subsystem of given a root system can be obtained from that of the given root system.

Recall that $\mathfrak{t}$ is a real Euclidean space of dimension $n$ and $\Phi \subset \mathfrak{t}^*$ an irreducible root system. 
Under the isomorphism $\mathfrak{t} \cong \mathfrak{t}^*$ induced from the inner product $( \ , \ )$ on $\mathfrak{t}$, the image of a root $\alpha \in \mathfrak{t}^*$ is denoted by $\hat\alpha \in \mathfrak{t}$.
We also denote $\check\beta \in \mathfrak{t}^*$ by the image of $\beta \in \mathfrak{t}$, that is, 
\begin{equation*}
\mathfrak{t} \cong \mathfrak{t}^*; \ \hat\alpha \mapsfrom \alpha, \ \beta \mapsto \check\beta.
\end{equation*}
Let $\alpha_1,\ldots, \alpha_n$ be the simple roots and its dual basis is denoted by $\beta_1,\ldots,\beta_n$ in this section.
Let $S$ be a nonempty subset of $[n]$ and $\mathfrak{t}'$ a subspace of $\mathfrak{t}$ spanned by $\hat\alpha_i$ for $i \in S$.
Note that $\mathfrak{t}'$ is orthogonal to $\beta_i$ for $i \notin S$:
$$
\mathfrak{t}'=\{x \in \mathfrak{t} \mid (\beta_i, x)=0 \ {\rm for \ all} \ i \notin S \}.
$$
The inner product on $\mathfrak{t}$ naturally induces that on $\mathfrak{t}'$.
Then, the isomorphisms $\mathfrak{t} \cong \mathfrak{t}^*$ and $\mathfrak{t}' \cong (\mathfrak{t}')^*$ via the inner products on $\mathfrak{t}$ and $\mathfrak{t}'$ respectively make the following commutative diagram:
\[\xymatrix{\mathfrak{t}\ar[r]^-{\cong} &\mathfrak{t}^*\ar@{->>}[d]\\ \mathfrak{t}'\ar@{^{(}->}[u]\ar[r]^-{\cong}&(\mathfrak{t}')^*}\]
Note that $(\mathfrak{t}')^*$ is isomorphic to the quotient space $\mathfrak{t}^*/{\rm span}\{\check{\beta_i} \mid i \notin S \}$.
We denote by $\overline{\alpha}$ the image of $\alpha \in \mathfrak{t}^*$ under the surjective map $\mathfrak{t}^* \to (\mathfrak{t}')^*$. 
We define $\Phi' \subset (\mathfrak{t}')^*$ as the image of a set $\{\alpha \in \Phi \mid \alpha(\beta_i) =0 \ {\rm for} \ i \notin S \}$ under the surjection $\mathfrak{t}^* \to (\mathfrak{t}')^*$.
Then, $\Phi'$ is a root system on $(\mathfrak{t}')^*$ and we can take $\{\overline{\alpha_i} \mid i \in S \}$ as the simple roots of $\Phi'$.

Let $\e_1,\ldots,\e_n$ (resp. $\e'_i \ (i \in S)$) be the exponents of the Weyl group $W$ (resp. $W'$) where $W$ and $W'$ are the Weyl groups associated with $(\mathfrak{t},\Phi)$ and $(\mathfrak{t}',\Phi')$ respectively.
A decomposition $\Phi^+=\coprod_{i=1}^n \ \Phi^+_i$ satisfying \eqref{eq:decomposition1} and \eqref{eq:decomposition2} induces the decomposition $\Phi'^+=\coprod_{i \in S} \ \Phi'^+_i$ satisfying \eqref{eq:decomposition1} and \eqref{eq:decomposition2} where $\Phi'^+_i$ is defined as $\Phi'^+_i=\{\overline{\alpha_{i,i+1}},\ldots, \overline{\alpha_{i,i+\e'_i}} \}$ for $i \in S$. 
Let $\CR=\Sym \mathfrak{t}^*$ and $\CR'=\Sym (\mathfrak{t}')^* \cong \CR/(\check{\beta_i} \mid i \notin S)$.
Let $\{\psi_{i,j} \in \Der \CR \mid i \in [n] \ {\rm and} \ i \leq j \leq \e_i \}$ be uniform bases for the ideal arrangements in $(\mathfrak{t},\Phi)$ of the form in Theorem~\ref{theorem:main1}.
We denote the invertible matrices associated with the uniform bases by $P_m$ for $0 \leq m \leq \height(\Phi^+)$. Namely, we can write 
\begin{align*}
\psi_{i,i}&=p_i \beta_i \ \ \ {\rm for} \ i \in [n], \\
\psi_{i,i+m}&=\sum_{j \in \Lambda_m} p_{ij} \alpha_{j,j+m}\psi_{j,j+m-1} \ \ \ {\rm for} \ 1 \leq m \leq \height(\Phi^+) \ {\rm and} \ i \in \Lambda_m,
\end{align*}
where $P_0=\diag(p_1,\ldots,p_n)$ and $P_m=(p_{ij})_{i,j \in \Lambda_m} \in \GL(\Lambda_m;\Q)$.
Then, we define $\{\psi'_{i,j} \in \Der \CR'=\CR' \otimes \mathfrak{t}' \mid i \in [n] \ {\rm and} \ i \leq j \leq \e_i \ \}$ as follows:
\begin{align*}
\psi'_{i,i}&=p_i \gamma_i \ \ \ {\rm for} \ i \in [n], \\
\psi'_{i,i+m}&=\sum_{j \in \Lambda_m} p_{ij} \overline{\alpha_{j,j+m}}\psi'_{j,j+m-1} \ \ \ {\rm for} \ 1 \leq m \leq \height(\Phi^+) \ {\rm and} \ i \in \Lambda_m,
\end{align*}
where $\{\gamma_i \in \mathfrak{t}' \mid i \in S\}$ is the dual basis of the simple roots $\{\overline{\alpha_i} \in (\mathfrak{t}')^* \mid i \in S\}$ for $\Phi'$ with the convention $\gamma_j=0$ whenever $j \notin S$.

\begin{lemma} \label{lemma:6-1}
Let $i$ and $j$ be positive integers with $1 \leq i \leq n$ and $i \leq j \leq \e_i$.
Then, 
$$
\psi'_{i,j}(\overline{\alpha})=\overline{\psi_{i,j}(\alpha)}
$$
for any $\overline{\alpha} \in \Phi'$. 
\end{lemma}

\begin{proof}
Since we can take $\{\overline{\alpha_k} \mid k \in S \}$ as the simple roots of $\Phi'$, it is enough to check that $\psi'_{i,j}(\overline{\alpha_k})=\overline{\psi_{i,j}(\alpha_k)}$ for any $k \in S$.
We prove this by induction on $m:=j-i$.
As the base case, when $m=0$, we have
$$
\psi'_{i,i}(\overline{\alpha_k})=p_i \gamma_i(\overline{\alpha_k})=p_i \delta_{ik}=\overline{\psi_{i,i}(\alpha_k)}. 
$$
Now we assume that $m>0$ and the claim holds for $m-1$.
Then, for any $i \in [n]$, we have
\begin{align*}
\psi'_{i,i+m}(\overline{\alpha_k})&=\sum_{j \in \Lambda_m} p_{ij} \overline{\alpha_{j,j+m}} \psi'_{j,j+m-1}(\overline{\alpha_k}) \\
&=\sum_{j \in \Lambda_m} p_{ij} \overline{\alpha_{j,j+m}} \overline{\psi_{j,j+m-1}(\alpha_k)} \ \ \ ({\rm by \ the \ inductive \ assumption})\\
&=\overline{\psi_{i,i+m}(\alpha_k)}.
\end{align*}
This completes the proof.
\end{proof}

\begin{proposition} \label{proposition:6-2}
A set of derivations $\{\psi'_{i,j} \in \Der \CR' \mid i \in S \ {\rm and} \ i \leq j \leq \e'_i \ \}$ forms uniform bases for the ideal arrangements in $(\mathfrak{t}',\Phi')$.
\end{proposition}

\begin{proof}
Let $I'$ be a lower ideal in $\Phi'^+$ and $h_{I'}: S \to \Z_{\geq 0}$ the associated Hessenberg function. 
We first show that $\psi'_{i,h_{I'}(i)} \in D(\A_{I'})$ for any $i \in S$.
Namely, we prove that $\psi'_{i,h_{I'}(i)}(\overline{\alpha}) \in \CR' \overline{\alpha}$ for any $\overline{\alpha} \in I'$.
Since we have $\psi'_{i,h_{I'}(i)}(\overline{\alpha})=\overline{\psi_{i,h_{I'}(i)}(\alpha)}$ by Lemma~\ref{lemma:6-1}, it is enough to prove that $\psi_{i,h_{I'}(i)}(\alpha) \in \CR \alpha$.
Let $I$ be a lower ideal in $\Phi^+$ defined as 
\begin{equation} \label{eq:6-1}
I=\{\alpha_{i,j} \mid i \in S \ {\rm and} \ \overline{\alpha_{i,j}} \in I' \}.
\end{equation}
Note that the Hessenberg function $h_I$ associated with $I$ is given by 
\begin{equation*}
h_I(i)=\begin{cases}
h_{I'}(i) & {\rm if} \ i \in S, \\
i         & {\rm if} \ i \notin S.
\end{cases}
\end{equation*}
Since $\{\psi_{i,j} \in \Der \CR \mid i \in [n] \ {\rm and} \ i \leq j \leq \e_i \}$ forms uniform bases for the ideal arrangements in $(\mathfrak{t},\Phi)$, we have $\psi_{i,h_{I'}(i)}(\alpha)=\psi_{i,h_{I}(i)}(\alpha) \in \CR \alpha$.
Hence, we obtain that $\psi'_{i,h_{I'}(i)} \in D(\A_{I'})$ for any $i \in S$.

From Theorem~\ref{theorem:Saito's criterion} it suffices to show that $\det(\psi'_{i,h_{I'}(i)}(\overline{\alpha_j}))_{i,j \in S}$ is equal to $\prod_{\overline{\alpha} \in I'} \overline{\alpha}$ up to a non-zero scalar multiplication. 
Considering the lower ideal $I$ in \eqref{eq:6-1}, from Theorem~\ref{theorem:Saito's criterion} we have 
$$
\det(\psi_{i,h_{I}(i)}(\alpha_j))_{i,j \in [n]} \ \dot{=} \ \prod_{\alpha \in I} \alpha
$$ 
because $\{\psi_{i,j} \in \Der \CR \mid i \in [n] \ {\rm and} \ i \leq j \leq \e_i \}$ forms uniform bases for the ideal arrangements in $(\mathfrak{t},\Phi)$. 
This implies that 
\begin{equation} \label{eq:6-2}
\det(\overline{\psi_{i,h_{I}(i)}(\alpha_j)})_{i,j \in [n]} \ \dot{=} \ \big(\prod_{\alpha \in I} \overline{\alpha} \big)=\big(\prod_{\overline{\alpha} \in I'} \overline{\alpha} \big) \ \ \ \ \ \ {\rm in} \ \CR',
\end{equation}
where the second equality follows from the definition \eqref{eq:6-1}. 
Noting that $\psi_{i,h_I(i)}(\alpha_j)=\psi_{i,i}(\alpha_j) \ \dot{=} \ \delta_{ij}$ for $i \notin S$, we have 
$$
\det(\psi_{i,h_{I}(i)}(\alpha_j))_{i,j \in [n]} \ \dot{=} \ \det(\psi_{i,h_{I}(i)}(\alpha_j))_{i,j \in S} \ \dot{=} \ \det(\psi_{i,h_{I'}(i)}(\alpha_j))_{i,j \in S}.
$$
From this together with \eqref{eq:6-2} we obtain
$$
\det(\overline{\psi_{i,h_{I'}(i)}(\alpha_j)})_{i,j \in S} \ \dot{=} \ \big(\prod_{\overline{\alpha} \in I'} \overline{\alpha} \big).
$$ 
But the left hand side coincides with $\det(\psi'_{i,h_{I'}(i)}(\overline{\alpha_j}))_{i,j \in S}$ by Lemma~\ref{lemma:6-1}, so this completes the proof.
\end{proof}

\section{Uniform bases for type $E$} 
\label{appendix:typeE}

Let $\mathfrak{t}_{E_8}$ be the Euclidean space $V=\R^8$ and we have 
\[
\CR_{E_8} = \Sym(\mathfrak{t}_{E_8}^*)=\R[x_1,\ldots,x_8].
\]
We set the exponents $\e^{E_8}_1,\ldots, \e^{E_8}_8$ as 
\[
\e^{E_8}_1=19, \ \e^{E_8}_2=29, \ \e^{E_8}_3=23, \ \e^{E_8}_4=13, \ \e^{E_8}_5=11, \ \e^{E_8}_6=7, \ \e^{E_8}_7=1, \ \e^{E_8}_8=17.
\]

We take the simple roots $\alpha_1=\frac{1}{2}(x_1-x_2-x_3-x_4-x_5-x_6-x_7-x_8), \alpha_2=x_2-x_3, \alpha_3=x_3-x_4, \alpha_4=x_4-x_5, \alpha_5=x_5-x_6, \alpha_6=x_6-x_7, \alpha_7=x_7-x_8, \alpha_8=x_7+x_8$ so that a labeling of the Dynkin diagram is as follows:

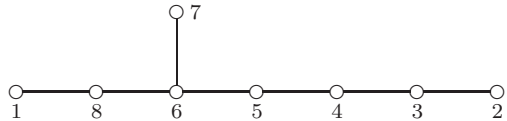
\begin{figure}[h]
\begin{center}
\begin{picture}(180,30)
\put(0,10){\circle{5}} 
\put(30,10){\circle{5}}
\put(60,10){\circle{5}}
\put(90,10){\circle{5}}
\put(120,10){\circle{5}}
\put(150,10){\circle{5}}
\put(180,10){\circle{5}}
\put(60,40){\circle{5}}

\put(2.3,10){\line(1,0){25}}
\put(32.3,10){\line(1,0){25}}
\put(62.3,10){\line(1,0){25}}
\put(92.3,10){\line(1,0){25}}
\put(122.3,10){\line(1,0){25}}
\put(152.3,10){\line(1,0){25}}
\put(60,12.3){\line(0,1){25}}

\put(-2,0){{\tiny 1}} 
\put(28,0){{\tiny 8}}
\put(58,0){{\tiny 6}}
\put(88,0){{\tiny 5}}
\put(118,0){{\tiny 4}} 
\put(148,0){{\tiny 3}}
\put(178,0){{\tiny 2}}
\put(65,37.5){{\tiny 7}}
\end{picture}
\end{center}
\vspace{-10pt}
\caption{Labeling of Dynkin diagram for type $E_8$.}
\label{picture:DynkindiagramTypeE8}
\end{figure} 
Then, the positive roots are the following forms
\begin{align}
& x_i \pm x_j \ \ \ (1 \leq i,j \leq 8), \label{eq:Epositive1} \\
& \frac{1}{2}(x_1 \pm x_2 \pm x_3 \pm x_4 \pm x_5 \pm x_6 \pm x_7 \pm x_8). \label{eq:Epositive2}
\end{align}

We arrange all positive roots of type $E_8$ as shown in Figure~\ref{picture:PositiveRootTypeE8}.
For simplicity, we denote the positive root $x_i \pm x_j$ of the form \eqref{eq:Epositive1} by $i \pm j$. Also, we denote by $\frac{1}{2}(i_1 i_2 \dots i_k)$ the positive root of the form \eqref{eq:Epositive2} such that coefficients of $x_i$ are positive for $i=i_1, i_2, \dots ,i_k$.
For example, the notation $\frac{1}{2}(12478)$ means the positive root $\frac{1}{2}(x_1+x_2-x_3+x_4-x_5-x_6+x_7+x_8)$.
Using the simple notations above, we arrange all positive roots of $E_8$ as shown in Figure~\ref{picture:PositiveRootTypeE8}.

\begin{figure}[h]
\begin{center}
\begin{picture}(500,310)
\put(0,300){{\tiny $2-3$}} 

\put(25,300){{\tiny $2-4$}}
\put(25,290){{\tiny $3-4$}}

\put(50,300){{\tiny $2-5$}}
\put(50,290){{\tiny $3-5$}}
\put(50,280){{\tiny $4-5$}}

\put(75,300){{\tiny $2-6$}}
\put(75,290){{\tiny $3-6$}}
\put(75,280){{\tiny $4-6$}}
\put(75,270){{\tiny $5-6$}}

\put(100,300){{\tiny $2-7$}}
\put(100,290){{\tiny $3-7$}}
\put(100,280){{\tiny $4-7$}}
\put(100,270){{\tiny $5-7$}}
\put(100,260){{\tiny $6-7$}}

\put(125,300){{\tiny $2-8$}}
\put(125,290){{\tiny $3-8$}}
\put(125,280){{\tiny $4-8$}}
\put(125,270){{\tiny $5-8$}}
\put(125,260){{\tiny $6-8$}}
\put(125,250){{\tiny $7-8$}}

\put(117,310){{\tiny $(a)$}}
\put(97,257){\line(0,1){50}}
\put(147,257){\line(0,1){50}}
\put(97,257){\line(1,0){50}}
\put(97,307){\line(1,0){50}}

\put(154,245){\line(0,1){5}}
\put(154,255){\line(0,1){5}}
\put(154,265){\line(0,1){5}}
\put(154,275){\line(0,1){5}}
\put(154,285){\line(0,1){5}}
\put(154,295){\line(0,1){5}}
\put(154,305){\line(0,1){5}}

\put(165,300){{\tiny $2+8$}}
\put(165,290){{\tiny $3+8$}}
\put(165,280){{\tiny $4+8$}}
\put(165,270){{\tiny $5+8$}}
\put(165,260){{\tiny $6+8$}}
\put(165,250){{\tiny $7+8$}}

\put(190,300){{\tiny $2+7$}}
\put(190,290){{\tiny $3+7$}}
\put(190,280){{\tiny $4+7$}}
\put(190,270){{\tiny $5+7$}}
\put(190,260){{\tiny $6+7$}}

\put(215,300){{\tiny $2+6$}}
\put(215,290){{\tiny $3+6$}}
\put(215,280){{\tiny $4+6$}}
\put(215,270){{\tiny $5+6$}}

\put(240,300){{\tiny $2+5$}}
\put(240,290){{\tiny $3+5$}}
\put(240,280){{\tiny $4+5$}}

\put(265,300){{\tiny $2+4$}}
\put(265,290){{\tiny $3+4$}}

\put(290,300){{\tiny $2+3$}}

\put(182,310){{\tiny $(\bar{a})$}}
\put(162,257){\line(0,1){50}}
\put(212,257){\line(0,1){50}}
\put(162,257){\line(1,0){50}}
\put(162,307){\line(1,0){50}}

\put(252,312){{\tiny $(b)$}}
\put(160,247){\line(0,1){62}}
\put(160,309){\line(1,0){152}}
\put(160,247){\line(1,0){27}}
\put(187,247){\line(0,1){8}}
\put(187,255){\line(1,0){27}}
\put(214,255){\line(0,1){11}}
\put(214,266){\line(1,0){23}}
\put(237,266){\line(0,1){11}}
\put(237,277){\line(1,0){25}}
\put(262,277){\line(0,1){11}}
\put(262,288){\line(1,0){25}}
\put(287,288){\line(0,1){10}}
\put(287,298){\line(1,0){25}}
\put(312,298){\line(0,1){11}}

\put(319,245){\line(0,1){5}}
\put(319,255){\line(0,1){5}}
\put(319,265){\line(0,1){5}}
\put(319,275){\line(0,1){5}}
\put(319,285){\line(0,1){5}}
\put(319,295){\line(0,1){5}}
\put(319,305){\line(0,1){5}}


\put(0,210){{\tiny $\frac{1}{2}(128)$}}
\put(0,200){{\tiny $\frac{1}{2}(138)$}}
\put(0,190){{\tiny $\frac{1}{2}(148)$}}
\put(0,180){{\tiny $\frac{1}{2}(158)$}}
\put(0,170){{\tiny $\frac{1}{2}(168)$}}
\put(0,160){{\tiny $\frac{1}{2}(178)$}}
\put(0,150){{\tiny $\frac{1}{2}(1)$}}

\put(30,210){{\tiny $\frac{1}{2}(127)$}}
\put(30,200){{\tiny $\frac{1}{2}(137)$}}
\put(30,190){{\tiny $\frac{1}{2}(147)$}}
\put(30,180){{\tiny $\frac{1}{2}(157)$}}
\put(30,170){{\tiny $\frac{1}{2}(167)$}}

\put(60,210){{\tiny $\frac{1}{2}(126)$}}
\put(60,200){{\tiny $\frac{1}{2}(136)$}}
\put(60,190){{\tiny $\frac{1}{2}(146)$}}
\put(60,180){{\tiny $\frac{1}{2}(156)$}}

\put(90,210){{\tiny $\frac{1}{2}(125)$}}
\put(90,200){{\tiny $\frac{1}{2}(135)$}}
\put(90,190){{\tiny $\frac{1}{2}(145)$}}

\put(120,210){{\tiny $\frac{1}{2}(124)$}}
\put(120,200){{\tiny $\frac{1}{2}(134)$}}

\put(150,210){{\tiny $\frac{1}{2}(123)$}}

\put(22,224){{\tiny $(\bar{b})$}}
\put(-5,157){\line(0,1){64}}
\put(-5,221){\line(1,0){184}}
\put(-5,157){\line(1,0){32}}
\put(27,157){\line(0,1){8}}
\put(27,165){\line(1,0){32}}
\put(59,165){\line(0,1){10}}
\put(59,175){\line(1,0){28}}
\put(87,175){\line(0,1){10}}
\put(87,185){\line(1,0){30}}
\put(117,185){\line(0,1){10}}
\put(117,195){\line(1,0){30}}
\put(147,195){\line(0,1){10}}
\put(147,205){\line(1,0){32}}
\put(179,205){\line(0,1){16}}

\put(102,222){{\tiny $(c)$}}
\put(60,177){\line(0,1){42}}
\put(60,219){\line(1,0){117}}
\put(60,177){\line(1,0){25}}
\put(85,177){\line(0,1){10}}
\put(85,187){\line(1,0){30}}
\put(115,187){\line(0,1){10}}
\put(115,197){\line(1,0){30}}
\put(145,197){\line(0,1){10}}
\put(145,207){\line(1,0){32}}
\put(177,207){\line(0,1){12}}

\put(185,145){\line(0,1){5}}
\put(185,155){\line(0,1){5}}
\put(185,165){\line(0,1){5}}
\put(185,175){\line(0,1){5}}
\put(185,185){\line(0,1){5}}
\put(185,195){\line(0,1){5}}
\put(185,205){\line(0,1){5}}
\put(185,215){\line(0,1){5}}

\put(195,210){{\tiny $\frac{1}{2}(12678)$}}
\put(195,200){{\tiny $\frac{1}{2}(13678)$}}
\put(195,190){{\tiny $\frac{1}{2}(14678)$}}
\put(195,180){{\tiny $\frac{1}{2}(15678)$}}

\put(235,210){{\tiny $\frac{1}{2}(12578)$}}
\put(235,200){{\tiny $\frac{1}{2}(13578)$}}
\put(235,190){{\tiny $\frac{1}{2}(14578)$}}

\put(275,210){{\tiny $\frac{1}{2}(12478)$}}
\put(275,200){{\tiny $\frac{1}{2}(13478)$}}

\put(315,210){{\tiny $\frac{1}{2}(12378)$}}

\put(205,224){{\tiny $(\bar{c})$}}
\put(190,177){\line(0,1){44}}
\put(190,221){\line(1,0){162}}
\put(190,177){\line(1,0){42}}
\put(232,177){\line(0,1){8}}
\put(232,185){\line(1,0){39}}
\put(271,185){\line(0,1){10}}
\put(271,195){\line(1,0){40}}
\put(311,195){\line(0,1){10}}
\put(311,205){\line(1,0){41}}
\put(352,205){\line(0,1){16}}

\put(290,222){{\tiny $(d)$}}
\put(235,187){\line(0,1){32}}
\put(235,219){\line(1,0){115}}
\put(235,187){\line(1,0){34}}
\put(269,187){\line(0,1){10}}
\put(269,197){\line(1,0){40}}
\put(309,197){\line(0,1){10}}
\put(309,207){\line(1,0){41}}
\put(350,207){\line(0,1){12}}

\put(360,145){\line(0,1){5}}
\put(360,155){\line(0,1){5}}
\put(360,165){\line(0,1){5}}
\put(360,175){\line(0,1){5}}
\put(360,185){\line(0,1){5}}
\put(360,195){\line(0,1){5}}
\put(360,205){\line(0,1){5}}
\put(360,215){\line(0,1){5}}



\put(-30,120){{\tiny $\frac{1}{2}(12568)$}}
\put(-30,110){{\tiny $\frac{1}{2}(13568)$}}
\put(-30,100){{\tiny $\frac{1}{2}(14568)$}}

\put(10,120){{\tiny $\frac{1}{2}(12468)$}}
\put(10,110){{\tiny $\frac{1}{2}(13468)$}}

\put(50,120){{\tiny $\frac{1}{2}(12368)$}}

\put(-20,134){{\tiny $(\bar{d})$}}
\put(-30,97){\line(0,1){34}}
\put(-30,131){\line(1,0){116}}
\put(-30,97){\line(1,0){37}}
\put(7,97){\line(0,1){8}}
\put(7,105){\line(1,0){39}}
\put(46,105){\line(0,1){10}}
\put(46,115){\line(1,0){40}}
\put(86,115){\line(0,1){16}}

\put(45,134){{\tiny $(e)$}}
\put(10,107){\line(0,1){22}}
\put(10,129){\line(1,0){74}}
\put(10,107){\line(1,0){34}}
\put(44,107){\line(0,1){10}}
\put(44,117){\line(1,0){40}}
\put(84,117){\line(0,1){12}}

\put(93,95){\line(0,1){5}}
\put(93,105){\line(0,1){5}}
\put(93,115){\line(0,1){5}}
\put(93,125){\line(0,1){5}}

\put(100,120){{\tiny $\frac{1}{2}(12458)$}}
\put(100,110){{\tiny $\frac{1}{2}(13458)$}}

\put(140,120){{\tiny $\frac{1}{2}(12358)$}}

\put(180,120){{\tiny $\frac{1}{2}(12348)$}}

\put(135,134){{\tiny $(\bar{e})$}}
\put(100,107){\line(0,1){22}}
\put(100,129){\line(1,0){74}}
\put(100,107){\line(1,0){34}}
\put(134,107){\line(0,1){10}}
\put(134,117){\line(1,0){40}}
\put(174,117){\line(0,1){12}}

\put(88,136){{\tiny $(f)$}}
\put(-32,95){\line(0,1){38}}
\put(-32,133){\line(1,0){248}}
\put(-32,95){\line(1,0){41}}
\put(9,95){\line(0,1){8}}
\put(9,103){\line(1,0){39}}
\put(48,103){\line(0,1){10}}
\put(48,113){\line(1,0){50}}
\put(98,113){\line(0,-1){8}}
\put(98,105){\line(1,0){39}}
\put(137,105){\line(0,1){10}}
\put(137,115){\line(1,0){79}}
\put(216,115){\line(0,1){18}}

\put(223,95){\line(0,1){5}}
\put(223,105){\line(0,1){5}}
\put(223,115){\line(0,1){5}}
\put(223,125){\line(0,1){5}}

\put(230,120){{\tiny $\frac{1}{2}(12567)$}}
\put(230,110){{\tiny $\frac{1}{2}(13567)$}}
\put(230,100){{\tiny $\frac{1}{2}(14567)$}}

\put(270,120){{\tiny $\frac{1}{2}(12467)$}}
\put(270,110){{\tiny $\frac{1}{2}(13467)$}}

\put(310,120){{\tiny $\frac{1}{2}(12367)$}}

\put(450,134){{\tiny $(h)$}}
\put(358,105){\line(0,1){26}}
\put(358,131){\line(1,0){116}}
\put(358,105){\line(1,0){38}}
\put(396,105){\line(0,1){10}}
\put(396,115){\line(1,0){78}}
\put(474,115){\line(0,1){16}}

\put(305,134){{\tiny $(g)$}}
\put(270,107){\line(0,1){22}}
\put(270,129){\line(1,0){74}}
\put(270,107){\line(1,0){34}}
\put(304,107){\line(0,1){10}}
\put(304,117){\line(1,0){40}}
\put(344,117){\line(0,1){12}}

\put(353,95){\line(0,1){5}}
\put(353,105){\line(0,1){5}}
\put(353,115){\line(0,1){5}}
\put(353,125){\line(0,1){5}}

\put(360,120){{\tiny $\frac{1}{2}(12457)$}}
\put(360,110){{\tiny $\frac{1}{2}(13457)$}}

\put(400,120){{\tiny $\frac{1}{2}(12357)$}}

\put(440,120){{\tiny $\frac{1}{2}(12347)$}}

\put(395,134){{\tiny $(\bar{g})$}}
\put(360,107){\line(0,1){22}}
\put(360,129){\line(1,0){74}}
\put(360,107){\line(1,0){34}}
\put(394,107){\line(0,1){10}}
\put(394,117){\line(1,0){40}}
\put(434,117){\line(0,1){12}}

\put(348,136){{\tiny $(\bar{f})$}}
\put(228,95){\line(0,1){38}}
\put(228,133){\line(1,0){248}}
\put(228,95){\line(1,0){41}}
\put(269,95){\line(0,1){8}}
\put(269,103){\line(1,0){39}}
\put(308,103){\line(0,1){10}}
\put(308,113){\line(1,0){48}}
\put(356,113){\line(0,-1){10}}
\put(356,103){\line(1,0){42}}
\put(398,103){\line(0,1){10}}
\put(398,113){\line(1,0){78}}
\put(476,113){\line(0,1){20}}

\put(483,95){\line(0,1){5}}
\put(483,105){\line(0,1){5}}
\put(483,115){\line(0,1){5}}
\put(483,125){\line(0,1){5}}



\put(0,60){{\tiny $\frac{1}{2}(12456)$}}
\put(0,50){{\tiny $\frac{1}{2}(13456)$}}

\put(40,60){{\tiny $\frac{1}{2}(12356)$}}

\put(80,60){{\tiny $\frac{1}{2}(12346)$}}

\put(120,60){{\tiny $\frac{1}{2}(12345)$}}

\put(130,74){{\tiny $(i)$}}
\put(-2,45){\line(0,1){26}}
\put(-2,71){\line(1,0){158}}
\put(-2,45){\line(1,0){38}}
\put(36,45){\line(0,1){10}}
\put(36,55){\line(1,0){120}}
\put(156,55){\line(0,1){17}}

\put(15,72){{\tiny $(\bar{h})$}}
\put(0,47){\line(0,1){22}}
\put(0,69){\line(1,0){114}}
\put(0,47){\line(1,0){34}}
\put(34,47){\line(0,1){10}}
\put(34,57){\line(1,0){80}}
\put(114,57){\line(0,1){12}}

\put(163,45){\line(0,1){5}}
\put(163,55){\line(0,1){5}}
\put(163,65){\line(0,1){5}}

\put(170,60){{\tiny $\frac{1}{2}(1245678)$}}
\put(170,50){{\tiny $\frac{1}{2}(1345678)$}}

\put(220,60){{\tiny $\frac{1}{2}(1235678)$}}

\put(270,60){{\tiny $\frac{1}{2}(1234678)$}}

\put(320,60){{\tiny $\frac{1}{2}(1234578)$}}

\put(370,60){{\tiny $\frac{1}{2}(1234568)$}}

\put(420,60){{\tiny $\frac{1}{2}(1234567)$}}

\put(255,72){{\tiny $(\bar{i})$}}
\put(170,47){\line(0,1){22}}
\put(170,69){\line(1,0){193}}
\put(170,47){\line(1,0){42}}
\put(212,47){\line(0,1){10}}
\put(212,57){\line(1,0){151}}
\put(363,57){\line(0,1){12}}

\put(430,74){{\tiny $(j)$}}
\put(168,45){\line(0,1){26}}
\put(168,71){\line(1,0){295}}
\put(168,45){\line(1,0){46}}
\put(214,45){\line(0,1){10}}
\put(214,55){\line(1,0){249}}
\put(463,55){\line(0,1){16}}

\put(468,45){\line(0,1){5}}
\put(468,55){\line(0,1){5}}
\put(468,65){\line(0,1){5}}


\put(0,10){{\tiny $1-3$}}
\put(0,0){{\tiny $1-2$}}

\put(25,10){{\tiny $1-4$}}

\put(50,10){{\tiny $1-5$}}

\put(75,10){{\tiny $1-6$}}

\put(100,10){{\tiny $1-7$}}

\put(125,10){{\tiny $1-8$}}

\put(15,22){{\tiny $(\bar{j})$}}
\put(0,-3){\line(0,1){22}}
\put(0,19){\line(1,0){148}}
\put(0,-3){\line(1,0){22}}
\put(22,-3){\line(0,1){10}}
\put(22,7){\line(1,0){126}}
\put(148,7){\line(0,1){12}}

\put(120,20){{\tiny $(k)$}}
\put(100,9){\line(0,1){8}}
\put(100,17){\line(1,0){46}}
\put(100,9){\line(1,0){46}}
\put(146,9){\line(0,1){8}}

\put(155,-5){\line(0,1){5}}
\put(155,5){\line(0,1){5}}
\put(155,15){\line(0,1){5}}

\put(160,10){{\tiny $1+8$}}

\put(185,10){{\tiny $1+7$}}

\put(210,10){{\tiny $1+6$}}

\put(235,10){{\tiny $1+5$}}

\put(260,10){{\tiny $1+4$}}

\put(285,10){{\tiny $1+3$}}

\put(310,10){{\tiny $1+2$}}

\put(180,20){{\tiny $(\bar{k})$}}
\put(160,9){\line(0,1){8}}
\put(160,17){\line(1,0){46}}
\put(160,9){\line(1,0){46}}
\put(206,9){\line(0,1){8}}
\end{picture}
\end{center}
\vspace{-10pt}
\caption{The arrangement of all positive roots for type $E_8$.}
\label{picture:PositiveRootTypeE8}
\end{figure} 

In Figure~\ref{picture:PositiveRootTypeE8} the partial order $\preceq$ on $\Phi^+_{E_8}$ is defined as follows. 
In each area separated by dotted lines we have the following relation 
\begin{enumerate}
\item if a root $\alpha$ is left-adjacent to a root $\beta$, then $\alpha \lessdot \beta$, 
\item if a root $\alpha$ is lower-adjacent to a root $\beta$, then $\alpha \lessdot \beta$.
\end{enumerate}
For the blocks $(\bullet)$ and $(\bar{\bullet})$ with same shape where the symbol $\bullet$ means $a,b,\ldots, k$, we have the following relation 
\begin{enumerate}
\item[(3)] if a root $\alpha$ is in the block $(\bullet)$ and $\beta$ is in the same place as $\alpha$ in the block $(\bar{\bullet})$, then $\alpha \lessdot \beta$. 
\end{enumerate}
For two positive roots $\alpha,\beta$, we define $\alpha \preceq \beta$ if there exist positive roots $\gamma_0,\ldots,\gamma_N$ such that $\alpha=\gamma_0 \lessdot \gamma_1 \lessdot \cdots \lessdot \gamma_N=\beta$.

\bigskip

Now we fix a decomposition $\Phi^+_{E_8}=\coprod_{i=1}^8 \Phi^+_i$ satisfying \eqref{eq:decomposition1} and \eqref{eq:decomposition2}.
Using the simple notations above again, we define the positive roots $\alpha_{i,j}$ by Table~\ref{table:positiverootsTypeE}.
{\tiny
\begin{table}[h]
\begin{tabular}{|c||c|c|c|c|c|c|c|c|c|c|c|c|c|c|c|} \hline
positive roots $\backslash$ $j$ & $2$ & $3$ & $4$ & $5$ & $6$ & $7$ & $8$ & $9$ & $10$ & $11$ & $12$ \cr
\hline
\hline $\alpha_{1,j}$ & $\frac{1}{2}(1)$ & $\frac{1}{2}(178)$ & $\frac{1}{2}(168)$ & $\frac{1}{2}(158)$ & $\frac{1}{2}(148)$ & $\frac{1}{2}(138)$ & $\frac{1}{2}(137)$ & $\frac{1}{2}(136)$ & $\frac{1}{2}(135)$ & $\frac{1}{2}(125)$ & $\frac{1}{2}(124)$    \cr
\hline $\alpha_{2,j}$ &       & $2-3$ &  $2-4$  & $2-5$ & $2-6$ & $2-7$ & $2-8$ & $2+7$ & $2+6$ & $2+5$ & $2+4$      \cr
\hline $\alpha_{3,j}$ &       &       & $3-4$ &   $3-5$   & $3-6$ &  $3-7$ &  $3-8$ & $3+7$ & $3+6$ & $3+5$ & $3+4$      \cr
\hline $\alpha_{4,j}$ &       &       &       &   $4-5$    & $4-6$ &  $4-7$ &  $4-8$ & $4+7$ & $4+6$ & $4+5$ & $\frac{1}{2}(145)$         \cr
\hline $\alpha_{5,j}$ &       &       &       &               & $5-6$ & $5-7$ &  $5-8$ & $5+7$ & $5+6$ & $\frac{1}{2}(156)$ & $\frac{1}{2}(15678)$          \cr
\hline $\alpha_{6,j}$ &       &       &       &               &          & $6-7$ &  $6-8$ & $6+7$ & $\frac{1}{2}(167)$ & $\frac{1}{2}(157)$ & $\frac{1}{2}(147)$         \cr
\hline $\alpha_{7,j}$ &       &       &       &               &          & &  $7-8$ & & & &        \cr
\hline $\alpha_{8,j}$ &       &       &       &               &          & &  & $7+8$ & $6+8$ & $5+8$ & $4+8$         \cr
\hline 
\end{tabular} \\ \vspace{1.25ex}
\begin{tabular}{|c||c|c|c|c|c|c|c|c|c|c|c|c|} \hline
 & $13$ & $14$ & $15$ & $16$ & $17$ & $18$ & $19$ & $20$ & $21$ \cr
\hline
\hline $\alpha_{1,j}$ &  $\frac{1}{2}(12478)$     &  $\frac{1}{2}(12468)$     &  $\frac{1}{2}(12458)$    &  $\frac{1}{2}(12457)$     &  $\frac{1}{2}(12456)$   &    $\frac{1}{2}(1245678)$   &  $\frac{1}{2}(1235678)$ & $\frac{1}{2}(1234678)$ &  \cr
\hline $\alpha_{2,j}$ &  $2+3$     &  $\frac{1}{2}(123)$     &  $\frac{1}{2}(12378)$    &  $\frac{1}{2}(12368)$     &  $\frac{1}{2}(12358)$   &   $\frac{1}{2}(12348)$  & $\frac{1}{2}(12347)$ & $\frac{1}{2}(12346)$  &  $\frac{1}{2}(12345)$    \cr
\hline $\alpha_{3,j}$ &  $\frac{1}{2}(134)$     &  $\frac{1}{2}(13478)$     &  $\frac{1}{2}(13468)$    &    $\frac{1}{2}(13458)$   &  $\frac{1}{2}(13457)$   &   $\frac{1}{2}(13456)$  & $\frac{1}{2}(1345678)$ & $1-2$  &  $1-3$    \cr
\hline $\alpha_{4,j}$ &  $\frac{1}{2}(14578)$     &  $\frac{1}{2}(14568)$     &   $\frac{1}{2}(14567)$   &  $\frac{1}{2}(13567)$     & $\frac{1}{2}(13467)$    &    &  &    &    \cr
\hline $\alpha_{5,j}$ &  $\frac{1}{2}(14678)$     &  $\frac{1}{2}(13678)$     &  $\frac{1}{2}(13578)$    & $\frac{1}{2}(13568)$      &     &     &  &   &    \cr
\hline $\alpha_{6,j}$ &  $\frac{1}{2}(146)$     &       &      &       &     &   &  &    &     \cr
\hline $\alpha_{8,j}$ &  $3+8$     &  $2+8$     &  $\frac{1}{2}(128)$    &  $\frac{1}{2}(127)$     &  $\frac{1}{2}(126)$   &   $\frac{1}{2}(12678)$   & $\frac{1}{2}(12578)$ & $\frac{1}{2}(12568)$  &  $\frac{1}{2}(12567)$   \cr
\hline
\end{tabular} \\ \vspace{1.25ex}
\begin{tabular}{|c||c|c|c|c|c|c|c|c|c|c|c|c|c|c|c|} \hline
 & $22$ & $23$ & $24$ & $25$ & $26$ & $27$ & $28$ &  $29$ & $30$ & $31$ \cr
\hline
\hline $\alpha_{1,j}$       &       &      &       &     &      &  &   &  & & \cr
\hline $\alpha_{2,j}$      &  $\frac{1}{2}(1234578)$     &  $\frac{1}{2}(1234568)$    &  $\frac{1}{2}(1234567)$     &  $1-8$   &  $1+7$    & $1+6$ & $1+5$  & $1+4$ & $1+3$ & $1+2$ \cr
\hline $\alpha_{3,j}$      &  $1-4$     &  $1-5$    &  $1-6$     & $1-7$    & $1+8$     &  &   & & &  \cr
\hline $\alpha_{8,j}$       &   $\frac{1}{2}(12467)$    &  $\frac{1}{2}(12367)$    & $\frac{1}{2}(12357)$      &  $\frac{1}{2}(12356)$   &      &  &   & & &  \cr
\hline
\end{tabular} \vspace{1.25ex}
\caption{The positive roots $\alpha_{i,j}$ in type $E_8$.}
\label{table:positiverootsTypeE}
\end{table}
}

Note that $\height(\Phi^+_{E_8})=29$ and the set $\Lambda_m$ in \eqref{eq:Lambdam} is given by Table~\ref{table:LambdaTypeE}.
{\tiny
\begin{table}[h]
\begin{tabular}{|c||c|c|c|c|c|c|c|c|} \hline
$m$ & $0,1$ & $2 \leq m \leq 7$ & $8 \leq m \leq 11$ & $12,13$ & $14 \leq m \leq 17$ & $18,19$ & $20 \leq m \leq 23$ & $24 \leq m \leq 29$ \cr
\hline
$\Lambda_m$ & $\{1,2,3,4,5,6,7,8\}$ & $\{1,2,3,4,5,6,8\}$ & $\{1,2,3,4,5,8\}$ & $\{1,2,3,4,8\}$ & $\{1,2,3,8\}$ & $\{1,2,3\}$ & $\{2,3\}$ & $\{2\}$      \cr
\hline 
\end{tabular} 
\vspace{5pt}
\caption{$\Lambda_m$ in type $E_8$.}
\label{table:LambdaTypeE}
\end{table}
}

For the decomposition $\Phi^+_{E_8}=\coprod_{i=1}^8 \Phi^+_i$ above, a \textbf{Hessenberg function for type $E_8$} is a function $h_I:\{1,\ldots,8\} \rightarrow \{1,\ldots,31\}$ defined in \eqref{eq:Hessft} associated to a lower ideal $I \subset \Phi^+_{E_8}$.

\begin{remark}
As in the case of type $F_4$ in Remark~\ref{remark:HessenbergFunctionF4}, Hessenberg functions for type $E_8$ can also be characterized by the specific conditions. We omit the details.
\end{remark}

By similar discussion on the case of type $F_4$, we obtain uniform bases $\{\psi^{E_8}_{i,i+m} \in \Der \CR \mid 0 \leq m \leq 29, i \in \Lambda_m \}$ as follows.
For each $m$ with $0 \leq m \leq 29=\height(\Phi^+_{E_8})$, we define the matrix $P_m^{E_8}=(p_{ij}^{(m)})_{i,j \in \Lambda_m}$ by Figure~\ref{picture:InvertibleMatricesTypeE8}.
{\tiny
\begin{figure}
\begin{align*}
&P_0^{E_8}= \begin{pmatrix} \begin {array}{cccccccc} 1&0&0&0&0&0&0&0\\ \noalign{\medskip}0
&1&0&0&0&0&0&0\\ \noalign{\medskip}0&0&1&0&0&0&0&0
\\ \noalign{\medskip}0&0&0&1&0&0&0&0\\ \noalign{\medskip}0&0&0&0&1&0&0
&0\\ \noalign{\medskip}0&0&0&0&0&1&0&0\\ \noalign{\medskip}0&0&0&0&0&0
&1&0\\ \noalign{\medskip}0&0&0&0&0&0&0&1\end {array} \end{pmatrix}, \   
P_1^{E_8}= \begin{pmatrix} \begin {array}{cccccccc} 1&0&0&0&0&0&0&0\\ \noalign{\medskip}0
&1&0&0&0&0&0&0\\ \noalign{\medskip}0&1&1&0&0&0&0&0
\\ \noalign{\medskip}0&1&1&1&0&0&0&0\\ \noalign{\medskip}0&1&1&1&1&0&0
&0\\ \noalign{\medskip}1&1&1&1&1&1&0&1\\ \noalign{\medskip}1&1&1&1&1&1
&1&1\\ \noalign{\medskip}1&0&0&0&0&0&0&1\end {array} \end{pmatrix}, \\
&P_2^{E_8}=
 \begin{pmatrix} \begin {array}{ccccccc} 1&0&0&0&0&0&0\\ \noalign{\medskip}0
&1&0&0&0&0&0\\ \noalign{\medskip}0&1&1&0&0&0&0
\\ \noalign{\medskip}0&1&1&1&0&0&0\\ \noalign{\medskip}-1&1&1&1&1&-1
&-1\\ \noalign{\medskip}-1&1&1&1&1&1&-1\\ \noalign{\medskip}1&1&1&
1&1&1&1\end {array} \end{pmatrix}, \ 
P_3^{E_8}=
 \begin{pmatrix} \begin {array}{ccccccc} 1&-\frac{1}{2}&-\frac{1}{2}&-\frac{1}{2}&-\frac{1}{4}&-\frac{1}{2}&-\frac{1}{4}
\\ \noalign{\medskip}0&1&0&0&0&0&0\\ \noalign{\medskip}0&1&1&0&0&0
&0\\ \noalign{\medskip}1&1&1&1&-1&-\frac{1}{2}&\frac{1}{2}\\ \noalign{\medskip}2&2&2
&2&1&-1&1\\ \noalign{\medskip}1&1&1&1&\frac{1}{2}&1&\frac{1}{2}
\\ \noalign{\medskip}2&-4&-4&-4&-2&-1&1\end {array} \end{pmatrix}, \\
&P_4^{E_8}= \begin{pmatrix} \begin {array}{ccccccc} 1&0&0&0&0&-\frac{1}{2}&0
\\ \noalign{\medskip}0&1&0&0&0&0&0\\ \noalign{\medskip}-\frac{2}{3}&1&1&-\frac{1}{3}
&0&\frac{1}{3}&-\frac{1}{6}\\ \noalign{\medskip}-2&3&3&1&0&1&-\frac{1}{2}
\\ \noalign{\medskip}-2&3&3&1&1&3&-\frac{1}{2}\\ \noalign{\medskip}0&0&0&0&0
&1&0\\ \noalign{\medskip}4&6&6&2&0&-2&1\end {array} \end{pmatrix}, \ 
P_5^{E_8}= \begin{pmatrix} \begin {array}{ccccccc} 1&-\frac{3}{8}&-\frac{3}{16}&-\frac{1}{8}&-\frac{1}{8}&-\frac{1}{4}&-\frac{1}{32}
\\ \noalign{\medskip}\frac{2}{3}&1&-\frac{3}{4}&-\frac{1}{12}&-\frac{1}{12}&-\frac{1}{6}&\frac{1}{12}
\\ \noalign{\medskip}\frac{4}{3}&2&1&-\frac{1}{6}&-\frac{1}{6}&-\frac{1}{3}&\frac{1}{6}\\ \noalign{\medskip}
2&3&\frac{3}{2}&1&-\frac{3}{2}&-3&\frac{1}{4}\\ \noalign{\medskip}2&3&\frac{3}{2}&1&1&-3&\frac{1}{4}
\\ \noalign{\medskip}1&\frac{3}{2}&\frac{3}{4}&\frac{1}{2}&\frac{1}{2}&1&\frac{1}{8}\\ \noalign{\medskip}8&-
18&-9&-1&-1&-2&1\end {array} \end{pmatrix}, \\
&P_6^{E_8}=
 \begin{pmatrix} \begin {array}{ccccccc} 1&\frac{3}{4}&0&0&0&0&\frac{1}{16}
\\ \noalign{\medskip}0&1&0&0&0&0&0\\ \noalign{\medskip}\frac{8}{3}&4&1&0&0&0
&\frac{1}{6}\\ \noalign{\medskip}16&18&3&1&0&-2&1\\ \noalign{\medskip}-8&-
6&0&0&1&2&-\frac{1}{2}\\ \noalign{\medskip}-4&-3&0&0&0&1&-\frac{1}{4}
\\ \noalign{\medskip}0&12&0&0&0&0&1\end {array} \end{pmatrix}, \ 
P_7^{E_8}= \begin{pmatrix} \begin {array}{ccccccc} 1&0&0&0&0&0&\frac{1}{16}
\\ \noalign{\medskip}0&1&0&0&0&0&\frac{1}{12}\\ \noalign{\medskip}\frac{8}{3}&2&1&0&0
&0&\frac{1}{3}\\ \noalign{\medskip}0&6&3&1&1&2&\frac{1}{2}\\ \noalign{\medskip}0&0
&0&0&1&0&0\\ \noalign{\medskip}-8&-3&-\frac{3}{2}&-\frac{1}{2}&\frac{1}{2}&1&-\frac{3}{4}
\\ \noalign{\medskip}0&0&0&0&0&0&1\end {array} \end{pmatrix}, \\
&P_8^{E_8}= \begin{pmatrix} \begin {array}{cccccc} 1&0&0&\frac{1}{16}&\frac{1}{8}&\frac{1}{16}
\\ \noalign{\medskip}0&1&0&0&0&\frac{1}{12}\\ \noalign{\medskip}\frac{4}{3}&2&1&\frac{1}{
12}&\frac{1}{6}&\frac{1}{4}\\ \noalign{\medskip}-16&0&0&1&-2&-1
\\ \noalign{\medskip}-8&0&0&\frac{1}{2}&1&-\frac{1}{2}\\ \noalign{\medskip}0&0&0&0
&0&1\end {array} \end{pmatrix}, \ 
P_9^{E_8}= \begin{pmatrix} \begin {array}{cccccc} 1&\frac{3}{2}&\frac{3}{4}&0&\frac{1}{8}&-\frac{1}{8}
\\ \noalign{\medskip}-\frac{1}{3}&1&-\frac{1}{4}&0&-\frac{1}{24}&\frac{1}{24}\\ \noalign{\medskip}
-\frac{2}{3}&2&1&0&-\frac{1}{12}&\frac{1}{12}\\ \noalign{\medskip}-\frac{16}{3}&-8&-4&1&\frac{4}{3}&-\frac{4}{3}\\ \noalign{\medskip}-4&-6&-3&0&1&-1\\ \noalign{\medskip}4&6&3&0
&1/2&1\end {array} \end{pmatrix}, \\
&P_{10}^{E_8}= \begin{pmatrix} \begin {array}{cccccc} 1&0&0&0&0&\frac{1}{4}
\\ \noalign{\medskip}-\frac{1}{3}&1&0&0&0&-\frac{1}{12}\\ \noalign{\medskip}\frac{2}{3}&0&
1&0&0&\frac{1}{6}\\ \noalign{\medskip}-\frac{16}{3}&0&-8&1&\frac{4}{3}&-\frac{8}{3}
\\ \noalign{\medskip}-4&0&-6&0&1&-2\\ \noalign{\medskip}0&0&0&0&0&1\end {array} \end{pmatrix}, \ 
P_{11}^{E_8}= \begin{pmatrix} \begin {array}{cccccc} 1&-3&0&0&0&0\\ \noalign{\medskip}0
&1&0&0&0&0\\ \noalign{\medskip}0&0&1&\frac{1}{8}&\frac{1}{6}&-\frac{1}{6}
\\ \noalign{\medskip}0&0&0&1&0&0\\ \noalign{\medskip}-8&24&-6&\frac{3}{4}&
1&-1\\ \noalign{\medskip}0&0&-6&\frac{3}{4}&1&1\end {array} \end{pmatrix}, \\
&P_{12}^{E_8}= \begin{pmatrix} \begin {array}{ccccc} 1&-3&0&0&0\\ \noalign{\medskip}0
&1&0&0&0\\ \noalign{\medskip}\frac{8}{3}&-8&1&\frac{1}{4}&\frac{1}{6}
\\ \noalign{\medskip}\frac{16}{3}&-16&0&1&\frac{2}{3}\\ \noalign{\medskip}8&-24&0
&0&1\end {array} \end{pmatrix}, \ 
P_{13}^{E_8}= \begin{pmatrix} \begin {array}{ccccc} 1&-3&-\frac{3}{4}&-\frac{3}{16}&0
\\ \noalign{\medskip}0&1&0&0&0\\ \noalign{\medskip}\frac{4}{3}&-4&1&-\frac{1}{4}
&0\\ \noalign{\medskip}\frac{16}{3}&-16&4&1&\frac{4}{3}
\\ \noalign{\medskip}0&0&0&0&1\end {array} \end{pmatrix}, \ 
P_{14}^{E_8}= \begin{pmatrix} \begin {array}{cccc} 1&-6&0&-\frac{1}{4}
\\ \noalign{\medskip}\frac{1}{6}&1&0&\frac{1}{24}\\ \noalign{\medskip}\frac{2}{3}&-4&1
&-\frac{1}{6}\\ \noalign{\medskip}4&-24&0&1\end {array}
\end{pmatrix}, \\
&P_{15}^{E_8}=\begin{pmatrix} \begin {array}{cccccccc} 1&0&0&\frac{1}{4}
\\ \noalign{\medskip}0&1&0&\frac{1}{24}\\ \noalign{\medskip}\frac{2}{3}&0&1&\frac{1}{6}\\ \noalign{\medskip}0&0&0&1\end {array} \end{pmatrix}, \ 
P_{16}^{E_8}= \begin{pmatrix} \begin {array}{cccc} 1&6&0&\frac{1}{4}
\\ \noalign{\medskip}0&1&0&0\\ \noalign{\medskip}\frac{2}{3}&4&1&\frac{1}{6}\\ \noalign{\medskip}0&24&0&1\end {array} \end{pmatrix}, \ 
P_{17}^{E_8}= \begin{pmatrix} \begin {array}{cccc} 1&0&\frac{3}{2}&0
\\ \noalign{\medskip}0&1&0&0\\ \noalign{\medskip}0&0&1
&0\\ \noalign{\medskip}4&24&6&1\end {array} \end{pmatrix}, \ 
P_{18}^{E_8}=
 \begin{pmatrix} \begin {array}{ccc} 1&0&\frac{3}{2}
\\ \noalign{\medskip}-\frac{1}{6}&1&-\frac{1}{4}\\ \noalign{\medskip}0&0&1\end {array} \end{pmatrix}, \\
&P_{19}^{E_8}= \begin{pmatrix} \begin {array}{ccc} 1&-6&\frac{3}{2}
\\ \noalign{\medskip}0&1&0\\ \noalign{\medskip}0&0&1\end {array} \end{pmatrix}, \ 
P_{20}^{E_8}=P_{21}^{E_8}= \begin{pmatrix} \begin {array}{cc} 1&\frac{1}{4}
\\ \noalign{\medskip}0&1\end {array} \end{pmatrix}, \ 
P_{22}^{E_8}=P_{23}^{E_8}= \begin{pmatrix} \begin {array}{cc} 1&0
\\ \noalign{\medskip}4&1\end {array} \end{pmatrix}, \ 
P_m^{E_8}=
\begin{pmatrix}
1  \\      
\end{pmatrix} \ \ \ (24 \leq m \leq 29).
\end{align*}
\vspace{-10pt}
\caption{The invertible matrices for type $E_8$.}
\label{picture:InvertibleMatricesTypeE8}
\end{figure} 
}
Here, we think of indexes for rows and columns of the matrix $P^{E_8}_m$ as $\Lambda_m$ given in Table~\ref{table:LambdaTypeE}, and we arrange them as in increasing order. 
We set
\begin{align*}
[\psi^{E_8}_{i,i}]_{i \in \Lambda_0}&=P_0^{E_8}[\alpha_{i}^*]_{i \in \Lambda_0},\\
[\psi^{E_8}_{i,i+m}]_{i \in \Lambda_m}&=P_m^{E_8}[\alpha_{i,i+m}\psi^{E_8}_{i,i+m-1}]_{i \in \Lambda_m} \ \ \ {\rm for} \ 1 \leq m \leq 29.
\end{align*}
Note that $\psi^{E_8}_{i,i}=\alpha_i^*$ for $1 \leq i \leq 8$, namely 
\begin{align*} 
&\psi^{E_8}_{1,1}=2\partial_1, \ \ 
\psi^{E_8}_{2,2}=\partial_1+\partial_2, \ \ 
\psi^{E_8}_{3,3}=2\partial_1+\partial_2+\partial_3, \ \ 
\psi^{E_8}_{4,4}=3\partial_1+\partial_2+\partial_3+\partial_4, \\
&\psi^{E_8}_{5,5}=4\partial_1+\partial_2+\partial_3+\partial_4+\partial_5, \ \ 
\psi^{E_8}_{6,6}=5\partial_1+\partial_2+\partial_3+\partial_4+\partial_5+\partial_6, \\
&\psi^{E_8}_{7,7}=\frac{1}{2}(5\partial_1+\partial_2+\partial_3+\partial_4+\partial_5+\partial_6+\partial_7-\partial_8), \\ 
&\psi^{E_8}_{8,8}=\frac{1}{2}(7\partial_1+\partial_2+\partial_3+\partial_4+\partial_5+\partial_6+\partial_7+\partial_8). 
\end{align*}

We can check by using Maple that $\{\psi^{E_8}_{i,h_I(i)} \mid 1 \leq i \leq 8 \} \subset D(\A_I)$ for any lower ideal $I \subset \Phi_{E_8}^+$.
From this together with $\det P^{E_8}_m \neq 0$ $(1 \leq m \leq 29)$, we obtain the following theorem
by Proposition~\ref{proposition:key} (see also Remark~\ref{remark:key}).

\begin{theorem} \label{theorem:psiE8}
The derivations $\{\psi^{E_8}_{i,i+m} \mid 0 \leq m \leq 29, i \in \Lambda_m \}$ form uniform bases for the ideal arrangements of type $E_8$.
\end{theorem}

We can also compute invertible matrices for types $E_6$ and $E_7$ by using Maple.
Here, the positive roots $\alpha_{i,j}$ for each types $E_6$ and $E_7$ denote the $\alpha_{i,j}$ of type $E_8$ with suitable range for $(i,j)$.
We only list the invertible matrices for types $E_6$ and $E_7$ in Fugures~\ref{picture:InvertibleMatricesTypeE7} and \ref{picture:InvertibleMatricesTypeE6}.
{\tiny
\begin{figure}
\begin{align*}
&P^{E_7}_0=  \left( \begin {array}{ccccccc} 1&0&0&0&0&0&0\\ \noalign{\medskip}0
&1&0&0&0&0&0\\ \noalign{\medskip}0&0&1&0&0&0&0
\\ \noalign{\medskip}0&0&0&1&0&0&0\\ \noalign{\medskip}0&0&0&0&1&0
&0\\ \noalign{\medskip}0&0&0&0&0&1&0\\ \noalign{\medskip}0&0&0&0&0
&0&1\end {array} \right), \ \  
P^{E_7}_1=  \left( \begin {array}{ccccccc} 1&0&0&0&0&0&0\\ \noalign{\medskip}0
&1&0&0&0&0&0\\ \noalign{\medskip}0&1&1&0&0&0&0
\\ \noalign{\medskip}0&1&1&1&0&0&0\\ \noalign{\medskip}1&1&1&1&1&0
&1\\ \noalign{\medskip}1&1&1&1&1&1&1\\ \noalign{\medskip}1&0&0&0&0
&0&1\end {array} \right), \ \  
P^{E_7}_2= \left( \begin {array}{cccccc} 1&0&0&0&0&0\\ \noalign{\medskip}0
&1&0&0&0&0\\ \noalign{\medskip}0&1&1&0&0&0
\\ \noalign{\medskip}-1&1&1&1&-1&-1\\ \noalign{\medskip}-1&1&1&1
&1&-1\\ \noalign{\medskip}1&1&1&1&1&1\end {array} \right), \\  
&P^{E_7}_3=  \left( \begin {array}{cccccc} 1&-\frac{1}{2}&-\frac{1}{2}&-\frac{1}{4}&-\frac{1}{2}&-\frac{1}{4}\\ \noalign{\medskip}0&1&0&0&0&0\\ \noalign{\medskip}1&1&1&-1&-\frac{1}
{2}&\frac{1}{2}\\ \noalign{\medskip}2&2&2&1&-1&1\\ \noalign{\medskip}1&
1&1&\frac{1}{2}&1&\frac{1}{2}\\ \noalign{\medskip}2&-4&-4&-2&-1&1\end {array}
 \right), \ \ 
P^{E_7}_4=  \left( \begin {array}{cccccc} 1&0&0&0&-\frac{1}{2}&0
\\ \noalign{\medskip}-\frac{2}{3}&1&-\frac{1}{3}&0&\frac{1}{3}&-\frac{1}{6}\\ \noalign{\medskip}-2
&3&1&0&1&-\frac{1}{2}\\ \noalign{\medskip}-2&3&1&1&3&-\frac{1}{2}
\\ \noalign{\medskip}0&0&0&0&1&0\\ \noalign{\medskip}4&6&2&0&-2
&1\end {array} \right), \\ 
&P^{E_7}_5= \left( \begin {array}{cccccc} 1&-\frac{3}{8}&-\frac{1}{8}&-\frac{1}{8}&-\frac{1}{4}&0
\\ \noalign{\medskip}0&1&0&0&0&0\\ \noalign{\medskip}0&3&1&-1&-2
&0\\ \noalign{\medskip}0&3&1&1&-2&0\\ \noalign{\medskip}0&\frac{3}{2}&
\frac{1}{2}&\frac{1}{2}&1&0\\ \noalign{\medskip}8&-9&-1&-1&-2&1\end {array}
 \right), \ \ 
P^{E_7}_6=  \left( \begin {array}{ccccc} 1&0&0&0&0\\ \noalign{\medskip}\frac{4}
{3}&1&0&0&0\\ \noalign{\medskip}16&6&1&0&-2
\\ \noalign{\medskip}-8&0&0&1&2\\ \noalign{\medskip}-4&0&0&0&1
\end {array} \right), \ \ 
P^{E_7}_7= 
  \left( \begin {array}{ccccc} 1&0&0&0&0\\ \noalign{\medskip}\frac{4}
{3}&1&0&0&0\\ \noalign{\medskip}8&6&1&0&0
\\ \noalign{\medskip}0&0&0&1&0\\ \noalign{\medskip}-8&-3&-\frac{1}{2}&
\frac{1}{2}&1\end {array} \right), \\ 
&P^{E_7}_8= \left( \begin {array}{cccc} 1&0&\frac{1}{8}&\frac{1}{8}
\\ \noalign{\medskip}\frac{2}{3}&1&\frac{1}{12}&\frac{1}{12}\\ \noalign{\medskip}-8&0
&1&-1\\ \noalign{\medskip}-8&0&1&1\end {array} \right), \ \  
P^{E_7}_9= \left( \begin {array}{cccc} 1&\frac{3}{2}&0&\frac{1}{8}
\\ \noalign{\medskip}0&1&0&0\\ \noalign{\medskip}-8&-12&1&0\\ \noalign{\medskip}0&0&0&1\end {array} \right), \ \ 
P^{E_7}_{10}= \left( \begin {array}{ccc} 1&0&0\\ \noalign{\medskip}-12&1&1\\ \noalign{\medskip}-12&0&1\end {array}
 \right), \ \  
P^{E_7}_{11}= \left( \begin {array}{ccc} 1&0&0\\ \noalign{\medskip}0&1&0\\ \noalign{\medskip}-12&1&1\end {array}
 \right), \\ 
&P^{E_7}_{12}= \left( \begin {array}{cc} 1&\frac{1}{12}
\\ \noalign{\medskip}0&1\end {array} \right), \ \  
P^{E_7}_{13}= \left( \begin {array}{cc} 1&0\\ \noalign{\medskip}12&1\end {array} \right), \ \ 
P_m^{E_7}=
\begin{pmatrix}
1  \\      
\end{pmatrix} \ \ \ (14 \leq m \leq 17).
\end{align*}
\vspace{-10pt}
\caption{The invertible matrices for type $E_7$.}
\label{picture:InvertibleMatricesTypeE7}
\end{figure} }
{\tiny
\begin{figure}
\begin{align*}
&P^{E_6}_0=\left( \begin {array}{cccccc} 1&0&0&0&0&0\\ \noalign{\medskip}0
&1&0&0&0&0\\ \noalign{\medskip}0&0&1&0&0&0
\\ \noalign{\medskip}0&0&0&1&0&0\\ \noalign{\medskip}0&0&0&0&1
&0\\ \noalign{\medskip}0&0&0&0&0&1\end {array} \right), \ \ 
P^{E_6}_1=\left( \begin {array}{cccccc} 1&0&0&0&0&0\\ \noalign{\medskip}0
&1&0&0&0&0\\ \noalign{\medskip}0&1&1&0&0&0
\\ \noalign{\medskip}1&1&1&1&0&1\\ \noalign{\medskip}1&1&1&1&1
&1\\ \noalign{\medskip}1&0&0&0&0&1\end {array} \right), \ \  
P^{E_6}_2=\left( \begin {array}{ccccc} 1&0&0&0&0\\ \noalign{\medskip}0
&1&0&0&0\\ \noalign{\medskip}-1&1&1&-1&-1
\\ \noalign{\medskip}-1&1&1&1&-1\\ \noalign{\medskip}1&1&1&1
&1\end
 {array} \right), \\ 
&P^{E_6}_3=
 \left( \begin {array}{ccccc} 1&-\frac{1}{2}&-\frac{1}{4}&-\frac{1}{2}&-\frac{1}{4}
\\ \noalign{\medskip}1&1&-1&-\frac{1}{2}&\frac{1}{2}\\ \noalign{\medskip}2&2
&1&-1&1\\ \noalign{\medskip}1&1&\frac{1}{2}&1&\frac{1}{2}\\ \noalign{\medskip}
2&-4&-2&-1&1\end {array} \right), \ \  
P^{E_6}_4= \left( \begin {array}{ccccc} 1&0&0&-\frac{1}{2}&0
\\ \noalign{\medskip}0&1&0&0&0\\ \noalign{\medskip}0&2&1&2
&0\\ \noalign{\medskip}0&0&0&1&0\\ \noalign{\medskip}4&2&0&-
2&1\end {array} \right), \ \  
P^{E_6}_5= \left( \begin {array}{cccc} 1&-\frac{1}{4}&-\frac{1}{8}&-\frac{1}{4}
\\ \noalign{\medskip}-2&1&0&0\\ \noalign{\medskip}-4&2&1&0
\\ \noalign{\medskip}0&1&\frac{1}{2}&1\end {array} \right), \\  
&P^{E_6}_6= \left( \begin {array}{ccc} 1&0&-\frac{1}{2}
\\ \noalign{\medskip}0&1&1\\ \noalign{\medskip}0&0&1\end {array} \right), \ \ 
P^{E_6}_7= \left( \begin {array}{ccc} 1&0&0\\ \noalign{\medskip}0&1&0\\ \noalign{\medskip}-2&1&1\end {array}
 \right), \ \  
P^{E_6}_8=\left( \begin {array}{cc} 1&0\\ \noalign{\medskip}2&1\end {array} \right), \ \
P_m^{E_6}=
\begin{pmatrix}
1  \\      
\end{pmatrix} \ \ \ (9 \leq m \leq 11).
\end{align*}
\vspace{-10pt}
\caption{The invertible matrices for type $E_6$.}
\label{picture:InvertibleMatricesTypeE6}
\end{figure} }
Their computations imply uniform bases for the ideal arrangements of types $E_6$ and $E_7$, respectively.
In what follows, we give a different way to construct uniform bases for types $E_6$ and $E_7$ from uniform bases in type $E_8$ by using Proposition~\ref{proposition:6-2}. 

Let $\mathfrak{t}_{E_7}$ be the hyperplane in $\mathfrak{t}_{E_8}=\R^8$ defined by the linear function $\alpha_2^*=x_1+x_2$:
$$
\mathfrak{t}_{E_7}=\{(x_1,\ldots,x_8) \in \R^8 \mid x_1+x_2=0 \}.
$$
Then we have 
\[
\CR_{E_7} = \Sym(\mathfrak{t}_{E_7}^*)=\R[x_1,\ldots,x_8]/(x_1+x_2).
\]
Similarly, let $\mathfrak{t}_{E_6}$ be the hyperplane in $\mathfrak{t}_{E_7}$ defined by the linear function $\alpha_3^*=\frac{1}{2}x_1-\frac{1}{2}x_2+x_3$: 
$$
\mathfrak{t}_{E_6}:=\{(x_1,\ldots,x_8) \in \mathfrak{t}_{E_7} \mid \frac{1}{2}x_1-\frac{1}{2}x_2+x_3=0 \}.
$$
Here, we remark that $\alpha_1^*,\alpha_3^*,\ldots,\alpha_8^*$ is the dual basis of the simple roots $\alpha_1,\alpha_3\ldots,\alpha_8$ for type $E_7$.
Then we have 
\[
\CR_{E_6} = \Sym(\mathfrak{t}_{E_6}^*)=\R[x_1,\ldots,x_8]/(x_1+x_2,\frac{1}{2}x_1-\frac{1}{2}x_2+x_3).
\]

A Hessenberg function $h^{E_8}$ for type $E_8$ such that $h^{E_8}(2)=2$ and $h^{E_8}(i) \leq i+\e_i^{E_7}$ for $i=1,3,\ldots,8$ is called a \textbf{Hessenberg function for type $E_7$}, denoted by $h^{E_7}$. 
Here, $\e_1^{E_7},\e_3^{E_7},\ldots, \e_8^{E_7}$ denote the exponents of type $E_7$ as follows:
\[
\e_1^{E_7}=9, \ \e_3^{E_7}=17, \ \e_4^{E_7}=13, \ \e_5^{E_7}=11, \ \e_6^{E_7}=7, \ \e_7^{E_7}=1, \ \e_8^{E_7}=5.
\]
Similarly, a Hessenberg function $h^{E_7}$ for type $E_7$ such that $h^{E_7}(3)=3$ and $h^{E_7}(i) \leq i+\e_i^{E_6}$ for $i=1,4,\ldots,8$ is called a \textbf{Hessenberg function for type $E_6$}, denoted by $h^{E_6}$ where  $\e_1^{E_6},\e_4^{E_6},\ldots, \e_8^{E_6}$ are the exponents of type $E_6$: 
\[
\e_1^{E_6}=5, \  \e_4^{E_6}=11, \ \e_5^{E_6}=8, \ \e_6^{E_6}=7, \ \e_7^{E_6}=1, \ \e_8^{E_6}=4.
\]
Let $\psi^{E}_{i,j}=\psi^{E_8}_{i,j}$ for $1 \leq i \leq 8$ and $i \leq j \leq \e^{E_8}_i$. 
For each $m$ with $0 \leq m \leq 29=\height(\Phi^+_{E_8})$, we denote by $P_m^{E}=(p_{ij})_{i,j \in \Lambda_m}$ the matrix $P_m^{E_8}$ for simplicity. 
We define $\psi'^{E}_{i,j}$ for $1 \leq i \leq 8$ and $i \leq j \leq \e^{E_8}_i$ as follows:
\begin{align*}
\psi'^{E}_{i,i}&=\begin{cases} \psi^{E_7}_{i,i} \ \ \ &{\rm if} \ i=1,3,4,5,6,7,8, \\
0 \ \ \ \ \ \ &{\rm if} \ i=2, 
\end{cases} \\
\psi'^{E}_{i,i+m}&=\sum_{j \in \Lambda_m} p_{ij} \overline{\alpha_{j,j+m}}\psi'_{j,j+m-1} \ \ \ {\rm for} \ 1 \leq m \leq \height(\Phi^+) \ {\rm and} \ i \in \Lambda_m,
\end{align*}
where $\psi^{E_7}_{i,i} \ (i=1,3,\ldots,8)$ is given by
\begin{align*}
\psi^{E_7}_{1,1}&=\partial_1-\partial_2, \ \ 
\psi^{E_7}_{3,3}=\frac{1}{2}\partial_1-\frac{1}{2}\partial_2+\partial_3, \ \ 
\psi^{E_7}_{4,4}=\partial_1-\partial_2+\partial_3+\partial_4, \\ 
\psi^{E_7}_{5,5}&=\frac{3}{2}\partial_1-\frac{3}{2}\partial_2+\partial_3+\partial_4+\partial_5, \ \ 
\psi^{E_7}_{6,6}=2\partial_1-2\partial_2+\partial_3+\partial_4+\partial_5+\partial_6, \\ 
\psi^{E_7}_{7,7}&=\frac{1}{2}(2\partial_1-2\partial_2+\partial_3+\partial_4+\partial_5+\partial_6+\partial_7-\partial_8), \ \ 
\psi^{E_7}_{8,8}=\frac{1}{2}(3\partial_1-3\partial_2+\partial_3+\partial_4+\partial_5+\partial_6+\partial_7+\partial_8).  
\end{align*}
We also define $\psi''^{E}_{i,j}$ for $1 \leq i \leq 8$ and $i \leq j \leq \e^{E_8}_i$ as follows:
\begin{align*}
\psi''^{E}_{i,i}&=\begin{cases} \psi^{E_6}_{i,i} \ \ \ &{\rm if} \ i=1,4,5,6,7,8, \\
0 \ \ \ \ \ \ &{\rm if} \ i=2, 3,
\end{cases} \\
\psi''^{E}_{i,i+m}&=\sum_{j \in \Lambda_m} p_{ij} \overline{\alpha_{j,j+m}}\psi''_{j,j+m-1} \ \ \ {\rm for} \ 1 \leq m \leq \height(\Phi^+) \ {\rm and} \ i \in \Lambda_m,
\end{align*}
where $\psi^{E_6}_{i,i} \ (i=1,4,\ldots,8)$ is described as
\begin{align*}
\psi^{E_6}_{1,1}&=\frac{2}{3}\partial_1-\frac{2}{3}\partial_2-\frac{2}{3}\partial_3, \ \ 
\psi^{E_6}_{4,4}=\frac{1}{3}\partial_1-\frac{1}{3}\partial_2-\frac{1}{3}\partial_3+\partial_4, \ \ 
\psi^{E_6}_{5,5}=\frac{2}{3}\partial_1-\frac{2}{3}\partial_2-\frac{2}{3}\partial_3+\partial_4+\partial_5, \\
\psi^{E_6}_{6,6}&=\partial_1-\partial_2-\partial_3+\partial_4+\partial_5+\partial_6, \ \ 
\psi^{E_6}_{7,7}=\frac{1}{2}(\partial_1-\partial_2-\partial_3+\partial_4+\partial_5+\partial_6+\partial_7-\partial_8), \\
\psi^{E_6}_{8,8}&=\frac{1}{2}(\frac{5}{3}\partial_1-\frac{5}{3}\partial_2-\frac{5}{3}\partial_3+\partial_4+\partial_5+\partial_6+\partial_7+\partial_8).  
\end{align*}

\begin{theorem} \label{theorem:subroottypeE}
A set of derivations $\{\psi'^{E}_{i,j} \in \Der \CR_{E_7} \mid i=1,3,4,5,6,7,8 \ {\rm and} \ i \leq j \leq \e^{E_7}_i \}$ forms uniform bases for the ideal arrangements in type $E_7$.
A set of derivations $\{\psi''^{E}_{i,j} \in \Der \CR_{E_6} \mid i=1,4,5,6,7,8 \ {\rm and} \ i \leq j \leq \e^{E_6}_i \}$ forms uniform bases for the ideal arrangements in type $E_6$.
\end{theorem}

Also, we can describe systematically explicit presentations of the cohomology rings of the regular nilpotent Hessenberg varieties for type $E$.
Define $g^{E}_{i,j}=\Quadraticmap(\psi^{E}_{i,j}), g'^{E}_{i,j}=\Quadraticmap(\psi'^{E}_{i,j}), g''^{E}_{i,j}=\Quadraticmap(\psi''^{E}_{i,j})$.
Note that $g^{E}_{i,j}=f^{E_8}_{i,j}$.
One can see inductively that 
\begin{align*}
g'^{E}_{i,j} &\equiv g^{E}_{i,j} \ \ \ {\rm mod} \ g^{E}_{2,2}, \\
g''^{E}_{i,j} &\equiv g^{E}_{i,j} \ \ \ {\rm mod} \ (g^{E}_{2,2},g^{E}_{3,3}). 
\end{align*}
This together with Theorem~\ref{theorem:subroottypeE} implies the following corollary.

\begin{corollary} \label{corollary:fijE}
Let $h$ be a Hessenberg function for type $E_n (n=6,7,8)$ and $\Hess(N,h)$ the associated regular nilpotent Hessenberg variety for type $E$.
Then, the following ring isomorphism holds
\begin{equation*}
H^*(\Hess(N,h)) \cong \R[x_1,\ldots,x_8]/(g^{E}_{1,h(1)},\ldots,g^{E}_{8,h(8)}).
\end{equation*}
Note that Hessenberg functions $h^{E_7}$ and $h^{E_6}$ for types $E_7$ and $E_6$ satisfy $h^{E_7}(2)=2$ and $h^{E_6}(2)=2, h^{E_6}(3)=3$.
\end{corollary}



\end{document}